\documentclass[12pt,english,a4paper]{smfart}
\usepackage[utf8]{inputenc}
\usepackage{graphicx}
\usepackage[T1]{fontenc}
\usepackage[english]{babel}
\usepackage{lmodern}
\usepackage{mathrsfs}  
\usepackage{smfthm}
\usepackage[headings]{fullpage}
\usepackage{amssymb}
\usepackage[all]{xy}
\usepackage[colorlinks=true, allcolors=magenta]{hyperref}
\usepackage{url}
\usepackage{tikz-cd}
\usepackage{comment}
\usepackage{flexisym}
\usepackage{breqn}
\usepackage{dirtytalk}
\usepackage{csquotes}
\usepackage[normalem]{ulem}

\theoremstyle{plain}
\newtheorem{thm}{Theorem}[section]
\newtheorem{cor}[thm]{Corollary}
\newtheorem{pro}[thm]{Proposition}
\newtheorem{lem}[thm]{Lemma}
\theoremstyle{definition}
\newtheorem{rmk}[thm]{Remark}
\newtheorem{exa}[thm]{Example}
\newtheorem{dfn}[thm]{Definition}

\let\sf\mathsf
\let\rm\mathrm
\let\cal\mathcal
\let\goth\mathfrak
\let\bb\mathbb
\let\hat\widehat
\let\tilde\widetilde
\let\phi\varphi

\let\epsilon\varepsilon

\def\Z{{\bf Z}}

\def\N{{\bf N}}
\def\R{{\bf R}}
\def\A{{\bf A}}

\def\Ker{{\mathrm{Ker}}}

\def\Coker{{\mathrm{Coker}}}

\renewcommand{\O}{{{\cal O}}}

\usepackage{graphicx} 

\newcommand{\tate}[2]{#1\langle#2\rangle}
\newcommand{\power}[2]{#1[\![#2]\!]}
\newcommand{\temp}[2]{#1[\![#2]\!]_{\mathrm{temp}}}
\newcommand{\dercat}{\mathcal{D}}

\newcommand{\dertens}[1]{\hat{\otimes}^{\mathbb{L}}_{#1}}

\newcommand{\derintHom}[3]{\mathbb{R}\underline{\mathrm{Hom}}_{#1}(#2,#3)}

\newcommand{\diff}[1]{\frac{\mathrm{d}}{\mathrm{d}#1}}
\newcommand{\fast}[2]{#1[\![#2]\!]_{\mathrm{fast}}}

\newcommand{\intv}{\mathcal{V}}
\newcommand{\temptube}[2]{]#1[_{#2,\mathrm{temp}}}
\newcommand{\conv}[2]{#1\{\!\{#2\}\!\}}
\newcommand{\norm}[1]{\left\|#1\right\|}
\newcommand{\wotimes}{\hat{\otimes}}
\newcommand{\tube}[2]{]#1[_{#2}}

\newcommand{\abs}[1]{\left|#1\right|}

\def\IHom{{\underline{\rm{Hom}}}}
\def\botimes{{\overline{\otimes}}}
\def\BanR{{\sf{Ban}_R}}
\def\Ind{{\sf{Ind}}}

\def\BornK{{\sf{Born}_K}}
\def\CBornK{{\sf{CBorn}_K}}
\def\BanK{{\sf{Ban}_K}}

\DeclareMathOperator*{\colim}{{\mathrm{colim}}}

\DeclareMathOperator*{\contcoprod}{{\coprod}^{\le 1}}

\DeclareMathOperator*{\quotecolim}{``\colim"}

\title{ The tempered disk and  the  tempered cohomology}
\author{ F. Bambozzi, B. Chiarellotto, P. Vanni}
\date{October 2024}

\address{Dipartimento di Matematica "Tullio Levi-Civita", Università degli Studi di Padova, Via Trieste, 63 - 35121 Padova}%
\email{federico.bambozzi@math.unipd.it}
\email{chiarbru@math.unipd.it}
\email{pietro.vanni@math.unipd.it}

\begin{document}

\maketitle


\begin{abstract}
Let $\mathcal V$ be a complete discretely valued ring of mixed characteristic, with fraction field $K$ and residue field $k$. Using the ind-Banach framework for derived analytic geometry, we view generic fibers of $\mathcal V$-schemes as derived analytic spaces. This approach yields a refined notion of spectrum, richer than that of classical rigid geometry: for example we may have  open subsets whose structure sheaf contains functions of logarithmic growth. In this setting, the transfer theorem for the log-growth of solutions of $p$-adic differential equations becomes a natural continuity statement, analogous to the classical transfer theorem for radii of convergence on Berkovich spaces. Motivated by ideas of Scholze, we introduce tempered tubular neighborhoods of smooth $k$-schemes and define a new tempered de Rham cohomology via the Hodge-completed derived de Rham complex of these tubes. Finally, we establish a comparison theorem showing that, for smooth proper $k$-schemes, tempered de Rham cohomology agrees with crystalline cohomology.
\end{abstract}
\tableofcontents

\section{Introduction}

In the classical theory of $p$-adic differential equations, the presence of a Frobenius structure imposes strong constraints on the radii of convergence of solutions. This phenomenon is often referred to as Dwork's trick \cite[Lemma 6.3]{dJ}, \cite{DworkGSul}. Such Frobenius structures typically arise from geometry, for instance from families of varieties, and the corresponding good convergence properties can be expressed either by saying that the differential equation is solvable at the generic point or, in terms of coefficients, by saying that it is overconvergent. Moreover, there is a finer relation between the Frobenius structure and the growth of solutions: as observed by Dwork \cite{Chiarellog}, the slopes of Frobenius are related to the logarithmic growth of
solutions.

Berkovich and Huber spaces provide a natural framework in which information can be transferred from generic points to classical points in rigid analytic geometry. In these theories both types of points belong to a single topological space associated with the analytic space. This makes it possible to transfer information, such as radii of convergence, by using continuity properties of the radius function \cite{Balda}, \cite{Pulita}. In this picture, radius conditions are represented geometrically by analytic domains. By contrast, logarithmic growth conditions are not represented by open analytic domains in the usual rigid or Berkovich geometry in the same direct way. Nevertheless, classical transfer theorems for logarithmic growth are available \cite{Christ}, \cite{Chiarellog}, \cite{ChTs2}.

In this article we work in derived analytic geometry setting provided ind-Banach algebras, and use it to associate to such algebras a spectrum. In this spectrum there are open subsets whose sections of the structure sheaf are functions with logarithmic growth. This allows us to reinterpret transfer theorems for logarithmic growth as continuity statements. In this sense, the framework developed here should be viewed as a natural analytic setting for studying the continuity of logarithmic growth conditions, in analogy with the study of radii of convergence in \cite{Pulita}, and for studying $F$-isocrystals with logarithmic decay in the sense of \cite{KM1}, \cite{KM2}.

 From this perspective, and inspired by ideas introduced by Scholze \cite{Scholzetalk}, the usual radius of convergence condition for power series is replaced by a logarithmic growth condition. In dimension one this leads to the tempered open disk, which is an open subset of our spectrum containing the open unit disk and contained in the closed unit disk. This suggests a tempered analogue of convergent rigid cohomology (where the open disk is used).

Recall that convergent rigid cohomology of a  variety in characteristic $p$ is computed by considering tubes of radius one inside the generic fibers of smooth formal $\mathcal V$-schemes into which the variety is embedded. The independence of the resulting cohomology from the auxiliary choices is a fundamental property of the theory. In the convergent setting, coefficients are convergent isocrystals, and the convergence condition may be viewed, in dimension one, as a condition that the generic radius of convergence is equal to one. It is therefore natural to replace the classical tube by a tempered tube and to define a corresponding tempered cohomology theory for schemes over a field $k$ of characteristic $p$.

For this new theory, one starts with a closed regular  embedding
\[
X_k \longrightarrow \hat P
\]
of a  $k$-scheme into  the special fiber of a smooth formal $\mathcal V$-scheme, $\hat P$, where $\mathcal V$ is a complete non-archimedean DVR with residue field $k$ and fraction field $K$. The classical tube of $X_k$ in the generic fiber of $\hat P$ is replaced by the \emph{tempered tube} $\temptube{X}{\hat P}$. We construct this object as a derived analytic space in the ind-Banach setting and show that it is an open analytic subspace of the generic fiber of $\hat P$, viewed as an ind-Banach derived analytic space. Locally, this open subspace is described by algebras of functions satisfying a logarithmic
growth condition at the boundary.

The tempered cohomology groups $\mathrm{H}^i_{\mathrm{temp}}(X_k)$ are then defined as the (Hodge-completed) derived de Rham cohomology of $\temptube{X}{\hat P}$. A priori, this construction may depend on the choice of the embedding into $\hat P$. We prove, however, that in our hypotheses of regularity  the resulting cohomology is independent of all choices. For smooth proper $X_k$, we show that this cohomology agrees with rigid, equivalently crystalline, cohomology. Thus the theory developed here should be regarded as a tempered analogue of convergent rigid cohomology. Extending it to non-proper varieties will require a tempered counterpart of Monsky-Washnitzer cohomology, involving logarithmic decay conditions at infinity, following ideas proposed by Scholze. We plan to pursue this direction in future work \cite{logdecay&other}. 

Another possible development concerns the introduction of new open subsets in our ind-Banach derived analytic spaces, defined by logarithmic decay conditions. This would provide a natural geometric framework for the coefficients associated with $p$-adic representations of the fundamental group. Indeed, generic representations give rise to convergent isocrystals \cite{KATZ}, while representations with finite local monodromy give rise via this Riemann-Hilbert correspondance to overconvergent isocrystals \cite{FINMON}.
Between these two classes lie other types of representations, some of which were studied in \cite{KM1} and \cite{KM2} in connection with logarithmic decay series, but without an underlying geometric framework. We hope that our approach will provide such a framework.

We now describe the article in  detail.

In Section \ref{section notation} we fix some notation. Let $R$ be a Banach ring. In Section \ref{section ind-banach} we review the theory of Ind-Banach modules over $R$, developed in \cite{BamBen}. In particular, we see that they form a quasi-abelian category: $\Ind(\BanR)$.
We also explain that the completed tensor product over $R$, $\wotimes_R$, endows $\Ind(\BanR)$ with a closed symmetric monoidal structure. We are then able to associate its derived infinity category:  that is  a closed symmetric monoidal stable $\infty$-category.  In the case $R=K$, where $K$ is a complete non-archimedean field, we then give an alternative and useful description of the category of essentially monomorphic objects in $\Ind(\BanR)$ in terms of bornological vector spaces (see Theorem \ref{compare ind and born}). We also recall some basic facts on bornological vector spaces. 

In Section \ref{section tempered functions} we introduce the algebra of tempered power series over $R$, denoted by $\temp{R}{t}$, as an element of $\Ind(\BanR)$ (Definition \ref{definition tempered series}). We prove moreover some technical lemmas concerning this algebra.

Let $\cal{E}$ be a complete and cocomplete closed symmetric monoidal $\infty$-category. In Section \ref{section spectrum} we explain how to attach a homotopically ringed space to $\cal{E}$ to be interpreted as the spectrum of $\cal{E}$: 
$\goth{S}(\cal{E})$, see Definition \ref{DEF}. 
This construction may be seen as a kind of  Hochster dual to the construction of \cite[Lecture V]{ClauSchCompl}, in the sense that the basis of compact open subsets of our space  ( that is the set of homotopy epimorphisms from $1$ in $\cal{E}$) corresponds to the lattice of closed subsets of their locale. In particular, when $\cal{E}=\dercat_{\infty}(\sf{Mod}_{\cal{R}})$, 
for $\cal{R}$ a commutative monoid in $\Ind(\BanR)$  (i.e. the symmetric  monoidal $\infty$-category associated to  $\cal{R}$-modules in   $\Ind(\BanR)$),  we define in this way the derived analytic spectrum  of $\cal{R}$. 
This analytic spectrum is a topological realization of the homotopy Zariski topology on $\dercat_{\infty}({\sf{Mod}_{\cal{R}}})$ of \cite{BamBen}. 
We then consider the space $\mathbf{A}^1_{R}= \goth{S}(\dercat_{\infty}(\sf{Mod}_{{R[t]}}))= \goth{S}({R[t]})$ that we call the \emph{analytic affine line}.
We prove that $\temp{R}{t}$ is 
the algebra of functions on an open subset of $\mathbf{A}^1_{R}$, the \emph{tempered disk centered in} $0$ (Proposition \ref{tempidem}). We additionally identify other open subsets of $\mathbf{A}^1_{R}$, in case $R=K$.

In Section \ref{section complementary}, we calculate the algebra of functions supported on the complement of the \emph{tempered disk centered at $\infty$} in $\mathbf{A}^1_{K}$ (see Proposition \ref{prop:complement_of_tempered}). These are the fast converging
series $\fast{K}{t}$ (Definition \ref{definition fast series}) or  limit (as $r \to +\infty$) of   $r$-log-decay  series  at $\infty$ (following \cite{KM1}, \cite{KM2}, see Remark \ref{r-decay}). To  the algebra given by such a series,  it is actually associated another open subset of $\mathbf{A}^1_{K}$: \emph{ the disk of fast converging series} around the closed unit disk. 
We see that the intersection of this fast closed disk centered at $0$ with the tempered disk centered at $\infty$ is non-empty (Proposition \ref{nonemptintersection}). We also discuss how, in this derived setting, the algebra of the functions naturally supported on the complement of the open unit disk at $\infty$ is given by the dagger algebra of overconvergent power series (see Remark \ref{dag} for a more precise statement).

In Section \ref{section transfer} we show that the Transfer Theorem for log-growth of solutions of a differential system over $K$ (Theorem \ref{thm tempered transfert}) can be interpreted as a continuity result in the framework of derived analytic spectra. This ultimately relies on the fact that $\temp{K}{t}$ defines an open subset of $\mathbf{A}^1_K$. This shows that some complications of the classical theories simply disappear in our setting.

In Section \ref{section tubes} we introduce the basic setup for tempered convergent cohomology. 
Let, as before,  $\cal{V}$ and $k$ respectively be the valuation ring of $K$ and the residue field of $K$, a $p$-adic discretely valued non-archimedean complete field.
Given a closed embedding
$$X\to { \hat P},$$
where ${ \hat P}$ is a $p$-adic affine  formal ($p$-adic) scheme,  $X$ is   scheme over $k$ defined by a regular ideal in the special fiber ${\hat  P}_k$ of $\hat  P$.  Suppose ${ \hat P}=\rm{Spf}(A)$ for some $p$-adic smooth formal algebra over $\cal{V}$. We define the tempered tube $\temptube{X}{{\hat  P}}$ in Definition \ref{defn:tempered_tube}, as an open in $\goth{S}(A_K)$. We prove that it contains the rigid analytic tube of \cite[Chapter 2]{LeStum} (see Corollary \ref{prop tempered tube is contained in rigid}) and that it is independent of the presentation (Proposition \ref{prop tempered is independent of the presentation}). The main result of the section is a tempered version of the Berthelot's \say{weak fibration Theorem} (Theorem \ref{weakfibr}). 
This theorem says that if we have a commutative  diagram
\begin{equation*}
\begin{tikzcd}
& { \hat P}' \arrow{dd}{u}\\
X \arrow{ur} \arrow[swap]{dr}&\\
& { \hat P}, 
\end{tikzcd}   
\end{equation*}
where ${\hat P}'\to {\hat P}$ is a smooth morphism of affine $p$-adic formal smooth schemes and the two diagonal maps satisfy the previous hypotheses, then the tempered tube of $X$ in ${\hat P}'$ is isomorphic to the tempered tube of $X$ in ${\hat P}$ times a tempered polydisk.

In Section \ref{section cohomology} we globalize the previous local definition,  we consider embeddings
$$X\to {\hat P},$$
where  $X$ is smooth over $k$ and ${\hat P}$ is a smooth formal scheme over $\cal{V}$ (i.e. they are not required to be affine anymore). 
We construct a derived analytic space associated to ${\hat P}$, ${\hat P}_K$, and a global version of the tempered tube 
$\temptube{X}{{\hat P}}$ inside it, essentially gluing the affine constructions in the context of derived ind-Banach geometry. The space ${\hat P}_K$ comes equipped with a notion of (analytic and complete) derived  de Rham complex and we set the tempered convergent cohomology of $X$
$$\rm{H}_{\rm{temp}}^{\bullet}(X) = 
    \mathrm{H}^i( {{\mathbb L}}{\hat \Omega}^{\wedge}_{{\hat P}_K/K}(\temptube{X}{{\hat P}}))$$
to be the de Rham cohomology of the tempered tube (notice that $\temptube{X}{{\hat P}}$ is an open subscheme of ${\hat P}_K$). We prove that this definition is independent of all the choices (Theorem \ref{thm tempered cohomology is weel def}).
Finally, for $X$ smooth proper, we see that 
$$\rm{H}_{\rm{temp}}^{\bullet}(X)\cong\rm{H}_{\rm{crys}}^{\bullet}(X),$$
that is Proposition \ref{prop comparison with crystalline}.

If we  compare to the classical  rigid case (see Remark \ref{dag}), then what we have calculated is an analogue of the  {\it convergent rigid cohomology}. Moreover,  if we want to deal with open varieties, $X$,  we should consider   the  {\it complement} of the tempered tube of ${\overline X}\setminus X$ ($\overline X$ is a compactification of $X$).  In the classical case,  sections on the complement are  given by dagger algebras (hence getting  rigid cohomology),  while here they should be  given by  fast converging  series (or "limit (as $r \to +\infty$) of   $r$-log-decay  series  at $\infty ={\overline X}\setminus X$"). We plan to study this subject in the future.
\bigskip

{\bf Acknowledgments.} We thank for help and suggestions: Devarshi Mukherjee, Peter Scholze, Bertrand To\"en and Gabriele Vezzosi. The authors were supported by the following grants: UNIPD DOR2315777/23 "Algebraic Geometry and Arithmetic Geometry"; MUR PRIN 2022 "The Arithmetic of Motives and L-functions".

\section{Notation}\label{section notation}
In this section we fix some notation that we will use throughout the article: 
\begin{itemize}
     \item $\sf{Sets}$ will be the category of sets.
    \item $\sf{Ab}$ will be the category of abelian groups.
     \item In general if  $\sf{E}$ indicates a category,  then $\sf{Psh}(\sf{E})$ will be the category of presheaves of sets over $\sf{E}$,  $\Ind(\sf{E})$ will be the Ind-completion of $\sf{E}$,
   and  $\Ind^{m}(\sf{E})$ will be the subcategory of essentially monomorphic objects of $\Ind(\sf{E})$.
    \item $\rm{lim}$ will denote the categorical limit, $\rm{colim}$ will denote the categorical colimit,  \say{$\rm{colim}$} will denote the colimit in $\Ind(\sf{E})$.
    \item If $\sf{E}$ is a symmetric monoidal category, its monoidal functor is denoted by $\botimes$, and the unit by $1$. The category of commutative monoids in $\sf{E}$ will be called $\sf{Comm}(\sf{E})$. If $\sf{E}$ is moreover closed, its internal hom functor will be denoted by $\underline{\rm{Hom}}$. 
    \item If, in general,  $A$ is a commutative monoid in a symmetric monoidal category $(\sf{E}, \botimes)$ we will denote by $\sf{Mod}_{A}$ the category of modules over $A$, its monoidal structure will be denoted by $\botimes_{A}$ (if it exists).
    \item We will also use the language of $\infty$-categories and extend the above conventions of category theory to $\infty$-categories.
    \item If $E\rightarrow E'$ is a morphism in a stable $\infty$-category, $[E\rightarrow E']$ will denote its fiber.

 \item All along the present  work, $K$ will be a nonarchimedean complete field of characteristic $0$, with valuation ring $\intv$, maximal ideal $\mathfrak{m}$ and residue field $k$ (of characteristic $p$). The absolute value on $K$ will be denoted by $\abs{\cdot}_{K}$. Then 
    $\sf{Vec}_{K}$ will denote the category of vector spaces over $K$, $\BornK$ will be the category of bornological modules over $K$,  $\CBornK$ will be the category of complete bornological modules over $K$.
    
  \item If $R$ is a nonarchimedean Banach ring, with norm $\abs{\cdot}_R$, then by  $\BanR$ we mean  the category of Banach modules over $R$ and if  $M\in\BanR$, its norm will be denoted by $\norm{\cdot}_{M}$. By  $\wotimes_{R}$ we understand   the completed tensor product over $R$ (see \cite[Subsection 2.1.7]{BoschAna} for the definition). If $c\in\R_{> 0}$, $R_{c}$ will denote the object of $\BanR$ with the same underlying module of $R$ but with norm rescaled by $c$.

    \item If $m\in\N$, $R[t_1,\dots,t_m]$ is the polynomial algebra over $R$ in $m$ variables.
    \item If $m\in\N$, $\power{R}{t_1,\dots,t_m}$ is the formal power series algebra over $R$ in $m$ variables.
    \item If $m\in\N$, $\tate{R}{t_1,\dots,t_m}$ is the Tate algebra over $R$ in $m$ variables.
    \item If $m\in\N$, $\temp{R}{t_1,\dots,t_m}$ is the tempered power series algebra over $R$ in $m$ variables, see Definition \ref{definition multivariable tempered series}. %
    \end{itemize}

\section{Ind-Banach modules}\label{section ind-banach}
Let $R$ be a \emph{non-archimedean Banach ring}. Here we mean that $R$ is a ring equipped with a norm $|\cdot|_{R}$ that satisfies the non-archimedean triangle inequality, such that the ring multiplication is bounded, and that $R$ is complete for the metric  induced by $|\cdot|_{R}$ (see also \cite[Section 1.2.1]{BoschAna}).  In this section we recall the theory of the category of ind-Banach modules over $R$, developed in \cite[Section 3]{BamBen} and further in \cite{BBK19,bambozzi2018stein, BenKreFrechet}. 

\subsection{Quasi-abelian categories of Banach modules} \label{quasi-ab}
We recall the notion of a non-archimedean Banach module.
\begin{dfn}
    A non-archimedean Banach
    module over $R$ is a module over $R$ equipped with a norm $\norm{\cdot}_M:M\rightarrow\R_{\ge 0}$ such that the following conditions hold:
    \begin{itemize}
        \item $\norm{m}_M=0$, for $m \in M$, if and only if $m =0$.
        \item $M$ is complete for the metric induced by $\norm{\cdot}_M$.
        \item $\norm{m+m'}_M\le \max(\norm{m}_M,\norm{m'}_M)$ for any $m,m'\in M$.
        \item $\norm{rm}_M\le\abs{r}\norm{m}$, for any $r\in R$, $m\in M$.
    \end{itemize}
\end{dfn}
We denote by $\sf{Ban}_R$ the category of non-archimedean 
Banach module over $R$, with morphisms given by bounded $R$-linear maps. We denote by $\BanR^{\le 1}$ the wide subcategory of $\BanR$ with the same objects of $\BanR$ and morphism given by \emph{contracting morphisms}, namely those morphisms $f\colon M\to M'$ such that 
\[ \norm{f(m)}_{M'}\leq \norm{m}_M \] 
for every $m\in M$.
The following is \cite[Proposition 3.18]{BamBen}.

\begin{pro}
    The categories $\sf{Ban}_R$ and $\BanR^{\le 1}$ are quasi-abelian.
\end{pro}

In fact,  $\BanR$ is naturally enriched over abelian groups. The direct sum of two objects $M$ and $M'$ is given by the direct sum of the underlying abelian groups $M\oplus M'$ endowed with the norm 

\begin{equation*}
\norm{(a,b)}_{M\oplus M'} = \max(\norm{a}_M,\norm{b}_{M'}),
\end{equation*}
for $a\in M$, $b\in M'$.
If $f:M\rightarrow M'$ is a morphism in $\BanR$ its kernel is given by 
\begin{equation*}
\Ker(f)=f^{-1}(0)
\end{equation*} 
endowed with the norm induced by $M$, and its cokernel is given by 
\begin{equation*}
\Coker(f)=M'/\overline{f(M)}
\end{equation*} 
endowed with the residue norm. The category $\BanR$ carries a monoidal structure given by the completed tensor product $\wotimes_R$. Here $M\wotimes_R M'$ denotes the separated completion of $M\otimes_R M'$ with respect to the projective tensor norm; see also \cite[Subsection 2.1.7]{BoschAna}. We will use the following result, which is \cite[Proposition 3.27]{BamBen}.

\begin{pro}
    The category $(\BanR,\wotimes_R)$ is a closed symmetric quasi-abelian category with enough projective objects.
\end{pro}

In particular, given an object $M\in\BanR$, we define the object
\begin{equation*}
        c_0(M)=\left\{(c_m)_{m\in M- \{0\}} \in \prod_{m \in M - \{0\}} R_{\norm{m}_M} : \ \lim_{m\in M-\{0\}}\norm{c_m}_{R_{\norm{m}_M}}=0\right\}.
    \end{equation*}
 Where by  $ R_{\norm{m}_M}$ we indicate $R$ but with the norm for $ c_m \in R$ equal to  $\norm{c_m}_{R_{\norm{m}_M}}= \norm{c_m}_R \norm{m}_M$.  Here the product is the set-theoretic product, and the condition $\lim_m \norm{c_m}_{R_{\norm{m}_M}}=0$ means that, for every $\varepsilon>0$, the set of $m\in M\setminus\{0\}$ such that $\norm{c_m}_{R_{\norm{m}_M}}\geq \varepsilon$ is finite.
The module $c_0(M)$ is a Banach $R$-module when endowed with the norm
\begin{equation*}
    \norm{(c_m)_{m\in M-\{0\}}}_{(c_m)_{m\in M-\{0\}}}=\rm{sup}_{m\in M-\{0\}}\norm{c_m }_{R_{\norm{m}_M}}.
\end{equation*}
Equivalently, the module $c_0(M)$ is isomorphic to the coproduct 
\[ \contcoprod_{m\in M-\{0\}}R_{\norm{m}_M} \]
in $\BanR^{\le 1}$ (cfr. \cite[proof of Proposition 3.21]{BamBen}). The object $c_0(M)$ is projective in $\BanR$ (see \cite[Lemma 3.26]{BamBen}).
Moreover, the morphism
\begin{equation*}
    c_0(M)\rightarrow M
\end{equation*}
defined by
\begin{equation*}
    (c_m)_{m\in M-\{0\}}\mapsto\sum_{m\in M-\{0\}}c_m m
\end{equation*}
is a strict epimorphism (cfr. \cite[Lemma 3.27]{BamBen}).
In particular, $\BanR$ has enough projective objects and projective objects are flat (see \cite[Proposition 3.13]{Bambozzi_2024}). We will also use the following proposition.

\begin{pro}\label{tensor of banach projectives}
   Let $P,P'$ be two projective objects in $\BanR$, then $P\wotimes_R P'$ is also projective.    
\end{pro}
\begin{proof}
   From the discussion above, it follows that  every projective is a direct summand of an object of the form $c_0(M)$, for $M\in\BanR$,  it suffices to prove the claim for objects of this form. But if $M,M'\in\BanR$, then 
    \footnotesize
    \begin{equation*}
        c_0(M)\wotimes_R c_0(M')= \contcoprod_{m\in M-\{0\}}R_{\norm{m}_M}\wotimes_R \contcoprod_{m'\in M'-\{0\}}R_{\norm{m'}_M}\simeq \contcoprod_{m\in M-\{0\}}\contcoprod_{m'\in M'-\{0\}}R_{\norm{m}_M\norm{m'}_M'}.
    \end{equation*}
    \normalsize
    Here we use that $\BanR^{\leq 1}$ is closed symmetric monoidal, so that the completed tensor product preserves coproducts in each variable (see the proof of \cite[Lemma 3.26]{BamBen}). Applying the same argument as in \cite[Lemma 3.26]{BamBen}, we see that
    \[ \contcoprod_{m\in M\setminus\{0\}} \contcoprod_{m'\in M'\setminus\{0\}} R_{\norm{m}_M\norm{m'}_{M'}} \] 
    is projective in $\BanR$.
\end{proof}

By general results on Ind-categories of quasi-abelian categories \cite[Proposition 2.1.17]{Schnei}, the category $\Ind(\BanR)$ is complete and cocomplete and carries a natural elementary quasi-abelian structure. Moreover, the completed tensor product $\wotimes_R$ extends to a closed symmetric monoidal structure on $\Ind(\BanR)$ (see \cite[Proposition 2.1.19]{Schnei}). In the present Banach setting this is also recalled in \cite[Lemma 3.29]{BamBen}.

\begin{pro}\label{prop:ind-ban-is-quasi-abelian}
    The category $\Ind(\BanR)$ is a complete and cocomplete closed symmetric monoidal elementary quasi-abelian category with enough projectives. Moreover projective objects of $\Ind(\BanR)$ are flat. Additionally, if $P,P'$ are projectives in $\Ind(\BanR)$, then also $P\wotimes_R P'$ is.
\end{pro}
\begin{proof} 
The assertions about flatness and stability of projective objects under tensor products follow from the corresponding statements in $\BanR$, namely from \cite[Proposition 3.13]{Bambozzi_2024} and Proposition \ref{tensor of banach projectives}, together with \cite[Proposition 2.1.19]{Schnei}. 
\end{proof}
 
Assume now that $R$ is a Banach ring over a non-archimedean Banach field $K$ of characteristic $0$. We will use the following result. We refer to \cite[Definition 4.32]{kelly2021analytic} for the definition of a homotopical algebra context. 

\begin{thm}\label{thm:ind-ban-homotopical-algebra-context} 
The category $\mathcal C(\Ind(\BanR))$ of cochain complexes in $\Ind(\BanR)$, endowed with the projective model structure, is a homotopical algebra context. \end{thm} 
\begin{proof} 
This follows from \cite[Theorem 4.33]{kelly2021analytic}. Indeed, $\Ind(\BanR)$ is locally presentable by the general properties of ind-completions, it is elementary by \cite[Proposition 2.1.17]{Schnei}, and, by Proposition \ref{prop:ind-ban-is-quasi-abelian}, it has enough projective objects which are flat and closed under tensor products. 
\end{proof}

\begin{rmk}\label{symmetric}
   In particular, if $R$ is a Banach ring over $K$, then the $\infty$-category $\dercat_{\infty}(\Ind(\BanR))$ is a closed symmetric monoidal stable $\infty$-category. Its monoidal structure is induced by the left derived functor of $\wotimes_R$, which we denote by $\wotimes_R^{\bb L}$. The same applies to module categories over commutative algebra objects. Namely, if $A\in \sf{Comm}(\dercat_{\infty}(\Ind(\BanR)))$, then $\dercat_{\infty}(\sf{Mod}_A)$ is again a closed symmetric monoidal stable $\infty$-category.
\end{rmk}

We will now give a fundamental example of an object in $\Ind(\BanR)$.
\begin{exa}\label{polynomial algebra in IndBanR}
    Consider the $R$-module $R[t]_{\le n}$ of polynomials of degree less or equal than $n$, where $n\in \N$. We regard $R[t]_{\leq n}$ as a Banach $R$-module by endowing it with the norm
    \begin{equation*}
        \norm{a_nt^n+\dots+ a_1t+a_0}_{R[t]_{\le n}}=\rm{max}_{i=1,\dots, n}\abs{a_i}_R,
    \end{equation*}
    for $a_nt^n+\dots+ a_1t+a_0\in R[t]_{\le n}$.
    Note that we have isometric embeddings
    \begin{equation*}
        R[t]_{\le n}\hookrightarrow R[t]_{\le n+1}.
    \end{equation*}
    We define the polynomial algebra $R[t]$, as an object of $\Ind(\BanR)$, by 
    \[ R[t]=``\!\colim_{n\in\mathbb N}\!"\, R[t]_{\leq n}. \]
    It is an object of $\Ind^m(\BanR) \subset \Ind(\BanR)$, the full subcategory of \emph{essentially monomorphic ind-objects}, and the usual multiplication maps 
    \[ R[t]_{\leq n}\wotimes_R R[t]_{\leq n} \longrightarrow R[t]_{\leq 2n} \] 
    endow $R[t]$ with the structure of a commutative monoid in $\Ind(\BanR)$.
    Similarly one can consider the polynomial algebra in more variables $R[t_1,\dots, t_m]$, with $m\in\N$, as an ind-Banach ring, cf.\ \cite[pp. 15-16]{BassatMuk}. It follows directly from the construction that $R[t]$ is flat over $R$ (because it is projective). Moreover, in $\Ind(\BanR)$ one has \[ R[t_1]\wotimes_R\cdots\wotimes_R R[t_m] \simeq R[t_1,\dots,t_m]. \]
\end{exa}

\subsection{Bornological vector spaces} If $K$ is a non-archimedean non-trivially valued field, the category $\Ind(\sf{Ban}_K)$ is related to the theory of bornological vector spaces over $K$. We briefly review this concept in this subsection. The main reference for this is \cite[Subsection 3.3]{BamBen}.
\begin{dfn}
    Let $X$ be a set. A \emph{bornology} $\goth{B}$ on $X$ is a collection of subsets of $X$ that satisfies the following properties:
    \begin{enumerate}
        \item $\bigcup_{B\in\goth{B}}B=X$.
        \item If $B\in\goth{B}$ and $A\subset B$, then $A\in\goth{B}$.
        \item If $n\in\N$ and $B_1,\dots, B_n\in\goth{B}$, then $\bigcup_{i=1,\dots,n}B_i\in\goth{B}$.   
    \end{enumerate}
\end{dfn}
A set $X$ equipped with a bornology $\goth{B}$ is called a \emph{bornological set}. The elements of $\goth{B}$ are called \emph{bounded subsets}.

\begin{dfn}
    A function between two bornological sets $(X,\goth{B})$ and $(Y,\goth{B}')$ is said to be \emph{bounded} if it sends bounded subsets for $\goth{B}$ to bounded subsets for $\goth{B}'$.
\end{dfn}
We can now introduce bornological vector spaces over $K$.
\begin{dfn}
    A bornological vector space over $K$ is a vector space $V$ over $K$ that is also a bornological set, such that the addition and multiplication by scalars are bounded maps.
\end{dfn}
We denote by $\BornK$ the category of bornological vector spaces over $K$, with morphisms given by bounded $K$-linear maps.
\begin{rmk}
    There is a fully faithfull embedding
    \begin{equation*}
        \sf{Ban}_K\hookrightarrow \BornK,
    \end{equation*}
     given endowing each Banach space with the \emph{von Neumann bornology}, for which a subset is bounded if and only if it is absorbed by 
     all neighborhoods of the origin (see \cite[Lemma 3.17]{bambozzi2018stein} and \cite[p. 102]{Houzel}). This is the usual notion of boundedness for subsets of Banach spaces.
\end{rmk}

\begin{pro}
    The category $\BornK$ is complete and cocomplete.
\end{pro}

In particular the forgetful functor $\BornK \to \sf{Vec}_K$ commutes with all limits and colimits. 

\begin{dfn}
    A bornological vector space is called \emph{complete} if it is isomorphic to a monomorphic  filtered colimit of Banach spaces.
\end{dfn}

Thus we can define the full subcategory of complete bornological vector spaces $\CBornK\subset\BornK$.
In $\BornK$ we have a natural notion of convergence.

\begin{dfn}
    A sequence $\{x_n\}_{n\in\N}$ in $V\in\BornK$ \emph{converges to} $0$ \emph{in the sense of Mackey} if there exists a bounded subset $B\subset V$ such that for every $\lambda\in K^{\times}$ there exists an $N\in\N$ such that $x_n\in\lambda B$ for each $n\ge N$. A sequence $\{y_n\}_{n\in\N}$ in $V$ is said to \emph{converges to} $a\in V$ \emph{in the sense of Mackey} if the sequence $\{y_n-a\}_{n\in\N}$ converges to $0$ in the sense of Mackey. 
\end{dfn}

Note that by definition the convergence in the sense of Mackey coincides with the usual convergence in norm for objects in $\BanK$. 

\begin{dfn}
    Let $V\in\BornK$, a subset $U\subset V$ is called \emph{bornologically closed} if every sequence of elements in $U$ converging in the sense of Mackey converges to an element in $U$. 
\end{dfn}
The definition above helps to describe the quasi-abelian structure of $\CBornK$.

\begin{pro}
    The category $\CBornK$ is quasi-abelian.
\end{pro}
\begin{proof}
    This is proved in \cite[Lemma 3.53]{BamBen}.
\end{proof}

In particular morphism  $f:B\rightarrow B'$ in $\CBornK$ is a \emph{strict epimorphism} if it is surjective and $B'$ is endowed with the quotient bornology. Whereas $f$ is a \emph{strict monomorphism} if it is injective, $f(B)$ is bornologically closed in $B'$, and the bornology on $B$ agrees with the bornology induced by $B'$.

We now explain the relation between $\Ind(\BanK)$ and $\CBornK$.
The following is contained in \cite[Proposition 3.60]{BamBen}.

\begin{thm}\label{compare ind and born}
    There exists a functor
    \begin{equation*}
        \rm{diss}:\CBornK\rightarrow\Ind(\BanK)
    \end{equation*}    
    called \emph{dissection functor} that induces an equivalence of categories between $\CBornK$ and $\Ind^m(\BanK)$, the category of \emph{essentially monomorphic ind-objects}, and commutes with all limits and all filtered monomorphic colimits. The functor $\rm{diss}$ is isomorphic to the identity functor when restricted to $\BanK$. Moreover, the functor is exact in the sense of the theory of quasi-abelian categories.
\end{thm}

Theorem \ref{compare ind and born} suggests the following definition.

\begin{dfn}
Let $R$ be a Banach ring. The category of \emph{complete bornological modules over $R$} is the category of essentially monomorphic objects of $\sf{Ind}(\sf{Ban}_R)$.
\end{dfn}

The main reason for introducing the category $\CBornK$ is the following theorem.

\begin{thm}\label{thm:CBorn_IndBan_derived_equivalent}
    The dissection functor induces an equivalence of stable symmetric monoidal $\infty$-categories 
    \[ \rm{diss}: \dercat_\infty(\sf{CBorn}_R) \stackrel{\simeq}{\rightarrow} \dercat_\infty(\Ind(\sf{Ban}_R)). \]
\end{thm}
\begin{proof}
    This follows from \cite[Proposition 3.19]{Bambozzi_2024} using the fact that they are stable $\infty$-categories.
\end{proof}

Theorem \ref{thm:CBorn_IndBan_derived_equivalent} tells us that from the point of view of derived geometry, the categories $\CBornK$ and $\Ind(\BanK)$ are equivalent. However, $\CBornK$ is a much more concrete category than $\Ind(\BanK)$ and often permits the application of the methods and theorems of functional analysis. Therefore, we prefer to phrase our results in terms of bornological spaces instead of ind-Banach spaces whenever convenient.

\section{Tempered power series}\label{section tempered functions}

Let $R$ be a non-archimedean 
Banach ring. In this section we introduce the tempered power series over $R$ and prove some properties of these series that we will need later.
Let $n\in\N$, we set $\tate{R}{t}_{n}$ to be the $R$-module of power series 

\begin{equation*}
\tate{R}{t}_{n} = \left\{\sum_{i\in\N}a_it^i\in \power{R}{t}:a_i\in R,\ \lim_{i\rightarrow\infty}|a_i|(i+1)^{-n}\rightarrow 0\right\}.
\end{equation*}  

We can endow $\tate{R}{t}_n$ with the sup-norm 
\begin{equation}\label{sup norm}
\norm { \sum_{i\in \N}a_{i}t^i}_n = \sup_{i\in \N} \, ( |a_i|(i+1)^{-n}),
\end{equation}
and we see that $\tate{R}{t}_n$ is an element of $\BanR$. More precisely, $\tate{R}{t}_n$ is isomorphic to (see \cite[proof of Proposition 3.21]{BamBen})
\begin{equation} \label{eq:log_growth_projective}
    {\coprod_{i\in\N}}^{\le 1} R_{(i+1)^{-n}},
\end{equation}
where the coproduct is computed in the contracting category of Banach modules $\BanR^{\le 1}$. We also notice that $\tate{R}{t}_n$ is not an algebra with respect to the multiplication of power series.  And (see  section \ref{section notation})  by  $ R_{(i+1)^{-n}}$ we indicate $R$  with the  rescaled norm for $ x \in R$ equal to  $\norm{x}_{R_{(i+1)^{-n}}}= \norm{x}_R {(i+1)^{-n}}$.
Note that there are 
bounded monomorphisms
$$\tate{R}{t}_n\hookrightarrow\tate{R}{t}_{n+1}.$$

\begin{dfn}\label{definition tempered series}
    The ind-Banach module of \emph{tempered power series over} $R$ is the element of $\Ind(\BanR)$ defined as
    \begin{equation*}
        \temp{R}{t}=``\colim_{n\in\N}"\tate{R}{t}_n.
    \end{equation*}
\end{dfn}

Note that actually $\temp{R}{t}$ belongs to  $\Ind^m(\BanR)$ and that the usual multiplication of power series defines morphisms
$$\tate{R}{t}_n\wotimes_R\tate{R}{t}_n\rightarrow\tate{R}{t}_{2n}.$$
These morphisms endow $\temp{R}{t}$ with the structure of a commutative algebra object in $\Ind(\BanR)$.

\begin{dfn}
    Let $n\in\Z$, we denote
    $$ \power{R}{t}_n = \left\{\sum_{i\in\N}a_it_i\in \power{R}{t}:a_i\in R,\ \sup_{i\in\N}|a_i|(i+1)^{-n}<\infty\right\}.$$
    We endow $\power{R}{t}_n$ with the sup norm.
\end{dfn}

The object $\power{R}{t}_n$, endowed with the sup norm is isomorphic to the product 
$${\prod_{i\in \N}}^{\le 1} R_{(i+1)^{-n}}$$
in $\BanR^{\le 1}$ (see \cite[proof of Proposition 3.21]{BamBen}). This is the classical module of power series with coefficients in $R$ having log-growth $n$.
It is easy to see that
$$\temp{R}{t}=``\colim_{n\in\N}"\power{R}{t}_n.$$

\begin{lem}\label{tempered is flat}
    The module $\temp{R}{t}$ is a flat object in $\mathsf{Ind}(\mathsf{Ban}_R)$.
\end{lem}
\begin{proof}
    Since $$\tate{R}{t}_n\cong{\coprod_{i\in\N}}^{\le 1} R_{(i+1)^{-n}}$$ we have that $\tate{R}{t}_n$ is projective over $R$ \cite[Proposition 3.9]{Bambozzi_2024}, and therefore flat. Thus $\temp{R}{t}$ is flat because it is a filtered colimit of flat objects in an elementary quasi-abelian category.
\end{proof}

We now define multivariable tempered power series. We set 
$$\tate{R}{t_1,\dots,t_m}_n=\tate{R}{t_1}_n\wotimes_R\dots\wotimes_R\tate{R}{t_m}_n.$$
We can give a more concrete description of this object of $\BanR$.
\begin{pro}
    Let
    \begin{equation*}
       \mathcal{S} =
      \left\{\sum_{j_1,\dots, j_m\in\N}a_{j_1,\dots,j_m}t^{j_1}\dots t^{j_m}: \lim_{j_1+\dots+j_m\rightarrow +\infty}|a_{j_1,\dots,j_m}|((j_1+1)\dots (j_m+1))^{-n}=0\right\}  
    \end{equation*}
    endowed with the norm
    \begin{equation*}
        \norm{\sum_{j_1,\dots, j_m\in\N}a_{j_1,\dots,j_m}t^{j_1}\dots t^{j_m}}=\mathrm{sup}_{j_1,\dots,j_m\in\N}|a_{j_1,\dots,j_m}|((j_1+1)\dots (j_m+1))^{-n}.
    \end{equation*}
    We have a canonical isomorphism $$\tate{R}{t_1,\dots,t_m}_n\cong\mathcal{S}.$$      
\end{pro}
\begin{proof}
    The normed module $\cal{S}$ is isomorphic to
    $$\contcoprod_{j_1,\dots, j_m\in\N} 
 R_{((j_1+1)\dots (j_m+1))^{-n}},$$
    hence it is an object of $\BanR$. The claim follows from the fact that 
    $$\tate{R}{t_1,\dots,t_m}_n\cong\contcoprod_{j_1\in \N}R_{(j_1+1)^{-n}}\wotimes_R\dots\wotimes_R \contcoprod_{j_m\in\N}R_{(j_m+1)^{-n}},$$
    and from the fact that the coproducts in $\BanR^{\le 1}$ commute with 
    the completed tensor products, because $\BanR^{\le 1}$ is closed symmetric monoidal.
\end{proof}

We now define multivariable tempered power series.

\begin{dfn}\label{definition multivariable tempered series}
    The ind-Banach algebra of \emph{multivariable tempered power series in} $m$ \emph{variables} is 
    $$\temp{R}{t_1,\dots,t_m}=\quotecolim_{n\in\N}\tate{R}{t_1,\dots,t_m}_n.$$
\end{dfn}

We also define tempered power series for objects of $\Ind(\BanR)$.
Let $\mathcal{M}=\quotecolim_{j\in J} M_j$ be an object of $\Ind(\BanR)$, where $J$ is a set and $M_j\in\BanR$. We set $$\temp{\mathcal{M}}{t}=\mathcal{M}\wotimes_R\temp{R}{t}.$$ Note that this object is identified with \begin{equation*}
\quotecolim_{j\in J} \quotecolim_{n\in\N} M_j\wotimes_R\tate{R}{t}_n.
\end{equation*}
In particular, for any object $A \in \sf{Comm}(\Ind(\BanR))$ we have the associated algebra of tempered power series
\[ \temp{A}{t} = A \wotimes_R \temp{R}{t} = \quotecolim_{n \in \N} A \wotimes_R \tate{R}{t}_n = \quotecolim_{n \in \N} \tate{A}{t}_n. \]

The following is a technical lemma that we will need later on.

\begin{lem}\label{idclosedintemp}
    Let $A \in \sf{Comm}(\BanR)$.
    Let $g,g'\in A$ be multiplicative   unitary
elements, i.e. such that $\norm{g a}_A = \norm{g' a}_A = \norm{a}_A$ for all $a \in A$.
    Suppose moreover that $(g t- g')$ is not a zero divisor in $\temp{A}{t}$.
    Then the morphism induced by multiplication by $(gt-g')$:
    \begin{equation}\label{tempered ideal is strict}\temp{A}{t}\xrightarrow{(gt-g')}\temp{A}{t}\end{equation}
    is a strict monomorphism in $\Ind(\BanR)$.
\end{lem}
\begin{proof}
    Note that the morphism \eqref{tempered ideal is strict}
    sends $\tate{A}{t}_n$ in $\tate{A}{t}_n$.
    Therefore, it suffices to show that 
    \begin{equation}\label{strict morphism at some level}
    \tate{A}{t}_n\xrightarrow{(gt-g')}\tate{A}{t}_n
    \end{equation}
    is strict. Let  
    \begin{equation*}
    h(gt-g')\in\tate{A}{t}_n,\end{equation*} with $h=\sum_{i\in\N} a_it^i \in \tate{A}{t}_n$, and $n\in\N$. The coefficient of $t^i$ in $(gt-g')h$ is
    \[
        g a_{i-1}-g'a_i,
    \]
    with $a_{-1} = 0$.
    We will prove that 
    \begin{equation}
    \label{ineqimp}
    \norm{h}_n\le\norm{h(gt-g')}_n,
    \end{equation}
    with $|\cdot|_n$ the norm of $\tate{A}{t}_n$.
    This will imply the claim. 
    
    Set $C=\norm{(gt-g')h}_n$. Thus, for every $i\geq 0$, we have
    \[
        \norm{g a_{i-1}-g'a_i}_A (i+1)^{-n}\leq C .
    \]
    
    Fix $i\geq 0$. If
    \[
        \norm{a_{i-1}}_A\neq \norm{a_i}_A,
    \]
    then, since multiplication by $g$ and $g'$ is isometric, the non-archimedean triangle
    inequality gives
    \[
        \norm{g a_{i-1}-g'a_i}_A
        =
        \max\{\norm{a_{i-1}}_A,\norm{a_i}_A\}.
    \]
    In particular,
    \[
        \norm{a_i}_A(i+1)^{-n}\leq C.
    \]
    
    Suppose now that $\norm{a_{i-1}}_A=\norm{a_i}_A$. If $a_i=0$ there is nothing to
    prove. Otherwise, since $a_{-1}=0$, there exists a largest integer $\bar{i}<i$ such that
    \[
        \norm{a_{\bar{i}-1}}_A\neq \norm{a_i}_A.
    \]
    By maximality, $\norm{a_{\bar{i}}}_A=\norm{a_i}_A$. Applying the previous case to
    the index $\bar{i}$ gives
    \[
        \norm{a_{\bar{i}}}_A(\bar{i}+1)^{-n}\leq C.
    \]
    Since $\bar{i}\leq i$ and $n\geq 0$, we have $(i+1)^{-n}\leq(\bar{i}+1)^{-n}$.
    Therefore
    \[
        \norm{a_i}_A(i+1)^{-n}
        =
        \norm{a_{\bar{i}}}_A(i+1)^{-n}
        \leq
        \norm{a_{\bar{i}}}_A(\bar{i}+1)^{-n}
        \leq C.
    \]
    Taking the supremum over all $i$ proves \eqref{ineqimp}.

    Since $gt-g'$ is not a zero divisor, the map \eqref{strict morphism at some level} is
    injective. Moreover, if
    $(gt-g')f_m$ converges in $\tate{A}{t}_n$, then \eqref{ineqimp} implies that
    $\{f_m\}_m$ is Cauchy, hence converges to some $f\in\tate{A}{t}_n$, and the limit is
    $(gt-g')f$. Thus \eqref{strict morphism at some level} is a strict monomorphism in
    $\BanR$. Passing to the filtered Ind-colimit over $n$, we obtain that
    \eqref{tempered ideal is strict} is a strict monomorphism in $\Ind(\BanR)$.
\end{proof}

 \begin{rmk} In fact, the hypothesis that $(g t- g')$
is not a zero divisor could be weakened. The retraction remains continuous regardless, but we shall not make use of this more general result.
 \end{rmk}

 \medskip
 To generalize Proposition \ref{idclosedintemp}, we need to introduce  the notion of multiplicative element.

\begin{dfn} \label{defn:multiplicative_element}
Let $A\in \sf{Comm}(\sf{CBorn}_R)$, and let $g\in A$. We say that $g$ is
\emph{multiplicative} if there exists a presentation
\[
    A=\colim_{i\in I} A_i
\]
as a filtered monomorphic colimit of Banach $R$-modules such that multiplication by
$g$ is represented levelwise by isometries. More explicitly, for every $i\in I$ there
exists $j\in I$ and a representative
\[
    \mu_g^{i,j}\colon A_i\longrightarrow A_j
\]
of the map $a\mapsto ag$ such that
\[
    \norm{\mu_g^{i,j}(a)}_{A_j}=\norm{a}_{A_i}
\]
for every $a\in A_i$.
\end{dfn}
\medskip

\begin{cor} \label{cor:idclosedintemp}
Let $A\in \sf{Comm}(\sf{CBorn}_R)$, and let $g,g'\in A$ be multiplicative elements.
Assume that $gt-g'$ is not a zero divisor in $\temp{A}{t}$. Then multiplication by
$gt-g'$ induces a strict monomorphism
\[
    \temp{A}{t}
    \xrightarrow{\ gt-g'\ }
    \temp{A}{t}.
\]
\end{cor}

\begin{proof}
    Choose a presentation of $A$ as a filtered monomorphic colimit of Banach $R$-modules
    for which multiplication by $g$ and by $g'$ is represented levelwise by isometries. Since
    \[
        \temp{A}{t}
        =
        \colim_i \temp{A_i}{t},
    \]
    the proof of Lemma \ref{idclosedintemp} applies levelwise to this presentation. The same
    coefficient estimates show that multiplication by $gt-g'$ has closed image at each
    Banach level, and hence defines a strict monomorphism after passing to the filtered
    monomorphic colimit.
\end{proof}

We define now another algebra of power series: we will see later how it relates with the algebra of the  tempered ones. 

\begin{dfn}\label{definition fast series}
    We define the algebra of \emph{fast converging series}  to be
    $$\fast{R}{t}=\lim_{n\in\N}\power{R}{t}_{-n},$$
    where the limit is computed in $\Ind(\BanR)$ and 
    \[ \power{R}{t}_{-n} = \left\{\sum_{i\in\N}a_it^i\in \power{R}{t} : a_i\in R,\ \lim_{i \to \infty} \abs{a_i}(i+1)^n = 0 \right\}  \]
    equipped with the norm
    \[  
    \norm { \sum_{i\in \N}a_{i}t^i}_n = {\sup_{n \in \N} }\, ( \abs{a_i}(i+1)^n).
    \]
\end{dfn}

We also mention that we can define the multivariable algebra of fast converging series as
\[  \fast{R}{t_1, \ldots, t_n} = \fast{R}{t_1} \wotimes_R \cdots \wotimes_R \fast{R}{t_m}. \]
We will not study this multivariable algebra in detail here, since it will not be needed in this article  It will be studied in our future work  \cite{logdecay&other}

\section{The derived analytic spectrum of an ind-Banach ring}\label{section spectrum}

Let $\cal{E}$ be a closed symmetric monoidal $\infty$-category. In \cite[Construction 5.2]{ClauSchCompl}, Clausen and Scholze show that the collection of homotopy epimorphisms from the unit of $\cal{E}$ (or derived idempotent algebras in the terminology of \cite{ClauSchCompl}) is naturally a distributive lattice. In their construction, this lattice is interpreted as the lattice of closed subsets of a locale, which we denote here by $\mathcal S(\mathcal E)$. In general this locale need not be spatial, and hence need not come from an ordinary topological space. This point of view is used in \cite{ClauSchCompl} as a foundation for analytic geometry. Classical localizations in the usual models of algebraic and analytic geometry appear as closed subsets in their convention. For instance, if $A$ is a discrete ring and $A\to A_f$ is a localization, then $A\to A_f$ is a homotopy epimorphism in $\dercat_\infty(\sf{Mod}_A)$, and therefore determines a closed subset of $\mathcal S(\dercat_\infty(\sf{Mod}_A))$.
In this section we use the same lattice of homotopy epimorphisms, but with the opposite geometric convention: we construct from it a spectral topological space $\goth{S}(\mathcal E)$ whose compact open subsets correspond to homotopy epimorphisms from the unit. Thus $\goth{S}(\mathcal E)$ should be regarded as a Hochster-dual version of the Clausen--Scholze spectrum. If $A\in\sf{Comm}(\Ind(\BanR))$, then \[ \goth{S}\bigl(\dercat_\infty(\sf{Mod}_A)\bigr) \] will be our model for the derived analytic spectrum of $A$. This construction is compatible with the classical picture: for a discrete ring $A$, the localization $A\to A_f$ defines an open subset of $\goth{S}(\dercat_\infty(\sf{Mod}_A))$. Analogous results in analytic settings have been obtained in \cite{BenKre}, \cite{BamBen}, \cite{bambozzi2018stein}, \cite{BenKreFrechet}, \cite{BassatMuk}, \cite{bambozzi2023homotopy}. The subject is broader than this list of references, but these are the works most closely related to the present construction. The space $\goth{S}(\mathcal E)$ will be naturally equipped with a structure sheaf valued in $\mathcal E$.

\subsection{Preliminaries of lattice theory}
The main reference for this subsection is \cite{Johnstone}.
Recall that a \emph{lattice} is a partially ordered set $(L,\le)$ such that any finite set admits a least upper bound, also called \emph{join} and a greatest lower bound, also called \emph{meet}. In lattice theory, if $a,b\in L$ it is common use to denote the join and the meet of $\{a,b\}$ by $a\vee b$ and $a\wedge b$ respectively. More generally if $S$ is a subset of $L$, the join of $S$ (if it exists) will be denoted by $\bigvee S$. A lattice $L$ is called \emph{complete} if it admits also all infinite joins.
A lattice is called \emph{distributive} if 
$$a\wedge(b\vee c)=(a\wedge b) \vee (a\wedge c)$$
for any $a,b,c\in L$. We can then define the category $\sf{DLat}$ of distributive lattices setting its objects to be distributive lattices and morphisms to be functions that preserve finite meets and joins.

\begin{dfn}
A \emph{frame} is a complete distributive lattice that satisfies the so called \emph{infinite distributive law}:
$$ a\wedge\bigvee S=
 \bigvee_{s \in S} a \wedge s,$$
for every $a\in L$ and for every $S$ subset of $L$. The category $\sf{Frm}$ of frames, is the category whose objects are frames and whose morphisms are functions that preserve finite meets and arbitrary joins.
\end{dfn}

The following is the fundamental example of a frame.

\begin{exa}
    Let $X$ be a topological space. The poset of open subsets of $X$ (where the partial order is given by the inclusion) is a frame. We will denote it $\Omega(X)$.
\end{exa}

This example leads naturally to consider the category $\sf{Frm}^{\rm{op}}$ as the category of a sort of \say{pointless topological spaces} that we will denote by $\sf{Loc}$. The objects of $\sf{Loc}$ are called \emph{locales}.
Let $L$ be a distributive lattice. Recall that a \emph{descending subset} $S$ of $L$ is a subset of $L$ such that, if $a\in S$ and $a'\le a$, then $a'\in S$.
An \emph{ideal} of $L$ is a descending subset closed under finite joins. 
\begin{exa}
    Let $a\in L$, the principal ideal generated by $a$, denoted by $(a)$ is the ideal of $L$ whose elements are the $s\in L$ such that $s\le a$.
\end{exa}

The set of ideals of $L$, ordered by inclusion, will be denoted by $\mathrm{Idl}(L)$.

\begin{pro}\label{latt to frame compl}
    The poset $\rm{Idl}(L)$ is a frame.  
    Moreover the functor
    $$L\mapsto\rm{Idl}(L)$$
    defines a left adjoint to the forgetful functor
    $$\sf{Frm}\rightarrow \sf{DLat}.$$
\end{pro}
\begin{proof}
    See \cite[Corollary II.2.11]{Johnstone}.
\end{proof}
It is easy to see that the meet of two elements $I,I'\in\rm{Idl}(L)$ is the intersection, whereas the join is the ideal generated by $I$ and $I'$. This implies the following proposition.

\begin{pro}\label{incl in ideals}
    The function
    $$L\rightarrow\rm{Idl}(L)$$
    defined by
    $$a\mapsto (a)$$
    is a distributive lattice morphism.
\end{pro}
Propositions \ref{latt to frame compl} and \ref{incl in ideals} mean that $\rm{Idl}(L)$ is a kind of \say{frame completion} of $L$. 
We now want to see when we can associate a topological space to a locale.
An ideal $I$ of $L$ is called \emph{prime} if it satisfies the following condition: if $a$ and $b$ are elements of $L$ such that $a\wedge b\in I$, then $a\in I$ or $b\in I$. 
A \emph{prime element} of $L$ is an element $a\in L$ such that the ideal $(a)$ is prime. 
Let now $\cal{L}$ be a frame.
The set of prime elements of $\cal{L}$ will be denoted by $\rm{pt}(\cal{L})$.
If $a\in \cal{L}$ then we can define the set 
$$D_a=\{p\in \rm{pt}(\cal{L}): a \ \rm{is} \ \rm{not} \ \rm{contained} \ \rm{in} \ p\}.$$
The following proposition is \cite[Lemma II.1.3]{Johnstone}.
\begin{pro}
    The family $\{D_a\}_{a\in
     {L}}$  forms the basis for a topology on $\rm{pt}(\cal{L})$.
\end{pro}
Thus $\rm{pt}(\cal{L})$ is a topological space, endowed with a Zariski-like topology. The assignment $a\rightarrow D_a$ defines a morphism of frames 

\begin{equation}\label{spatial morph}
\cal{L} \to \Omega(\rm{pt}(\cal{L})).
\end{equation}

\begin{dfn}
    A frame $\cal{L}$ is called \emph{spatial} if the morphism \eqref{spatial morph} is an isomorphism.
\end{dfn}

A spatial frame is thus canonically associated to a topological space.

\begin{pro}
    Let $L$ be a distributive lattice then $\rm{Idl}(L)$ is a spatial frame. 
\end{pro}
\begin{proof}
    This follows from \cite[Theorem II.3.4 and Proposition II.3.2]{Johnstone}.
\end{proof}

The topological space associated to $\rm{Idl}(L)$  is moreover a spectral topological space.

\begin{pro}\label{ideal spectral space}
    The space $\rm{pt}({\rm{Idl}(L)})$ 
    is a spectral topological space, and principal ideals constitute a basis of compact open subset for it.
\end{pro}
\begin{proof}
    See \cite[p. 66 and the proof of Proposition II.3.2]{Johnstone}.
\end{proof}
\begin{rmk}
    The functor 
    $$L\mapsto\rm{pt(Idl(L))}$$ 
    defines indeed a duality between $\sf{DLat}$ and the category of spectral topological spaces (Stone's duality). Its inverse functor is given by
    $$X\mapsto K\Omega(X),$$
    where $K\Omega(X)$ is the lattice of compact open subsets of $X$ (see \cite[Corollary II.3.4]{Johnstone}).
\end{rmk}

\subsection{The spectrum of a closed symmetric stable $\infty$-category}\label{definition of the spectrum}

We begin by recalling the following definition.

\begin{dfn}
    Let $(\cal{E},\botimes)$ be a symmetric monoidal $\infty$-category. Let $A, A'$ be commutative monoids in $\cal{E}$. A \emph{homotopy epimorphism} in $\cal{E}$ is a morphism 
    $$A \to A'$$
    such that the multiplication morphism
    $$A'\botimes_{A} A'\to A'$$
    is an equivalence.
\end{dfn}

We denote by $\goth{I}(\cal{E})$ the collection of homotopy epimorphisms from $1$ in $\cal{E}$, i.e. homotopy epimorphisms of the form
$$1\to A,$$
with $A \in \sf{Comm}(\cal{E})$.
The following proposition is proved in \cite[Construction 5.2]{ClauSchCompl}.

\begin{pro}\label{unique map idempotents} 
Let $A,A'\in\goth I(\mathcal E)$. Then the space of maps $A\to A'$ compatible with the maps from $1$ is either empty or contractible.
\end{pro}

The proposition above 
permits us to give a partial order to $\goth{I}(\cal{E})$, setting $A'\le A$ if there exists a map $A\rightarrow A'$, for $A,A'\in\goth{I}(\cal{E})$, commuting with the morphisms from $1$. 

\begin{pro}\label{alg object form dist lattice}
    Let $\cal{E}$ be a closed symmetric stable $\infty$-category. Then $(\goth{I}(\cal{E}),\le)$ is a distributive lattice. 
\end{pro}
 \begin{proof}
     See \cite[Construction 5.2]{ClauSchCompl}, cfr. also \cite[Theorem 5.5]{BalmerFrame}.
 \end{proof}
 
In particular, if $A,A'\in\goth{I}(\cal{E})$, then $A\wedge A'=A\botimes A'$, whereas $A\vee A'$ is the fiber $$[A\oplus A'\rightarrow A\botimes A']$$ (see again \cite[Contruction 5.2]{ClauSchCompl}, or \cite[Proposition 5.2]{BalmerFrame}). Here the bracket denotes the fiber in the stable $\infty$-category $\mathcal E$.
By Proposition \ref{ideal spectral space} from $\goth{I}(\cal{E})$ we can construct the spectral space $\rm{pt}(\rm{Idl}(\goth{I}(\cal{E})))$.
For the sake of brevity, we will rename this space to $\goth{S}(\cal{E})$.

Let now $A\in\goth{I}(\cal{E})$. We will write $D_A$ for the corresponding open set in $\goth{S}(\cal{E})$. Note that the map 
\begin{equation*}
     D_A\mapsto A
\end{equation*}
defines a presheaf on the category of homotopy epimorphisms from $1$ in the category $\cal{E}$, with values in $\sf{Comm}(\cal{E})$. We denote this presheaf by $\O_{\goth{S}(\cal{E})}$. 

\medskip

 \begin{rmk} \label{rmk:cover_is_conservative}
   Building on this description of $A\vee A'$ it is easy corollary of Lurie-Barr-Beck  theorem  that the condition on $D_A$ and $D_{A'}$ to cover $\goth{S}(\cal{E})$ is equivalent to asking that the family of localizations $\{1 \to A, 1 \to A' \}$ is conservative. In the sense that the restriction functor $\cal{E} \to \sf{Mod}_A(\cal{E}) \oplus \sf{Mod}_{A'}(\cal{E})$ is conservative, where the latter are the categories of $A$ and $A'$ modules in $\cal{E}$. 
   This is the cover condition of the homotopical Zariski topology introduced in \cite{Vez} and used for bornological algebras in \cite{BenKre}, \cite{BamBen}, \cite{bambozzi2018stein}, \cite{bambozzi2023homotopy}, \cite{Bambozzi_2024}, and other papers.
 \end{rmk}

\begin{pro} \label{AFF}
    The presheaf $\O_{\goth{S}(\cal{E})}$ extends uniquely to a sheaf (in the sense of \cite[p. 763, Subsection 7.3.3]{LurieHTT}) on $\goth{S}(\cal{E})$.
\end{pro}
\begin{proof}
    We will apply \cite[Theorem 7.3.5.2]{LurieHTT}.
    Note that if $0$ is the zero object of $\cal{E}$, then $D_{0}=\varnothing$. So
    $$\O_{\goth{S}(\cal{E})}(\varnothing)=0,$$
    that is the final object of $\sf{Comm}(\cal{E})$.
    Thus it suffices to see that if $1\to A$, $1\to A'$ are homotopy epimorphisms, then the square
     \begin{equation*}
       \begin{tikzcd}
      \O_{\goth{S}(\cal{E})}(D_A\cup D_{A'})\arrow{r} \arrow{d} & \O_{\goth{S}(\cal{E})}(D_A) \arrow{d} \\
      \O_{\goth{S}(\cal{E})}(D_{A'})\arrow{r} & \O_{\goth{S}(\cal{E})}(D_A\cap D_{A'})
     \end{tikzcd}
    \end{equation*}
    is homotopy cartesian. But this follows immediatly from the definitions of meets and joins in $\goth{I}(\cal{E})$.
\end{proof}

\begin{rmk}\label{functions on complementary}
    As explained in \cite[Construction 5.2]{ClauSchCompl}, a homotopy epimorphism $1\to A$ gives rise to a localization of $\cal{E}$ whose local objects are the $A$-modules. More precisely, the forgetful functor \[ i\colon \sf{Mod}_A\hookrightarrow \mathcal E \] is fully faithful. It admits a left adjoint \[ i^l(E)=E\botimes A \] and a right adjoint \[ i^r(E)=\underline{\mathrm{Hom}}(A,E), \] for $E\in\mathcal E$.
    We also have the Verdier quotient (see \cite[Subsection 5.1]{blumberg2013universal} for a discussion in the $\infty$-categorical setting)
    $$j:\cal{E}\rightarrow\cal{E}/\sf{Mod}_A.$$
    This is a Bousfield localization that comes equipped with a fully faithful left adjoint $j_{l}$ determined by
    $$j_lj(E)=[E\rightarrow E\botimes A],$$
    and a fully faithful right adjoint $j_{r}$ determined by
    $$j_{r}j(E)=\IHom([1\to A],E).$$
    In particular we have a diagram
    \begin{equation}\label{categorical recollement}
      \begin{tikzcd}
          \sf{Mod}_A\arrow[r, "i"] & \cal{E}\arrow[r, "j"]\arrow[l, bend right=50, "i^r"]\arrow[l, bend left=50, "i^l"] & \cal{E}/\sf{Mod}_A.\arrow[l, bend left=50, "j_l"]\arrow[l, bend right=50, "j_r"]
      \end{tikzcd}
    \end{equation}
    The object $\IHom([1\rightarrow A],1)$ in \cite{ClauSchCompl} is therefore interpreted as \emph{functions supported on the 
    complement of $D_A$}, see \cite[Proposition 5.5]{ClauSchCompl}.  
\end{rmk}
\medskip

We now specialize this construction to the case of module categories over ind-Banach algebras over a non-Archimedean commutative ring $R$.

\begin{dfn} \label{DEF}
    Let $A\in\sf{Comm}(\dercat_{\infty}(\Ind(\BanR)))$. By Remark \ref{symmetric}, $\dercat_{\infty}(\sf{Mod}_{A})$ is a closed symmetric stable $\infty$-category, then the \emph{derived analytic spectrum} of $
    A$ is the couple 
    \begin{equation}
    \label{the derived analytic spectrum}
    (\goth{S}(\dercat_{\infty}(\sf{Mod}_{A})),\O_{\goth{S}(\dercat_{\infty}(\sf{Mod}_{A}))}).
    \end{equation}
    To simplify the notation we will simply write 
    $$
    (\goth{S}(A),\O_{\goth{S}(A)})
    $$ 
    to denote the object \eqref{the derived analytic spectrum}.  
\end{dfn}

\begin{exa}\label{affidem}
    Let $A_K$ be an affinoid algebra over $K$. Every $K$-affinoid localization of $A_K$ is a
    homotopy epimorphism in $\dercat_{\infty}(\sf{Mod}_{A_K})$ by \cite[Theorem 5.16]{BenKre}. So it defines an open
    subset in $\goth{S}(A_K)$.  
\end{exa}

    The motivation 
    for the construction of $\goth{S}(\cal{E})$ comes from homotopical algebraic geometry. In fact, homotopy epimorphisms correspond to 
    formal Zariski open immersions in the sense of \cite[Definition 1.2.6.1 (3)]{Vez} in the homotopical geometry associated with $\cal{E}$. Moreover, joins of homotopy epimorphisms encode the notion of Zariski cover
    of \cite[Definition 1.2.5.1]{Vez}. 
    This homotopical Zariski topology was also used in \cite{BenKre} and \cite{BamBen} to construct derived analytic spaces over $R$. 
    
    \begin{rmk}
        Note that any morphism $A\xrightarrow{f} B$ in $\sf{Comm}(\Ind(\BanR))$ induces a continuous map
        \begin{equation}\label{map between analytic spectra}\goth{S}(f):\goth{S}(B)\rightarrow\goth{S}(A).
        \end{equation}
        Indeed, if $A\to C$ is a homotopy epimorphism, then the base change $B \to B \wotimes^{\bb{L}}_A C$ is again a homotopy epimorphism. In fact,
        $$(B\wotimes^{\bb{L}}_A C) \wotimes^{\bb{L}}_B (B\wotimes^{\bb{L}}_A C) \cong C\wotimes^{\bb{L}}_A (B \wotimes^{\bb{L}}_B B) \wotimes^{\bb{L}}_A C \cong C\wotimes^{\bb{L}}_A C \wotimes^{\bb{L}}_A B \cong C \wotimes^{\bb{L}}_A B.$$
        Moreover, for two homotopy epimorphisms $A \to C$ and $A \to D$, and any morphism $A \to B$
        $$[(C\wotimes^{\bb{L}}_A B)\oplus (D\wotimes^{\bb{L}}_A B)\rightarrow (C\wotimes^{\bb{L}}_A B)\wotimes^{\bb{L}}_{B} (D\wotimes^{\bb{L}}_A B)]\cong[C\oplus D\to C\wotimes^{\bb{L}}_{A}D]\wotimes^{\bb{L}}_A B,$$
        because the derived tensor product also commutes with homotopy fibers and direct sums.
        Therefore we obtain a lattice morphism
        \begin{equation}\label{morphism of lattices}\goth{I}(\dercat_{\infty}(\sf{Mod}_A))\to \goth{I}(\dercat_{\infty}(\sf{Mod}_B)),
        \end{equation}
        that induces, by functoriality, the desired continuous map
        \eqref{map between analytic spectra}.
        Note that we also get a morphism of sheaves
        $$ \goth{S}(f)^{-1}(\O_{\goth{S}(A)})\rightarrow \O_{\goth{S}(B)},$$
        that on the basic open $D_C$,  for $C\in \goth{I}(\dercat_{\infty}(\sf{Mod}_A))$,  is given by 
        $$C\to C\wotimes^{\bb{L}}_A B.$$
         Thus $\goth{S}(f)$ is a morphism of \emph{homotopically ringed spaces}. 
    Note  that it is easy to check that if $A \to B$ is a homotopy epimorphism, then the map \eqref{map between analytic spectra} is a topological embedding. 
    \end{rmk}

 \subsection{The affine space over a Banach ring}
 Let $R$ be a Banach ring. In this subsection we study affine spaces over $R$ in the ind-Banach setting.
 
 \begin{dfn}\label{definition affine space}
     Consider the polynomial algebra $R[t_1,\dots,t_d]\in\Ind(\BanR)$ introduced in Example \ref{polynomial algebra in IndBanR}.
     The $d$-dimensional \emph{affine space} over $R$ is the derived analytic spectrum $\goth{S}(R[t_1,\dots,t_d])$. We will denote it by $\mathbf{A}^{d}_{R}$.  
 \end{dfn}

The topological space underlying $\A_R^d$ is very large and it is unreasonable to expect a complete description of it. Thus, we limit ourselves to identifying some families of open
subsets of $\mathbf{A}^1_{R}$. To do so, we use a criterion used in \cite[Section 5]{BassatMuk} to identify when a morphism is a homotopy epimorphism. We recall it for the reader's convenience.

 Let $(\sf{C},\botimes)$ be closed symmetric monoidal quasi-abelian category with enough flat projective objects, and let $X\in\sf{C}$.
 Recall that an \emph{element} of $X$ is a morphism $1\rightarrow X$ form the unit object. In the classical category of modules over a ring this notion coincides with the notion of an element of the underlying set. So, we will write $x\in X$ to say that $x$ is an element of $X$.
 \begin{dfn}
     Let $A\in\sf{Comm}(\sf{C})$, 
     and let $a$ be an element of 
     $A$. The \emph{diagonal Koszul resolution} of $A$, denoted by $K_A$, is the complex
     $$A\botimes A\xrightarrow{(a\botimes 1-1\botimes a)} A\botimes A.$$
     This notation means that the morphism is induced by multiplication by $(a\botimes 1-1\botimes a)$. The pair $(A,a)$ is said to satisfy \emph{the strictness condition} if $A$ is flat and if the morphism 
     $(a\botimes 1-1\botimes a)$ is strict and its cokernel is identified with $A$.
 \end{dfn}
 The following example is an example of a pair that satisfies the strictness condition. 
 \begin{exa}
     The pair $(R[t],t)$ satisfies the strictness condition in $\Ind(\BanR)$, see \cite[Lemma 4.6]{BassatMuk}.
 \end{exa}
 The criterion of \cite[Section 5]{BassatMuk} is expressed in the following theorem.
 \begin{thm}\label{Taylor criter}
     Let $A\xrightarrow{f} A'$ be a morphism in $\sf{C}$ and let $a\in A$, $a'\in A'$ such that $f(a)=a'$ and the pairs $(A,a)$ and $(A',a')$ satisfy the strictness condition. Then $f$ is a homotopy epimorphism. 
 \end{thm}
 \begin{proof}
    See \cite[Theorem 5.3]{BassatMuk}.
\end{proof}
 This criterion is used in \cite[Proposition 5.3]{BassatMuk} to prove the following proposition.
 \begin{pro}
     Let $\tate{R}{t}$ be the one variable Tate algebra over $R$. Then, 
     the canonical morphism $R[t] \to \tate{R}{t}$ is a homotopy epimorphism.
 \end{pro}
 \begin{proof}
     The claim follows from the fact that $(\tate{R}{t}, t)$ satisfies the strictness condition, see \cite[Corollary 5.6]{BassatMuk}.
 \end{proof}
 Thus
 $\tate{R}{t}$ defines an open 
 subset in $\mathbf{A}^1_{R}$. We can identify it as the \emph{closed unit disk centered in $0$}. 
 
 We will also use the following stability property of homotopy epimorphisms under
tensor products.

 \begin{pro}\label{prod idempotent}
     Let $A,A' \in \sf{Comm}(\sf{C})$, and let $A\to B$ and $A\to B'$ be homotopy epimorphisms. Then $A\botimes^{\mathbb{L}} A'\to B\botimes^{\mathbb{L}} B' $ is a homotopy epimorphism.
 \end{pro}
\begin{proof}
This is \cite[Proposition 3.6]{BassatMuk}.
\end{proof}
 
\begin{cor}
The algebra $\tate{R}{t_1,\dots,t_d}$ defines an open subset of $\mathbf A^d_R$,
which we call the \emph{closed unit polydisk centered at $0$}.
\end{cor}

\begin{pro}\label{tempidem}
    The canonical morphism \[ R[t]\longrightarrow \temp{R}{t} \] is a homotopy epimorphism. Equivalently, $\temp{R}{t}$ defines an open subset of $\mathbf A^1_R$.
\end{pro}
\begin{proof}
   We have   to check that the multiplication map $\temp{R}{t}\dertens{R[t]}\temp{R}{t}\rightarrow\temp{R}{t}$ is an isomorphism.
    By \cite[Lemma 4.6]{BassatMuk}, Lemma \ref{tempered is flat} and 
    Theorem \ref{Taylor criter} it suffices to prove that the sequence
\begin{equation}
    0\rightarrow\temp{R}{t}\hat{\otimes}_R\temp{R}{x}\xrightarrow{(t-x)}\temp{R}{t}\hat{\otimes}_R\temp{R}{x}\rightarrow\temp{R}{z}\rightarrow 0
\end{equation}
is strictly exact. To prove this it is enough to show that the ideal $(t-x)\temp{R}{t}\hat{\otimes}_R\temp{R}{x}$ is closed\footnote{This is a consequence of the bornological open mapping theorem for LB-spaces, see \cite{BamClosed}.}. We write $\temp{R}{t}\hat{\otimes}_R\temp{R}{x}=\bigcup_{n\in\N}\tate{R}{t,x}_n$, where on $\tate{R}{t,x}_n$ we have the tensor norm given by $\mathrm{sup}_{i,j\in\N}|a_{i,j}|/((i+1)(j+1))^n$ for a series
$\sum_{i,j}a_{i,j}t^ix^j$. Let 
\begin{equation*}
f\in\tate{R}{t,x}_n
\end{equation*} 
such that $f(x,x)=0$, which is equivalent to $f \in (t-x)\temp{R}{t}\hat{\otimes}_R\temp{R}{x}$.
Following 
the strategy of \cite[Section 4]{BassatMuk}, we will prove that the map 
\begin{equation}\label{division by (x-t)}f\mapsto\frac{f(t,x)-f(x,x)}{t-x}
\end{equation} 
is controlled by the $n$-norm of $f$ (i.e. the norm of $f$ in $\tate{R}{t,x}_{n}$). It will follow that the ideal 
\begin{equation*}
(t-x)\temp{R}{t}\hat{\otimes}\temp{R}{x}
\end{equation*} 
is closed, because if $(t-x)f_m$ is a sequence converging in $\tate{R}{t,x}_{n}$, $f_m$ converges in $\tate{R}{t,x}_{2n}$. 
This will entail that the map \eqref{division by (x-t)} induces bounded injective morphisms of $R$-Banach modules
$$\phi_n:\tate{R}{t,x}_n\to \tate{R}{t,x}_{2n},$$ for $n\in\N$. Passing to colimits we obtain a map
$$\phi: (t-x)\temp{R}{t, x} \cong \colim_{n \in \N} \tate{R}{t,x}_n \to \colim_{n \in \N} \tate{R}{t,x}_{2 n},$$
    that by cofinality satisfies $\phi\circ (t-x)=\rm{id}_{\temp{R}{t,x}}$, $(t-x)\circ \phi=\rm{id}_{(t-x)\temp{R}{t, x}}$ by construction.

So let $f(x, t) = \sum_{i,j}a_{i,j}t^ix^j$, we have that
\begin{equation*}\frac{f(t,x)-f(x,x)}{t-x}=\sum_{i>0,j\ge 0}\sum_{l=0}^{i-1}  a_{i, j}t^{i-1-l}x^{j+l}=\sum_{h\ge 0,s\ge 0}(\sum_{l=0}^{s}a_{h+1+l,s-l}) t^hx^s.
\end{equation*} 
Then
\begin{dmath*}
\left\|\frac{f(t,x)-f(x,x)}{t-x}\right\|_{2n}\le \mathrm{sup}_{h\ge 0,s\ge 0}\mathrm{max}_{l=0,\dots, s}(|a_{h+1+l,s-l}|)(h+1)^{-2n}(s+1)^{-2n}\le 2^n\mathrm{sup}_{h\ge 0,s\ge 0}\mathrm{max}_{l=0,\dots, s}(|a_{h+1+l,s-l}|(h+1+l+1)^{-n}(s-l+1)^{-n})\le
2^n\mathrm{sup}_{i,j\in\N}|a_{i,j}|((i+1)(j+1))^{-n}\le 2^n\|f\|_n,
\end{dmath*}
where the second inequality follows from 
\[ 2^n(h+l+2)^{-n}(s-l+1)^{-n} \geq (h+1)^{-2n}(s+1)^{-2n}, 
\]
for $h\in\N$, $s\in\N$, $l\in\N$, $l\le s$. For $n=0$ there is nothing to prove, so we assume $n>0$. Since all the quantities involved are positive, the inequality is equivalent, after taking the $n$-th root and then inverting, to 
\[ (h+l+2)(s-l+1) \leq 2(h+1)^2(s+1)^2. \] 
We prove this latter inequality. Since $0\leq l\leq s$, we have $s-l+1\leq s+1$ and 
\[ h+l+2 \leq (h+1)+(s+1) \leq 2(h+1)(s+1). \] 
Multiplying we obtain 
\[ (h+l+2)(s-l+1) \leq 2(h+1)(s+1)^2. \] 
Finally, since $h+1\geq 1$, we have 
\[ (h+l+2)(s-l+1) \leq 2(h+1)^2(s+1)^2, \] 
which is equivalent to the desired inequality.
\end{proof}

In the following lemmas we determine some further open subsets of $\mathbf A^1_R$. For simplicity, we restrict to the case $R=K$.

\begin{lem}\label{boundidem}
    The algebra $\power{K}{t}_0$ defines an open subset in $\mathbf{A}^1_K$.
\end{lem}
\begin{proof}
    By \cite[Lemma 3.49 and 50]{bambozzi2018stein} $\power{K}{t}_0$ is flat over $K$. We can apply again Theorem \ref{Taylor criter}.
    Using the same argument as in the proof of Proposition \ref{tempidem}, we only need to prove that if $f\in\power{K}{t,x}_0$, then the map
    \begin{equation*}
        f\mapsto\frac{f(t,x)-f(x,x)}{t-x}
    \end{equation*}
    is bounded. But this is evident as 
    \begin{equation*}
        \frac{f(t,x)-f(x,x)}{t-x}=\sum_{h\ge 0,s\ge 0}(\sum_{l=0}^{s}a_{h+1+l,s-l})t^hx^s,
    \end{equation*}
    if $f(t, x) = \sum_{i,j} a_{i,j}t^ix^j$ as before, and we get by the ultrametric triangle inequality
    \[ 
    \begin{aligned} \left\| \frac{f(t,x)-f(x,x)}{t-x} \right\|_0 &\leq \sup_{h,s\geq 0} \max_{0\leq l\leq s} \abs{a_{h+1+l,s-l}}_K \\ &\leq \sup_{i,j\geq 0}\abs{a_{i,j}}_K = \norm{f}_0. \end{aligned} 
    \] 
    Thus the map 
    \[ f\mapsto \frac{f(t,x)-f(x,x)}{t-x} \] 
    is bounded on $\power{K}{t,x}_0$.
\end{proof}

\medskip

\begin{lem} \label{lem:fast_is_open}
    The algebra $\fast{K}{t}$ of Definition \ref{definition fast series} defines an open subset in $\mathbf{A}^1_K$. 
\end{lem}
\begin{proof}
   
    By \cite[Lemma 3.50]{bambozzi2018stein} the algebra $\fast{K}{t}$ is flat over $K$.
     Thus, by Theorem \ref{Taylor criter}, we need to show that
     \begin{equation}\label{fast koszul resolution}
    0\rightarrow\fast{K}{t}\hat{\otimes}_K\fast{K}{x}\xrightarrow{(t-x)}\fast{K}{t}\hat{\otimes}_K\fast{K}{x}\rightarrow\fast{K}{z}\rightarrow 0
     \end{equation}
     is strictly exact. Again, we use the strategy of Proposition \ref{tempidem} and reduce the proposition to prove that the map 
    \begin{equation*}
        f\mapsto\frac{f(t,x)-f(x,x)}{t-x}
    \end{equation*}
    is bounded.
    If $f(t, x) =\sum_{i,j}a_{i,j}t^ix^j$ then we have  
    \begin{dmath*}
     \left\|\frac{f(t,x)-f(x,x)}{t-x}\right\|_{-n}=\mathrm{sup}_{h\ge 0,s\ge 0}\mathrm{max}_{l}(|a_{h+1+l,s-l}|((h+1)^{n}(s+1)^{n}))\le 
     \mathrm{sup}_{h\ge 0,s\ge 0}\mathrm{max}_{l}(|a_{h+1+l,s-l}|((h+1+l+1)^{2n}(s-l+1)^{2n}))
     \le\mathrm{sup}_{i\ge 0,j\ge 0}|a_{i,j}|((i+1)(j+1))^{2n}=\|f\|_{-2n}.
    \end{dmath*}
     This is because $$(h+1)^n(s+1)^n=(h+1)^n(s-l+l+1)^n\le (h+1)^n(s-l+1)^n(l+1)^n\le(h+1+l+1)^{2n}(s-l+1)^{2n}$$ for each $h,s\in\N$, and each $l\le s$.
     \end{proof}

\begin{rmk} 
The results above give a chain of open subsets of $\mathbf A^1_K$. More precisely, the natural morphisms of $K[t]$-algebras \[ \tate{K}{t} \longrightarrow \power{K}{t}_0 \longrightarrow \temp{K}{t} \longrightarrow \lim_{\rho<1}\tate{K}{\rho^{-1}t} \] induce inclusions \[ \goth S\!\left(\lim_{\rho<1}\tate{K}{\rho^{-1}t}\right) \subset \goth S(\temp{K}{t}) \subset \goth S(\power{K}{t}_0) \subset \goth S(\tate{K}{t}) \subset \mathbf A^1_K . \] We interpret these respectively as the usual open unit disk, the tempered unit disk, the bounded unit disk, and the closed unit disk centered at $0$ (see also \cite[Theorem 5.8]{BassatMuk}).    
\end{rmk}

\section{Tempered decomposition of the affine line over a field}\label{section complementary}

In this section we consider the topological space \[ \mathbf A^1_K=\goth S(K[t]). \]   We will interpret $\fast{K}{t}$ as the algebra of functions supported on the complement of the tempered disk at infinity, in the sense of Remark \ref{functions on complementary}. We already know that $\fast{K}{t}$ determines an open subset of $\mathbf A^1_K$ by Lemma \ref{lem:fast_is_open}. We will show that this open subset and the punctured tempered disk at infinity cover $\mathbf A^1_K$, and that their intersection is non-empty. This shows that   $\mathbf A^1_K $ can be covered by two open subsets: one associated with the \emph{punctured tempered disk at infinity}, and the other associated with the algebra of fast-converging series at $0$.  

For $n \in \N$ and $m \in \Z$ consider 
$$ K^m_{n,\infty}[t]=\left\{\sum_{i\le m}a_it^i :a_i\in K, \ \lim_{i \to \infty} \frac{|a_{-i}|}{(i+1)^n}\rightarrow0\right\}.$$
This is clearly a Banach space over $K$ when endowed with the norm 
  $$\norm{\sum_{i\le m}a_it^i }_{K^m_{n,\infty}}= \max \{ \max_{0 \le i \le m} |a_i|,  \sup_{i\in\N} \left \{\frac{|a_{-i}|}{(i+1)^{n}} \right \} \}.$$
Equivalently, this is the space of Laurent series whose positive part has degree at most $m$, and whose negative tail has logarithmic growth at most $n$ in the variable $t^{-1}$.

We set 
$$ K^m_{\mathrm{temp},\infty}[t]= \colim_{n\in\mathbb N} K^m_{n,\infty}[t],$$
where the colimit is computed as bornological vector space, and similarly for $n \in \N$
$$ K_{n,\infty}[t] = \colim_{m\in\N} K_{n,\infty}^m[t], $$
and finally 
$$ K_{\mathrm{temp},\infty}[t] = \colim_{m\in\N} K_{\mathrm{temp},\infty}^m[t] = \colim_{n\in\N} K_{n,\infty}[t].$$ 
Thus 
$$ K^m_{\mathrm{temp},\infty}[t]=\left\{\sum_{i\le m}a_it^i : a_i\in K, \exists \ n\in\N  \ \lim_{i \to \infty} \frac{|a_{-i}|}{(i+1)^n} \rightarrow 0, \right\},$$
and 
$$ K_{\mathrm{temp},\infty}[t]=\left\{\sum_{i\le m}a_i t^i : a_i\in K, \ \lim_{i \to \infty} \frac{|a_{-i}|}{(i+1)^n} \rightarrow 0, \ \text{for some } m, n \in\N \right\}.$$
It is straightforward to check that $K_{\mathrm{temp},\infty}[t]$ is a bornological algebra over $K$.
We interpret it as the algebra of functions on the tempered disk at $\infty$ in $\mathbf{A}^1_K$. Strictly speaking, the objects defined above live in $\CBornK$. Via the dissection functor of Theorem \ref{compare ind and born}, we can regard them as objects of $\Ind^m(\BanK)$, and we keep the same notation for them.

\begin{rmk}
    The algebra $K_{\mathrm{temp},\infty}[t]$ defines an open subset of $\mathbf A^1_K$. Indeed, the canonical map $K[t,t^{-1}] \to  K_{\mathrm{temp},\infty}[t]$ is a homotopy epimorphism because it is obtained as the base change 
    $$
    \begin{tikzcd}
    K[t^{-1}] \arrow{r} \arrow{d}  &  \temp{K}{t^{-1}} \arrow{d} \\
    K[t, t^{-1}] \arrow{r}  &  K[t,t^{-1}]\wotimes_{K[t^{-1}]}^{\bb{L}}\temp{K}{t^{-1}}\cong K[t,t^{-1}]\wotimes_{K[t^{-1}]}\temp{K}{t^{-1}} \cong  K_{\mathrm{temp},\infty}[t]
    \end{tikzcd} 
    .$$
   Notice that $K[t, t^{-1}]$ is flat over $K[t]$ or over $K[t^{-1}]$. Therefore $K[t,t^{-1}] \to K_{\mathrm{temp},\infty}[t]$ is a homotopy epimorphism. Since $K[t]\to K[t,t^{-1}]$ is also a homotopy epimorphism, the claim follows by the stability by composition of homotopy epimorphisms. 
\end{rmk}
\medskip

We call the corresponding open subset \[ \goth S(K_{\mathrm{temp},\infty}[t])\subset \mathbf A^1_K \] the \emph{punctured tempered disk at infinity}. We now compute the algebra associated with the complement of this open subset in $\mathbf A^1_K$, in the sense of Remark \ref{functions on complementary}. We follow the argument of \cite[Proposition 5.6]{ClauSchCompl}.

\begin{pro} \label{prop:complement_of_tempered}
    The algebra of functions supported on the complement of the punctured tempered disk at infinity in $\mathbf A^1_K$ is naturally identified with
    $$\fast{K}{t}=\left\{\sum_{i \in \N} a_it^i:\lim_{i \to \infty} |a_i|(i + 1)^n\rightarrow 0, \forall n\in\N\right\}.$$
\end{pro}
\begin{proof}
    Following the discussion of \cite[Lecture V]{ClauSchCompl}, computing the algebra of functions supported on the complement of $\goth S(K_{\mathrm{temp},\infty}[t])$ in $\mathbf A^1_K$ amounts to computing
    \begin{equation}\label{complementary}
    \derintHom{K[t]}{[K[t] \rightarrow K_{\mathrm{temp},\infty} \left[t\right]]}{K[t]}.
    \end{equation} 
    The complex $[K[t] \rightarrow K_{\mathrm{temp},\infty}[t]]$ is quasi-isomorphic to 
    $$ \frac{K_{\mathrm{temp},\infty}[t]}{K[t]}[-1] \cong K_{\mathrm{temp}, \infty}^{-1}[t][-1] $$
    the module of tempered series with negative degrees, shifted by -1. Then \eqref{complementary} is identified to
    \begin{equation}\label{complementaryunion}
        \derintHom{K[t]}{K_{\mathrm{temp}, \infty}^{-1}[t]}{K[t]}[1]\cong \mathbb{R} \lim_{n\in\N}\derintHom{K[t]}{K_{n,\infty}^{-1}[t]}{K[t]}[1].
    \end{equation}
    We first compute the derived $\mathrm{Hom}$s as $K$-modules.
    Notice that $K_{n,\infty}^{-1}$ is projective in $\mathsf{Ind}(\mathsf{Mod}_{K})$, because it is isomorphic to a module of the form of equation \eqref{eq:log_growth_projective} in Section \ref{section tempered functions}.
     Moreover it is a compact object (because it is Banach over $K$). So, renaming $K[t]$ as $K[u]$ we have 
    \begin{equation}\label{hom over K}
    \underline{\mathrm{Hom}}_K(K_{n,\infty}^{-1}[t], K[u])[1] \cong \bigoplus_{l\in\N} \underline{\mathrm{Hom}}_K(K_{n,\infty}^{-1}[t],K)u^l[1].
    \end{equation}
    Since the dual of convergent sequences is given by bounded sequences, see \cite[Corollary 3.50]{BenKreFrechet}, we obtain that \eqref{hom over K} is isomorphic to 
    \begin{equation*}
        \bigoplus_{l\in\N}\power{K}{t}_{-n}u^l[1].
    \end{equation*}
    To prove the isomorphism 
    \begin{equation}\label{calculation mod (t-u)}
    \derintHom{K[t]}{K_{n,\infty}^{-1}[t]}{K[t]}[1] \cong \underline{\mathrm{Hom}}_K(K_{n,\infty}^{-1}[t],K[u])[1]/(t-u),
    \end{equation} 
    one takes the cone of the strict (by Corollary \ref{cor:idclosedintemp}) monomorphism
    \begin{equation*}
        \bigoplus_{l\in\N}\power{K}{t}_{-n}u^l[1]\xrightarrow{(t-u)}\bigoplus_{l\in\N}\power{K}{t}_{-n}u^l[1],
    \end{equation*}
    that is $\power{K}{t}_{-n}$.
    In fact $K[t]$ is clearly projective over $K$, and the projective objects over $K[t]$ are direct summands of base changes of projective objects over $K$. Thus projective objects over $K[t]$ are projective when considered over $K$ by Proposition \ref{tensor of banach projectives}.
    So, if $P^{\bullet}$ is a resolution of $K_{n,\infty}^{-1}[-1]$ by projective $K[t]$-modules, we can compute 
    $$\derintHom{K[t]}{K_{n,\infty}^{-1}[t]}{K[t]} \cong \underline{\mathrm{Hom}}_{K[t]}(P^{\bullet}, K[t])$$ 
    as 
    $$\underline{\mathrm{Hom}}_{K}(P^{\bullet},K[u])/(t-u) \cong \underline{\mathrm{Hom}}_{K} (K_{n,\infty}^{-1}[t], K[u])/(t-u)[-1] \cong \power{K}{t}_{-n},$$ 
    and we obtain \eqref{calculation mod (t-u)}. We are going to show in the next lemma that the  projective system that defines $\fast{K}{t}$ is a nuclear one. Therefore, by \cite{bambozzi2018stein} \S3.3, it follows that
    $$\lim_{n\in\N}\power{K}{t}_{-n}= {\mathbb R} \lim_{n\in\N}\power{K}{t}_{-n},$$
     concluding the proof.
    \end{proof}

\begin{lem} 
    The projective limit 
    $$
    \cdots \to \power{K}{t}_{-n- 1} \to \power{K}{t}_{-n} \to \power{K}{t}_{-n + 1 } \to \cdots
    $$
    is nuclear. Therefore, $\fast{K}{t}$ is a nuclear bornological space.
    \end{lem}
    \begin{proof} 
    It is enough to prove the statement for the cofinal projective system 
    \[ \cdots \longrightarrow \tate{K}{t}_{-(n+1)} \longrightarrow \tate{K}{t}_{-n} \longrightarrow \tate{K}{t}_{-(n-1)} \longrightarrow \cdots , \]
    where 
    \[ \tate{K}{t}_{-n} = \left\{ \sum_{i\geq 0}a_i t^i: \lim_{i\to\infty}\abs{a_i}_K(i+1)^n=0 \right\}. \]
    For each $n$, the Banach space $\tate{K}{t}_{-n}$ is of countable type. Hence, after replacing the norm by an equivalent orthogonal norm, it is linearly homeomorphic to $c_0$. More explicitly, we may choose a sequence $(b_{n,i})_{i\geq 0}$ in $K^\times$ such that \
    \[ D_n (i+1)^n \le \abs{b_{n,i}}_K \le C_n (i+1)^n \] 
    for some constant $D_n,C_n > 0$. In the discrete case  we may suppose $\abs{K}=p^{\Z}$,  if we denote by $\ell(i+1)$ the $p$-adic length of $i+1$, we may take  $b_{n,i}$ such that $ \abs{b_{n,i}}= p^{n \ell (i+1)}$. Thus, we have 
    \[ \abs{b_{n,i}}_K= p^{n \ell (i+1)} \le (i+1)^n \le p \abs{b_{n,i}}_K . \] 
    Then the map 
    \[ \tate{K}{t}_{-n}\longrightarrow c_0, \qquad \sum_{i\geq 0}a_i t^i\longmapsto (a_i b_{n,i})_{i\geq 0} \] 
    is a linear homeomorphism. Under the identifications 
    \[ \tate{K}{t}_{-(n+1)}\simeq c_0, \qquad \tate{K}{t}_{-n}\simeq c_0, \] 
    the transition map 
    \[ \tate{K}{t}_{-(n+1)} \longrightarrow \tate{K}{t}_{-n} \] 
    is represented by the diagonal operator 
    \[ c_0\longrightarrow c_0, \qquad (x_i)_{i\geq 0} \longmapsto (\lambda_i x_i)_{i\geq 0}, \] 
    where 
    \[ \lambda_i=\frac{b_{n,i}}{b_{n+1,i}}. \] 
    By construction 
    \[ \lim_{i \to \infty} \abs{\lambda_i}_K = 0. \]
    Therefore this diagonal operator is compactoid by \cite[Theorem 8.1.9(ii)]{Schikhof}. Thus, the transition maps are nuclear and hence the projective system is nuclear. Applying \cite[Lemma 3.60]{bambozzi2018stein} we get that $\fast{K}{t}$ is a nuclear bornological space.
\end{proof}

\begin{rmk} By a similar argument, one can also prove that all the natural maps in the inductive limit defining the tempered functions  $\tate{K}{t}_{n} \rightarrow \tate{K}{t}_{n+1}$  are nuclear. Thus, $\temp{K}{t}$ is a nuclear bornological space as well.
 \end{rmk}

\medskip

\begin{rmk} \label{r-decay}  
Along the same lines, other algebras can also be interpreted in this framework. For instance some algebras associated to the  log-decay series  introduced in \cite{KM1}, \cite{KM2}. We plan to investigate these algebras in forthcoming work \cite{logdecay&other}.
\end{rmk}

\bigskip

We now show that the intersection of the fast converging  disk in $0$ with the tempered disk at $\infty$ inside the affine line is non-empty.

\begin{pro}\label{nonemptintersection}
    The derived tensor product $K_{\mathrm{temp},\infty}[t]\hat{\otimes}^{\mathbb{L}}_{K[t]}\fast{K}{t}$ is naturally identified with the algebra
    \begin{equation*}
    \fast{K_{\mathrm{temp},\infty}}{t}=\left\{\sum_{i\in\Z}a_it^i:\lim_{i\to + \infty}|a_i|(i+1)^n\rightarrow 0\;\forall \;n,\;\exists m \in \N
    \lim_{i\to-\infty} \frac{|a_i|}{(-i+1)^m}\rightarrow 0 \right\}.
    \end{equation*}
    In particular $\goth{S}(\fast{K_{\mathrm{temp},\infty}}{t})$ is a not empty open subset of $\A_K^1$.
\end{pro}
\begin{proof}
    We have that $K_{\mathrm{temp},\infty}[t]$ is quasi-isomorphic to the flat (see Lemma \ref{tempered is flat}) resolution 
    \begin{equation*}[\temp{K[t]}{z}\xrightarrow{(tz-1)}\temp{K[t]}{z}].\end{equation*} 
    Indeed, the monomorphism induced by multiplication by $(tz-1)$ is strict by Corollary \ref{cor:idclosedintemp}. 
    It follows that
    \begin{equation*}
    K_{\mathrm{temp},\infty}[t]\hat{\otimes}^{\mathbb{L}}_{K[t]}\fast{K}{t}
    \end{equation*} 
    is computed by the complex $[\temp{\fast{K}{t}}{z}\xrightarrow{(tz-1)}\temp{\fast{K}{t}}{z}]$, and the claim follows applying the following lemma.
\end{proof}

\begin{lem}
    The ideal generated by $(tz-1)$ in $\temp{\fast{K}{t}}{z}$ is closed.
\end{lem}

\begin{proof}

We have 
    \begin{equation*}\temp{\fast{K}{t}}{z}\cong\colim_{n\in\N}\power{\fast{K}{t}}{z}_n 
    \cong \colim_{n\in\N} \lim_{m\in\N} \power{\power{K}{t}_{-m}}{z}_n
    \end{equation*} 
    where we applied \cite[Lemma 5.17]{BenKreFrechet} to commute the projective limit with the tensor product. Since $(tz-1)$ stabilizes $\power{\power{K}{t}_{-m}}{z}_n$, it is enough to prove that, for each $n$,  the module $(tz-1)\power{\power{K}{t}_{-m}}{z}_n$ is closed in $\power{\power{K}{t}_{-m}}{z}_n$, for a cofinal system of $m$. For simplicity we can take the sequence of $m$'s given by $rn=m$, $r \in \N$, $r>1$,  since 
    $$\temp{\fast{K}{t}}{z} \cong \colim_{n\in\N} \lim_{{ m\in\N} } \power{\power{K}{t}_{-m}}{z}_n.$$
   
    To prove this, we perform a calculation similar to the one done in the proof of Lemma \ref{idclosedintemp}. Let $h'=(zt-1)h$ in $\power{\power{K}{t}_{-m}}{z}_n$, with $h=\sum_{i,j\in\N}a_{i,j}t^jz^i$, and with $a_{i,j}\in K$. 
    We remark that on $\power{\power{K}{t}_{-m}}{z}_n$ we have the norm
    $$ \norm{h}_{-m,n}=\norm{\sum_{i,j\in\N}a_{i,j}t^jz^i}_{-m,n}=\sup_{i,j\in\N}|a_{i,j}|(j+1)^m(i+1)^{-n}.$$
    It is enough to prove that $\norm{h}_{-m,n}\le\norm{h(tz-1)}_{-m,n}= C$. Write 
    \[ (tz-1)h = \sum_{i,j\geq 0} b_{i,j}t^jz^i, \qquad b_{i,j}=a_{i-1,j-1}-a_{i,j}, \] 
    where $a_{i,j}=0$ if either index is negative.   Given $(i,j)$,  we want now to consider the behaviour of  $(j+l+1)^m(i+l+1)^{-n}$ with respect to $(j+(l+1)+1)^m(i+(l+1)+1)^{-n}$, where $l \in N$. By a direct calculation ($m=rn$ is our hypothesys)  we see that if $j \leq ri+(r-1)$, then for each  $l \in N$ we do have  and increasing
    $$(j+l+1)^{rn}(i+l+1)^{-n} < (j+(l+1)+1)^{rn}(i+(l+1)+1)^{-n}. 
    $$
    In case $j > ri+(r-1)$, then we have that the function is decreasing:
    $$(j+l+1)^{rn}(i+l+1)^{-n} > (j+(l+1)+1)^{2n}(i+(l+1)+1)^{-n}. 
    $$
    for $l=0,1,  \dots, \lfloor {j-ri-(r-1) \over {r-1}}\rfloor-2$ (where by $\lfloor - \rfloor$ we indicate the integer given by the  floor function). Let's denote  by $\alpha= \lfloor {j-ri-(r-1) \over {r-1}}\rfloor-1 $.   In fact once we are in $l=\alpha $, when we consider  $\alpha+1$ we have $(i+\alpha +1, j+\alpha +1)$: and for this couple we are in the case where the $j$-coordinates is strictly bigger than   $r$ times the $i$ coordinates plus $r-1$, hence   from that step  on  (i.e for $ l \geq \alpha +1 $) the value is increasing, going to $\infty$. 
    
    Let's start the proof. Because for negative index we have $a_{i,j}=0$, we can see that the terms $a_{i,0}$ and $a_{0,j}$ satisfies our 
    $$
    |a_{i,j}|(j+1)^{rn}(i+1)^{-n} \leq C.$$

    Note that  in $(i,0)$ and $(0,1),(0,2), \dots , (0,r-1)$,we are in the hypotheses that $j\leq ri -(r-1)$, while for $(0,j)$, $j\geq r$, not. Consider now  that we know the results  for $a_{i,j}$
 Suppose first that $j \leq ri+(r-1)$  We know that $\abs{a_{i,j}-a_{i+1,j+1}}(j+1)^{rn}(i+1)^{-n}$ is bounded by $C$. If  $\abs{a_{i,j}}\neq\abs{a_{i+1,j+1}}$ this implies that $\abs{a_{i+1,j+1}-a_{i,j}}=\mathrm{max}(\abs{a_{i+1,j+1}},\abs{a_{i,j}})$, so $|a_{i+1,j+i}|(j+2)^{rn}(i+2)^{-n}$ is bounded by $C$. 
    Now suppose that $\abs{a_{i+1,j+1}}=\abs{a_{i,j}}$. From the fact that   $\abs{a_{i+l,j+l}} \to 0$ (as we said $(j+l+1)^{rn}(i+l+1)^{-n}$ is monotone increasing   to $\infty$   and  $h=\sum_{i,j\in\N}a_{i,j}t^jz^i \in \power{\power{K}{t}_{-m}}{z}_n$) for a $l>1$ we may  have 
    $$\abs{a_{i,j}}= \abs{a_{i+1,j+1}}= \abs{a_{i+2,j+2}}= \dots = \abs{a_{i+l-1,j+l-1}}
    \neq \abs{a_{i+l,j+l}}$$
    But then as before we will have   $\abs{a_{i+l-1,j+l-1}-a_{i+l,j+l}}(j+l+1)^{rn}(i+l+1)^{-n} \leq C$,  then   $ \abs{a_{i+l-1,j+l-1}-a_{i+l,j+l}}=\mathrm{max}(\abs{a_{i+l-1,j+l-1}},\abs{a_{i+l,j+l}})$, so $|a_{i+l-1,j+l-1}|(j+l+1)^{rn}(i+l+1)^{-n}\leq C$. From the fact  that $(j+l+1)^{rn}(i+l+1)^{-n}$ is increasing on $l$, we can conclude for our ${a_{i+1,j+1}}$.

    Suppose now we know the result for  $a_{j,i}$ but with  $j > ri+(r-1)$, we have
    
    $$
    |a_{i,j}|(j+1)^{rn}(i+1)^{-n} \leq C.$$

     In this case we can define $\alpha$ as before. We want to show that $
    |a_{i+1,j+1}|(j+2)^{rn}(i+2)^{-n} \leq C.$ If $(j+1)> r(i+1)+ (r-1)$  we have that  $(j+1)^{rn}(i+1)^{-n} \geq  (j+2)^{rn}(i+2)^{-n}$: hence we conclude easily both if  $\abs{a_{i,j}}=\abs{a_{i+1,j+1}}$ or if  $\abs{a_{i,j}}\neq \abs{a_{i+1,j+1}}$.  The problem arises if  $(j+1)\leq r(i+1)+ (r-1)$: between $(i,j)$ and $(i+1,j+1) $  the function moves from decreasing to increasing.  And then we  have two possibilities. If $(j+1)^{rn}(i+1)^{-n}> (j+2)^{rn}(i+2)^{-n}$ we conclude directly as in the decreasing case.  If   $(j+1)^{rn}(i+1)^{-n}< (j+2)^{rn}(i+2)^{-n}$, and  $\abs{a_{i,j}}\neq \abs{a_{i+1,j+1}}$ we apply the fact that $ \abs{a_{i,j}-a_{i+1,j+1}}=\mathrm{max}(\abs{a_{i,j}},\abs{a_{i+1,j+1}})$ , if $\abs{a_{i,j}}= \abs{a_{i+1,j+1}}$ we have then to  consider  $a_{i+1,j+1}$ and   $a_{i+2,j+2}$: but now we are in the increasing setting: we apply the   method of the first part of the proof:
    $$
    |a_{i+1,j+1}|(j+2)^{rn}(i+2)^{-n} \leq C.$$
    
    \end{proof}
     
\bigskip

\begin{rmk}
    Proposition \ref{nonemptintersection} implies in particular that the in the recollement diagram  
    \begin{equation*}
      \begin{tikzcd}
          \dercat_{\infty}(\sf{Mod}_{K_{\mathrm{temp},\infty}[t]})\arrow[r, "i"] & \dercat_{\infty}( \sf{Mod}_{K[t]})\arrow[r, "j"]\arrow[l, bend right=50, "i^r"]\arrow[l, bend left=50, "i^l"] & \dercat_{\infty}( \sf{Mod}_{K[t]})/\dercat_{\infty}(\sf{Mod}_{K_{\mathrm{temp},\infty}[t]}),\arrow[l, bend left=50, "j_l"]\arrow[l, bend right=50, "j_r"]
      \end{tikzcd}
    \end{equation*}
     constructed as in Remark \ref{functions on complementary}, 
   we cannot replace  
    $$
    \dercat_{\infty}( \sf{Mod}_{K[t]})/\dercat_{\infty}(\sf{Mod}_{K_{\mathrm{temp},\infty}[t]})$$
    
     with $ \dercat_{\infty}(\sf{Mod}_{ \fast{K}{t} })$. In particular, the latter category is strictly bigger than the former: this is because both are full subcategories of $\dercat_{\infty}( \sf{Mod}_{K[t]})$.  Thus $ \dercat_{\infty}(\sf{Mod}_{ \fast{K}{t} })$ should be thought of as the category of modules on an open neighbourhood of the complement of $\goth S(K_{\mathrm{temp},\infty}[t])$ in $\A_K^1$, not as the categorical quotient itself. Moreover, the fact that $\fast{K}{t}$ defines an open subset of $\mathbf A^1_K$ is an additional property. In general, the algebra 
     \[ \mathbf R\underline{\mathrm{Hom}}_{K[t]} \left( [K[t]\to B], K[t] \right) \] 
     associated with the complement of an open subset need not itself define a homotopy epimorphism from $K[t]$. We will discuss this phenomenon further in forthcoming work \cite{logdecay&other}.
\end{rmk}
\medskip

It remains to check that the family of maps $\{ K[t] \to K_{\mathrm{temp},\infty}[t], K[t] \to \fast{K}{t}] \}$ is a cover for the homotopy Zariski topology. This is done in the next proposition by a direct computation:

\begin{pro} \label{prop:cover}
There is a Mayer--Vietoris type fiber sequence in $\dercat_{\infty}(\sf{Mod}_{ K[t] })$:
\[ K[t] \to K_{\mathrm{temp},\infty}[t] \oplus \fast{K}{t} \to \fast{K_{\mathrm{temp},\infty}}{t}. \]
\end{pro}
\begin{proof}
The  surjective restriction map $K_{\mathrm{temp},\infty}[t] \oplus \fast{K}{t} \to \fast{K_{\mathrm{temp},\infty}}{t}$  has kernel $K[t]$.\end{proof}

\begin{rmk} \label{dag} In the derived setting, along the previous lines,  one can prove  (essentially by repeating the computations done in \cite[Proposition 5.6]{ClauSchCompl} in the complex analytic case) that the complement of the closed dagger analytic unit disk at infinity is given by the open unit disk at $0$. In terms of algebras, we get a fiber sequence in $\dercat_{\infty}(\sf{Mod}_{ K[t]} )$
\[ K[t] \to K \{\!\{t^{-1}\}\!\} \oplus \tate{K}{t}^\dag \to K \{\!\{t^{-1}\}\!\} \wotimes_{K[t]}^{\mathbb{L}}\tate{K}{t}^\dag,  \]
analogous to that of Proposition \ref{prop:cover}. Here $\tate{K}{t}^\dag$ is the weak completion of the polynomial algebra $K[t]$ i.e. the free  dagger algebra, and $K \{\!\{t^{-1}\}\!\}$ is the algebra of analytic functions on the open unit disk centered at $\infty$. 
\end{rmk}

\section{Geometric interpretation of log-growth transfer theorems for $p$-adic differential equations}\label{section transfer} 

This section is mainly expository. Its purpose is to explain how a classical transfer phenomenon for $p$-adic differential equations, traditionally formulated in analytic terms, becomes a simple geometric statement in the derived analytic setting developed above. More precisely, we show that the log-growth transfer theorem can be interpreted as a continuity statement for the spectrum $\goth S(A)$ associated with a suitable affinoid algebra $A$.

We illustrate this point by revisiting the transfer theorem for solutions of non-archimedean differential equations with logarithmic growth. In general, the term \emph{Transfer Theorem} refers to a result which transfers a property from the ``generic point'' to a specified point. A classical example concerns the radius of convergence of a $p$-adic differential equation. In the framework of Berkovich spaces, this phenomenon has been interpreted as a continuity theorem (see \cite{Balda,Pulita}). There is, however, another kind of transfer theorem, involving not only the radius of convergence at the generic point, but also the logarithmic growth of solutions within their domain of convergence. This type of transfer theorem was introduced by Dwork--Robba \cite{DworkRob} (see also Christol \cite[Section 5]{Christ}). It was later studied by Chiarellotto--Tsuzuki \cite{Chiarellog,ChTs2} and in subsequent works (see for example \cite{tsu23} and \cite{Oh}).

Our aim is to show that, from the perspective of derived analytic geometry, the Transfer Theorem for logarithmic growth conditions can also be interpreted as a continuity theorem. Indeed, series with logarithmic growth, namely tempered series, can be viewed as analytic functions on a suitable open subset of the affine line. Since continuous maps preserve open subsets under pullback, the transfer theorem for log-growth conditions becomes a natural consequence of continuity.

\medskip

In this section we assume  that $K$ is a complete discrete valuation field. We now consider $\tate{K}{t}$ as a differential ring endowed with the usual differential operator $\diff{t}$ and a  differential equation of the type 
\begin{equation}\label{diffequation}
    \diff{t}\mathbf{y}=G\mathbf{y},
\end{equation}
where $G$ is a matrix in $\mathrm{Mat}_{m\times m}(\tate{K}{t})$ and $\mathbf{y}$ may be thought as  a  formal solution  in $(\power{K}{t})^m$. In this section with the notation  $\mathcal{E}$ we  denote  the so called \emph{Amice ring} i.e.
$$\mathcal{E}= \left \{\sum_{i\in\Z}a_it^i: a_i \in { K},  |a_i|\xrightarrow{i\rightarrow -\infty} 0,\, |a_i|\,\mathrm{bounded} \right \} .
$$
It  is a complete discrete valuation field.

We have a natural inclusion $\tate{K}{t}\rightarrow  \mathcal{E}$. Denote by $\tau: \tate{K}{t}\rightarrow \power{\mathcal{E}}{w}_0$ (where $w$ is a formal variable) the function \say{development at the generic point} of \cite[Proposition 2.5.1]{Christ}: 
$$\tau(f)=\sum_{i} \left ( \left ( \diff{t} \right )^i \frac{f}{i!} \right )w^i. $$
Here  $f\in \tate{K}{t}$, and  $((\diff{t})^if)$  is thought as an element of  $\mathcal{E}$ (see also \cite[Section 3.2]{Chiarellog}). Note that $\tau$ is a morphism of differential rings, if we endow $\power{\mathcal{E}}{w}_0$ with the derivation $\diff{w}$. Since the map $\tau$ factors trough $\power{\tate{K}{t}}{w}_0$, it sits in the commutative diagram
\begin{center}
\begin{equation}\label{trianglalgebras}
    \begin{tikzpicture}
  \matrix (m) [matrix of math nodes,row sep=2em,column sep=4em,minimum width=2em]
  {
 \, & \power{\tate{K}{t}}{w}_0& \,\\
\, &\, &\, \\
\tate{K}{t} & \, & \power{K}{w}_0.  \\};
  \path[-stealth]
    (m-3-1) edge node [left] {$\tau$} (m-1-2)
    (m-1-2) edge node [right] {$t=0$} (m-3-3)
    (m-3-1) edge node [above] {$t\mapsto w$} (m-3-3);
\end{tikzpicture}
\end{equation}
\end{center}
Where by $t=0$ we simply indicate the projection taking $t=0$. All the rings in the previous diagram are Banach, hence we can consider the associated spectra in $\Ind(\BanK)$ as constructed in Section \ref{section spectrum}. Then, in the derived analytic setting , we obtain a diagram of continuous maps
\begin{center}
\begin{equation*}\label{trianglalgebras}
    \begin{tikzpicture}
  \matrix (m) [matrix of math nodes,row sep=2em,column sep=4em,minimum width=2em]
  {
 \, & \goth{S}(\power{\tate{K}{t}}{w}_0)& \,\\
\, &\, &\, \\
\goth{S}(\tate{K}{t}) & \, & \goth{S}(\power{K}{w}_0),  \\};
  \path[-stealth]
    (m-1-2) edge node [left] {$\tau$} (m-3-1)
    (m-3-3) edge node [right] {$t=0$} (m-1-2)
    (m-3-3) edge node [above] {$t\mapsto w$} (m-3-1);
\end{tikzpicture}
\end{equation*}
\end{center}
that is the restriction (i.e. is induced by base change at the levels of algebras) of the diagram on affine spaces: 
\begin{center}
\begin{equation*}\label{trianglalgebras}
    \begin{tikzpicture}
  \matrix (m) [matrix of math nodes,row sep=2em,column sep=4em,minimum width=2em]
  {
 \, & \mathbf{A}^2_K& \,\\
\, &\, &\, \\
\mathbf{A}^1_K & \, & \mathbf{A}^1_K,  \\};
  \path[-stealth]
    (m-1-2) edge node [left] {$\tau$} (m-3-1)
    (m-3-3) edge node [right] {$i_2$} (m-1-2)
    (m-3-3) edge node [above] {$\mathrm{ex}$} (m-3-1);
\end{tikzpicture}
\end{equation*}
\end{center}

 where $i_2$ is the inclusion of the $w$-axis and $\mathrm{ex}$ is the exchange of the variables. 
We know that by the formal Cauchy Theorem (see for example \cite[Section 6]{Chiarellotto}) the differential system \eqref{diffequation}  admits a full set of formal solutions in $\power{K}{t}$,
given by \begin{equation*}\mathcal{Y}=\sum_{m\in\N}G_{[m]}t^m,\end{equation*} where $G_{[0]}=\mathrm{Id}$, $G_{[1]}=G$, and $G_{[m]}$ is defined recursively as \begin{equation*}{1\over m}(\diff{t}G_{[m-1]}+GG_{[m-1]}).\end{equation*}
By applying $\tau$  to the  differential system \eqref{diffequation} we obtain 
\begin{equation}\label{transfequat}
    \diff{w}\mathbf{y}=\tau(G)\mathbf{y}.
\end{equation}

In our derived analytic setting, if we know that the formal solutions of equation \eqref{transfequat} are in $\temp{\tate{K}{t}}{w}$, then by continuity of the map  $i_2$ (i.e. $t=0$) we have that the formal solutions of \eqref{diffequation} live in their inverse image open, i.e. in $\temp{K}{w}$, whence in $\temp{K}{t}$. This recovers the content of the Transfer Theorem in this setting: when all the solutions have log growth (and in particular radius one) at the generic point, then they share the same properties (log growth and radius one) at the origin. We have thus proved the following thereom.

\begin{thm}(Tempered Transfer \cite[Section 4]{Christ})\label{thm tempered transfert}
    If the differential system \eqref{transfequat} admits a full set of solutions in $\temp{\tate{K}{t}}{w}$, then \eqref{diffequation} admits a full set of solutions in $\temp{K}{t}$.
\end{thm}

\medskip

It may be of  interest  to note that the Amice ring  determines an open subset in the affine space $\mathbf{A}^1_{K}$. Indeed, we have the following proposition.

\bigskip

\begin{pro}
    The Amice ring $\mathcal{E}$ is isomorphic to the derived tensor product $\power{K}{t}_0\dertens{K[t]}\tate{K}{\frac{1}{t}}[t]$. In particular,  $\goth{S}(\cal{E})$ is isomorphic to the derived intersection of two open subsets in $\mathbf{A}^1_{K}$.
\end{pro}
\begin{proof}
    We claim that $\tate{K}{\frac{1}{t}}[t]$ is quasi-isomorphic to the complex of projective objects (in $\sf{Mod}_{K[t]}$) 
    \begin{equation*}\tate{K[t]}{y}\xrightarrow{(ty-1)}\tate{K[t]}{y}.
    \end{equation*}
    This follows from the fact that the ideal $(ty-1)\tate{K[t]}{y}$ is closed in $\tate{K[t]}{y}$ by Corollary \ref{cor:idclosedintemp}.
    Using the computations of Lemma \ref{idclosedintemp}, it is easy to check that $(ty-1)\tate{\power{K}{t}_0}{y}$ is closed in $\tate{\power{K}{t}_0}{y}$ (notice that the Gauss norm is multiplicative). 
    Therefore, tensoring the projective resolution we get the complex
    \begin{equation*}
        \power{K}{t}_0\hat{\otimes}_{{K[t]}}[\tate{K[t]}{y}\xrightarrow{(ty-1)}\tate{K[t]}{y}]{\cong}[\tate{\power{K}{t}_0}{y}\xrightarrow{(ty-1)}\tate{\power{K}{t}_0}{y} ]{\cong}\frac{\tate{\power{K}{t}_0}{y}}{(ty-1)}{\cong}\mathcal{E}
    \end{equation*}
    computes $\power{K}{t}_0\dertens{K[t]}\tate{K}{\frac{1}{t}}[t]$.
\end{proof}

\section{Tempered tubes: towards a tempered convergent cohomology}\label{section tubes}

We recall that, in our notation, $k$ will be a perfect field that we will see as the residual field of a complete non-archimedean valuation ring,   $\cal V$ of characteristic 0, whose fraction field will be indicated by $K$.

In this section we introduce the notion of {\it tempered tube} of a $k$-scheme: this should be seen as a log-growth open tube containing the usual tube of the rigid setting (see \cite[Chapter 2]{LeStum}). The aim is to use such a tube to define a tempered analogue of the (rigid) convergent cohomology for $k$-schemes.
\smallskip

\begin{rmk}  
The guiding example is the scheme given by a point, the origin, of  $\mathbf{A}^1_k$. We may consider the formal scheme $\mathrm{Spf}(\tate{\cal V}{x})$, viewed as a lifting of $\A_k^1$. The generic fiber of $\mathrm{Spf}(\tate{\cal V}{x})$ is the Tate algebra $\tate{K}{x}$.
In the derived analytic spectrum associated with $\tate{K}{x}$, namely the closed unit disk, the usual tube of the origin is the open unit disk. In the framework developed above, there is also a larger open neighborhood, the tempered tube, whose sections of the structure sheaf are given by the tempered power series $\temp{K}{x}$.
\end{rmk}
\smallskip

We now introduce the local construction of the tempered tube. Let $A$ be a topologically finitely generated, flat $p$-adic formal $\mathcal V$-algebra, and let $A_K$ denote its associated affinoid $K$-algebra.
\smallskip

\begin{dfn}  \label{defn:tempered_tube}
Let $X$ be an affine algebraic variety over $k$, and suppose that $X$ is equipped with a closed immersion 
\[ X\hookrightarrow {\hat P}_k, \qquad {\hat P}=\operatorname{Spf}(A). 
\]
Suppose that the closed immersion $X \rightarrow {\hat P}_k$ is defined by the ideal $({ f}_1,\dots,{ f}_s) \subset A_k$. Consider liftings of the generators ${\tilde f}_1,\dots,{\tilde f}_s$ in $A$. We define \emph{the tempered tube} of $X$ in ${\hat P}$ to be the affine derived analytic space associated with the algebra
\begin{equation*}
\temp{A_{K}}{y_1,\dots,y_s}/(y_1-\tilde{f}_1, \dots , y_s-\tilde{f}_s ),
\end{equation*} 
where the quotient is computed in the category of complete bornological algebras. We will denote the tempered tube of $X$ in ${\hat P}$ by $\temptube{X}{\hat P}$.
\end{dfn}
\smallskip

 In the next sections we will prove, under suitable hypotheses, that the definition does not depend on the choice of the $f_i$'s and their liftings. We will also show that it determines an open subset of the topological space $\goth{S}(A_K)$. This latter space contains the usual opens of the classical analytic space associated with $\operatorname{Sp}(A_K)$, see \cite{BenKre}.
\smallskip

\begin{rmk}\label{the rigid tube}
The definition of the tempered tube is motivated by the analogy with the classical construction. Keeping the same notation, if we denote by $\conv{A_K}{y_1,\dots,y_s}$ the series converging on the unit open polydisk in $\mathbf{A}^s_{A_K}$, then the Fr\'echet algebra $\conv{A_K}{y_1,\dots,y_s}/(y_1-\tilde{f}_1,\dots, y_s-\tilde{f}_s)$ defines the classical rigid analytic tube of $X$ in ${\hat P}$, as defined in \cite[Chapter 2]{LeStum}, and denoted by $\tube{X}{P}$.
Equivalently, the classical tube can be  defined by the algebra 
\[ 
\lim_{m/l>0} \frac{\tate{A_K}{y_1,\dots,y_s}}{(\pi^m y_1-\tilde{f}_1^l,\dots,\pi^m y_s-\tilde{f}^l)},
\]
where $\pi\in\mathcal{V}, |\pi|<1$. But
\[
    \lim_{m/l>0} \frac{\tate{A_K}{y_1,\dots,y_s}}{(\pi^m y_1-\tilde{f}_1^l,\dots,\pi^m y_s-\tilde{f}^l)} \simeq 
    \lim_{m/l>0} \frac{\tate{A_K}{y_1,\dots,y_s,t_1,\dots,t_s}}{(\pi^m y_1-t_1^l,\dots,\pi^m y_s-t_s^l,t_1-\tilde{f}_1,\dots,t_s-\tilde{f}_s)} \simeq
\]
\[
    \frac{\conv{A_K}{y_1,\dots,y_s}}{(y_1-\tilde{f}_1,\dots,y_s-\tilde{f}_s)} = 
    \lim_{r<1} \frac{\tate{A_K}{r^{-1} y_1,\dots,r^{-1} y_s}}{( y_1-\tilde{f}_1,\dots,y_s-\tilde{f}_s^l)}
\]
 where $\tate{A_K}{r^{-1} y_1,\dots,r^{-1} y_s}$ is the algebra of power-series convergent in the polydisk of polyradius $(r, r, \cdots, r)$.
\end{rmk}
\medskip

 \subsection{The setting} 
We cannot prove the independence of the tempered tube from the embedding in full generality. We therefore restrict ourselves to the case of regular immersions in smooth formal schemes.
 
From now on let $K$ be discretely valued. Let $X$ be an affine algebraic variety over $k$, and let ${\hat P}=\mathrm{Spf}(A)$ be a smooth connected affine formal scheme over $\intv$ with smooth special fiber ${\hat P}_k = \rm{Spec}(A_k)$ such that there is a regular immersion $X\hookrightarrow {\hat P}_{k}$. 
 
We keep the notation of the previous subsection. In this situation, $\temptube{X}{P}$ is defined by elements $f_1,\dots,f_s$ that form a regular sequence in $A_k$, such that $X=V(f_1,\dots,f_s)$ and $\tilde{f}_1,\dots, \tilde{f}_s$ is regular sequence in $A$. The elements $\tilde{f}_1,\dots, \tilde{f}_s$ form a regular sequence because $A$ is supposed smooth.
We now study some properties of $\temptube{X}{{\hat P}}$ in this setting.
 \smallskip

\begin{lem}\label{multnorm}
    Under the previous hypotheses, the supremum norm of $A_K$ is multiplicative. Moreover, it coincides with the affinoid norm on $A_K$ (see Remark \ref{rmk:sup_norm}).
    \end{lem} 
   \begin{proof}
By \cite[Proposition 6.2.3.5]{BoschAna}, the supremum norm on $A_K$ is multiplicative if and only if the reduction $A_K^\circ/A_K^{\circ\circ}$ is an integral domain, where $A_K^\circ$ denotes the subring of power-bounded elements and $A_K^{\circ\circ}$ the ideal of topologically nilpotent elements.
 Since $A\otimes_{\mathcal{V}}k$ is reduced we have that $A$ is integrally closed in $A_K$ by \cite[Proposition 1.1]{Bosch1995}.
 But the power bounded elements $A_K^{\circ}$ form precisely the integral closure of $A$ in $A_K$ (see \cite[Theorem 3.1.17]{BoschLec}), and so $A_K^{\circ}=A$ since $A$ is smooth (hence normal). 
 It suffice to show that $A_K^{\circ \circ}=\mathfrak{m} A$, because $A\otimes_{\mathcal{V}}k$ is reduced, and thus $A_K^{\circ}/A_K^{\circ \circ}\cong A\otimes_{\mathcal{V}}k$, that is an integral domain. 
 The inclusion $A_K^{\circ \circ}  \supset \mathfrak{m}A$ is clear because all element of $\mathfrak{m}A$ are topologically nilpotent. The other inclusions follows from the fact that if $x\in A_{K}^{\circ \circ}$ and $\pi\in\mathfrak{m}$, 
then there exists an $m\in\N$ such that $x^m\in\pi A\subset\mathfrak{m}A$, because $x$ is topologically nilpotent. Thus the image of $x$ in $A/ \mathfrak{m} A = A_k$ is nilpotent. Since $A_k$ is reduced, this image is zero, and hence $x \in \mathfrak m A$.  
\end{proof}
\smallskip
From now on we will consider $A_K$ as a Banach ring endowed with the supremum norm.
\smallskip

\begin{rmk} \label{rmk:sup_norm}
  The values of the valuation on $K$ coincide with the values of the  norm on $A_K$.
    In fact, let  $\alpha:\tate{\mathcal{V}}{x_1,\dots,x_m}\rightarrow A$ be a presentation of $A$. Then \begin{equation*}\alpha(\tate{K}{x_1,\dots,x_m}^{\circ \circ})=\alpha(\tate{\mathfrak{m}}{x_1,\dots,x_m})=\mathfrak{m}A=A_{K}^{\circ \circ}\end{equation*}
    by the proof of Lemma \ref{multnorm}. This implies that the residue norm induced by $\alpha$ is equal to the supremum norm by \cite[Proposition 6.4.3.4]{BoschAna}. But the values of such residue norm are the same as the ones of the norm on $K$, since they are limits of the values of the Gauss norm on $\tate{K}{x_1,\dots,x_m}$, which has the same values as the ones on $K$ (see \cite[Proposition 6.1.1.2]{BoschAna}).
\end{rmk}
\bigskip

We remark that the multiplication by $y_i-\tilde{f}_i$  restricts to an $A_K$-linear  continuous endomorphism in  every 
$\power{A_{K}}{y_1,\dots,y_s}_n$. To simplify notation, let us define
\[ S_n = \power{A_K}{y_1,\dots,y_s}_n, \qquad S_{\mathrm{temp}} =\temp{A_K}{y_1,\dots,y_s}, \qquad S_{\mathrm{for}} =\power{A_K}{y_1,\dots,y_s}. \] 
For $1\leq r\leq s$, set 
\[ g_i =y_i-\widetilde f_i, \qquad I_r = (g_1,\dots,g_r). \]

\begin{lem}\label{lem:tempered_lifting} 
Let $1\leq r\leq s$, and let $n\in \N$. Suppose that $h\in S_n$ belongs to the ideal $I_r S_{\mathrm{for}}$. Then there exist 
\[ H_1,\dots,H_r\in S_n \] 
such that 
\[ h=\sum_{i=1}^r g_iH_i. \] 
Moreover, if $\norm{h}_n\leq 1$, the $H_i$ can be chosen so that 
\[ \max_i\norm{H_i}_n\leq 1. \] 
\end{lem}

\begin{proof} 
Let 
\[ h\in S_n\cap I_r S_{\mathrm{for}} \] 
and assume that $\norm{h}_n\leq 1$. Let us write 
\[ h=\sum_{i=1}^r g_iG_i \] 
with $G_i\in S_{\mathrm{for}}$. 
We will modify the tuple $G=(G_1,\dots,G_r)$ by the Koszul relations in order to obtain a new tuple with coefficients satisfying the $n$-log growth bound. Write 
\[ h=\sum_{J\in\mathbb N^s} h_Jy^J, \qquad J=(j_1,\dots,j_s), \qquad y^J=y_1^{j_1}\cdots y_s^{j_s}. \] 
We put 
\[ w_J:=(j_1+1)^n\cdots(j_s+1)^n. \] 
Thus $\norm{h}_n\leq 1$ means precisely that $\abs{h_J}\leq w_J$ for every $J$. We order $\N^s$ first by total degree and then lexicographically. We will construct, by induction on this order, a $r$-tuple $B=(B_1,\dots,B_r)$ such that
\[ h=\sum_{i=1}^r g_iB_i, \] 
and 
\begin{equation}\label{BOUND} \max_{1\leq i\leq r}\abs{b^i_J}\leq w_J. \end{equation}
 We  write 
\[ B_i=\sum_{J\in\N^s} b^i_J y^J. \] 
Let $e_i\in\N^s$ be the $i$-th standard basis vector, and use the convention $b^i_{J_0-e_i}=0$ if $j_i=0$. The coefficient of $y^{J_0}$ in $\sum_{i=1}^r g_iB_i$ is 
\[ \sum_{i=1}^r b^i_{J_0-e_i} - \sum_{i=1}^r \widetilde f_i b^i_{J_0}. \] 
 Assume that  all coefficients of multidegree strictly smaller than a fixed $J_0$ satisfy the  inequality \eqref{BOUND}: here we are also considering our convention about negative exponents. Since this coefficient is $h_J$, we have 
\[ h_J= \sum_{i=1}^r b^i_{J_0-e_i} - \sum_{i=1}^r \widetilde f_i b^i_{J_0}. \] 
By induction, each coefficient $b^i_{J_0-e_i}$ satisfies  
\[ \abs{b^i_{J-e_i}}\leq w_{J-e_i}\leq w_J. \] 
Since also $\abs{h_J}\leq w_J$, the non-archimedean triangle inequality gives 
\[ \left| \sum_{i=1}^r \widetilde f_i b^i_J \right| \leq w_J. \] 
If $\max_i\abs{b^i_J}\leq w_J$,  
there is nothing to do. Otherwise set 
\[ \rho_J = \max_i\abs{b^i_J}>w_J. \] 
By Remark \ref{rmk:sup_norm}, the values of the norm on $A_K$ are the same as the values of the norm on $K$. Hence we can choose $\lambda\in K$ with $\abs{\lambda}=\rho_J$. Then $\abs{b^i_J/\lambda}\leq 1$ for every $i$, and 
\[ \left| \sum_{i=1}^r \widetilde f_i\frac{b^i_J}{\lambda} \right| \leq \frac{w_J}{\rho_J} <1. \] 
Thus, after reduction modulo $A_K^{\circ\circ}$, we get a relation 
\[ \sum_{i=1}^r \overline{b^i_J/\lambda}\, f_i = 0 \] 
in $A_k$. Since $f_1,\dots,f_r$ is a regular sequence in $A_k$, the first Koszul homology vanishes. Therefore the relation 
\[ \left( \overline{b^1_J/\lambda}, \dots, \overline{b^r_J/\lambda} \right) \] 
lies in the image of the Koszul differential of degree $2$ for $f_1, \ldots, f_r$
\[ A_k^{\binom r 2}\longrightarrow A_k^r. \] 
Choose liftings of the corresponding elements of $A_k^{\binom r2}$ to elements of $A^{\binom r2}$. Denote the lifted tuple by $\widetilde\alpha$. 
Then, denoting by $\varphi: A^{\binom r 2} \to A^r$ the Koszul differential of degree $2$ for $\widetilde f_1, \ldots, \widetilde f_r$, we have 
\[ \left( b^1_J,\dots,b^r_J \right) = \lambda\,\phi(\widetilde\alpha)_J + (c_1,\dots,c_r), \] 
where $\abs{c_i}<\rho_J$ for every $i$. Here $\phi(\widetilde\alpha)_J$ denotes the contribution of $\phi(\widetilde\alpha y^J)$ to the coefficient of multidegree $J$. Now replace $B$ by 
\[ B-\phi(\lambda\widetilde\alpha\, y^J). \]
This does not change the value of \[ \sum_{i=1}^r g_iB_i, \] because the image of $\phi$ lies in the kernel of the Koszul differential of degree $1$
\[ (B_1,\dots,B_r)\mapsto \sum_{i=1}^r g_iB_i. \] 
Moreover, this operation does not change any coefficient of multidegree strictly smaller than $J$. After this replacement, the maximum norm of the coefficient of multidegree $J$ is strictly smaller than $\rho_J$. Since $K$ is discretely valued, repeating this procedure finitely many times gives a representative for which 
\[ \max_i\abs{b^i_J}\leq w_J. \] 
This completes the induction step. Applying the construction to all multidegrees $J\in\N^s$, we obtain a tuple $H=(H_1,\dots,H_r)$ such that 
\[ h=\sum_{i=1}^r g_iH_i \] 
and 
\[ \abs{h^i_J}\leq w_J \] 
for every coefficient of every $H_i$. Hence 
\[ \max_i\norm{H_i}_n\leq 1. \]  
\end{proof}

From the fact that $\tilde{f}_1,\dots, \tilde{f}_s$ is a regular sequence in $A_K$ it easily follows that also   the sequence $y_1-\tilde{f}_1,\dots, y_s-\tilde{f}_s$ is regular in $\power{A_{K}}{y_1,\dots,y_s}$.  Now we want to see that it is a regular sequence in  $\temp{A_{K}}{y_1,\dots,y_s}$ as well.

\begin{lem}\label{regular sequence in tempered} 
The sequence 
\[ y_1-\widetilde f_1,\dots,y_s-\widetilde f_s \] 
is regular in $\temp{A_K}{y_1,\dots,y_s}$. \end{lem} 
\begin{proof} 
The sequence 
\[ y_1-\widetilde f_1,\dots,y_s-\widetilde f_s \] 
is regular in $S_{\mathrm{for}}$. Indeed, after quotienting by $y_1-\widetilde f_1,\dots,y_{i-1}-\widetilde f_{i-1}$, the element $y_i-\widetilde f_i$ is still monic in the variable $y_i$, and hence is not a zero divisor. We now prove regularity in $S_{\mathrm{temp}}$. Let $I_{i-1}:=(y_1-\widetilde f_1,\dots,y_{i-1}-\widetilde f_{i-1}).$ By Lemma \ref{lem:tempered_lifting}, the natural map 
\[ S_{\mathrm{temp}}/I_{i-1}S_{\mathrm{temp}} \longrightarrow S_{\mathrm{for}}/I_{i-1}S_{\mathrm{for}} \]
is injective. Since $y_i-\widetilde f_i$ is not a zero divisor in the target, it is not a zero divisor in the source. 
This holds for every $i=1,\dots,s$, so the sequence is regular in $S_{\mathrm{temp}}$. 
\end{proof}

\begin{lem}\label{closedtube}
For every $n\in\N$, the map
\[
    S_n^s\longrightarrow S_n,
    \qquad
    (H_1,\dots,H_s)\mapsto \sum_i (y_i-\widetilde f_i)H_i
\]
is strict. In particular, the ideal
\[
    (y_1-\widetilde f_1,\dots,y_s-\widetilde f_s)S_n
\]
is closed in $S_n$.
\end{lem}

\begin{proof}
Let $h$ be in the image and suppose $\norm{h}_n\leq 1$. By Lemma
\ref{lem:tempered_lifting}, applied with $r=s$, there exist
$H_1,\dots,H_s\in S_n$ such that
\[
    h=\sum_i (y_i-\widetilde f_i)H_i
\]
and
\[
    \max_i\norm{H_i}_n\leq 1.
\]
This says precisely that the quotient norm on the coimage agrees with the induced norm
on the image. Hence the map is strict. Since $S_n$ is Banach, the image is closed.
\end{proof}

\medskip 

We can now prove the following proposition.

\begin{pro}\label{prop:koszul_strict_resolution}
The Koszul complex associated with
\[
    y_1-\widetilde f_1,\dots,y_s-\widetilde f_s
\]
is a strict resolution of
\[
    \frac{\temp{A_K}{y_1,\dots,y_s}}{(y_1-\widetilde f_1,\dots,y_s-\widetilde f_s)}.
\]
\end{pro}
\begin{proof} 
By Lemma \ref{regular sequence in tempered}, the sequence 
\[ y_1-\widetilde f_1,\dots,y_s-\widetilde f_s \] 
is regular in $\temp{A_K}{y_1,\dots,y_s}$. In particular, its restriction to each Banach level $S_n$ is algebraically exact in the sense needed here: the image of each Koszul differential is the kernel of the next one. Thus the only point is strictness. For the last differential 
\[ S_n^s\longrightarrow S_n, \qquad (H_1,\dots,H_s)\longmapsto \sum_i g_iH_i, \] strictness follows from Lemma \ref{closedtube}. For the previous differentials, their images are kernels of continuous maps, hence are closed Banach subspaces. By the open mapping theorem, every continuous surjection from a Banach space onto a closed image is strict. Therefore all Koszul differentials are strict epimorphisms onto their images, and the Koszul complex is strictly exact on each Banach level $S_n$. Passing to the filtered colimit over $n$ gives the strict resolution in $\temp{A_K}{y_1,\dots,y_s}$. 
\end{proof}

\begin{rmk}\label{rmk:tempered_quotient_levels} 
Lemma \ref{lem:tempered_lifting} implies that, for every $n\in\N$, 
\[ (y_1 - \widetilde f_1, \ldots, y_s - \widetilde f_s) S_n = (y_1 - \widetilde f_1, \ldots, y_s - \widetilde f_s) S_{\mathrm{temp}}\cap S_n. \]
Indeed, if $h\in S_n$ belongs to $(y_1 - \widetilde f_1, \ldots, y_s - \widetilde f_s) S_{\mathrm{temp}}$, then it belongs a fortiori to $(y_1 - \widetilde f_1, \ldots, y_s - \widetilde f_s) S_{\mathrm{for}}$, and Lemma \ref{lem:tempered_lifting} shows that $h\in (y_1 - \widetilde f_1, \ldots, y_s - \widetilde f_s) S_n$. Consequently, the transition maps \[ S_n/IS_n\longrightarrow S_{n+1}/IS_{n+1} \] are injective, and we have \[ \frac{S_{\mathrm{temp}}}{IS_{\mathrm{temp}}} \cong \colim_{n\in\N} \frac{S_n}{IS_n}. \] Equivalently, we may write \[ \frac{\temp{A_K}{y_1,\dots,y_s}} {(y_1-\widetilde f_1,\dots,y_s-\widetilde f_s)} \cong \bigcup_{n \in \N} \frac{\power{A_K}{y_1,\dots,y_s}_n} {(y_1-\widetilde f_1,\dots,y_s-\widetilde f_s)}. \] \end{rmk}

\bigskip 

As a by-product of the lifting lemma, we obtain the following comparison between the algebra associated with the tempered tube and the algebra associated with the usual rigid analytic tube. 

\begin{cor}\label{prop tempered tube is contained in rigid} 
There is a canonical injective morphism of $K$-algebras 
\[ \frac{\temp{A_K}{y_1,\dots,y_s}} {(y_1-\widetilde f_1,\dots,y_s-\widetilde f_s)} \longrightarrow \frac{\conv{A_K}{y_1,\dots,y_s}} {(y_1-\widetilde f_1,\dots,y_s-\widetilde f_s)}. \] 
\end{cor} 
\begin{proof} 
The natural inclusion 
\[ \temp{A_K}{y_1,\dots,y_s} \hookrightarrow \conv{A_K}{y_1,\dots,y_s} \] 
induces the displayed morphism on quotients. To prove that it is injective, it is enough to prove injectivity after composing with the natural inclusion 
\[ \conv{A_K}{y_1,\dots,y_s} \hookrightarrow \power{A_K}{y_1,\dots,y_s}. \] 
Thus it suffices to show that the morphism 
\[ \frac{\temp{A_K}{y_1,\dots,y_s}} {(y_1-\widetilde f_1,\dots,y_s-\widetilde f_s)} \longrightarrow \frac{\power{A_K}{y_1,\dots,y_s}} {(y_1-\widetilde f_1,\dots,y_s-\widetilde f_s)} \] 
is injective. This is a consequence of the equality 
\[ S_{\mathrm{temp}}\cap I S_{\mathrm{for}} = IS_{\mathrm{temp}}, \] which follows from Lemma \ref{lem:tempered_lifting}:  hence the corollary is proved. 
\end{proof}

We can now prove that the tempered tube defines an open subset of $\goth S(A_K)$. 

\begin{pro}\label{prop:tempered_tube_open} 
The canonical morphism 
\[ A_K \longrightarrow \frac{\temp{A_K}{y_1,\dots,y_s}} {(y_1-\widetilde f_1,\dots,y_s-\widetilde f_s)} \] 
is a homotopy epimorphism. Hence it defines an open subset of $\goth S(A_K)$. 
\end{pro}

\begin{proof}

    If 
    $$\frac{\temp{A_{K}}{y_1,\dots,y_s}}{(y_1-\tilde{f}_1,\dots, y_s-\tilde{f}_s)}=0$$ 
    this is clear. So we assume the contrary.
    We claim that
    \begin{equation}\label{derived intersection of tubes}
    \frac{\temp{A_{K}}{y_1,\dots,y_s}}{(y_1-\tilde{f_1},\dots, y_s-\tilde{f}_s)} \cong \frac{\temp{A_{K}}{y_1}}{(y_1-\tilde{f_1})} \dertens{A_K} \dots \dertens{A_K} \frac{\temp{A_{K}}{y_s}}{(y_s-\tilde{f_s})}.
    \end{equation} 
    In fact by Lemma \ref{closedtube} $\temp{A_{K}}{y_1}/(y_1-\tilde{f_1})$ admits the flat resolution (see Lemma \ref{tempered is flat})
    \begin{equation}
        0\rightarrow\temp{A_{K}}{y_1}\xrightarrow{(y_1-\tilde{f_1})}\temp{A_{K}}{y_1}\rightarrow 0.
    \end{equation}
    And so $\temp{A_{K}}{y_1}\dertens{A_K}\temp{A_{K}}{y_2,\dots,y_s}/(y_1-\tilde{f_1},\dots, y_s-\tilde{f}_s)$ is computed by
    \begin{dmath*}
         [0\rightarrow\temp{A_{K}}{y_1}\xrightarrow{(y_1-\tilde{f_1})}\temp{A_{K}}{y_1} \rightarrow 0]\hat{\otimes}_{A_K} \frac{\temp{A_{K}}{y_2,\dots,y_s}}{(y_2-\tilde{f}_2,\dots, y_s-\tilde{f}_s)} \simeq
        \frac{\temp{A_{K}}{y_1,\dots,y_s}}{(y_1-\tilde{f}_1,\dots, y_s-\tilde{f}_s)},
    \end{dmath*}
    where the isomorphism follows from Lemma \ref{closedtube} and Lemma \ref{regular sequence in tempered}.
    Thus \eqref{derived intersection of tubes} follows by induction.
    It suffices then to prove the lemma for $s=1$. Consider the homotopical pushout diagram
    \begin{equation}
    \begin{tikzcd}
      \tate{A_K}{y_1} \arrow[r] \arrow[d]
      & \temp{A_K}{y_1} \arrow[d] \\
       \tate{A_K}{y_1}/(y_1-f_1) \arrow[r, "\phi"]
      & \temp{A_K}{y_1}/(y_1-f_1).
    \end{tikzcd}
    \end{equation}
    Since homotopy pushouts preserve homotopy epimorphisms (see \cite[Proposition 3.4]{BassatMuk}) we have that $\phi$ is a homotopy epimorphism. We can conclude using the fact that the Weierstrass localization $A_K \to \tate{A_K}{y_1}/(y_1-f_1)$ is a homotopy epimorphism (see \cite[Theorem 5.16]{BenKre} and the lemmas before it) and the stability of homotopy epimorphisms by composition. 
\end{proof}

This open set will be indicated as   the \emph{tempered tube} of $X$ in ${\hat P}$, $\temptube{X}{{\hat P}}$. A priori, the tempered tube depends on the chosen liftings. $\widetilde f_1,\dots,\widetilde f_s$. We will see that, in our case, this is not so (Proposition \ref{prop tempered is independent of the presentation}). First some lemmas.

 \medskip
\begin{lem}\label{open_disk_is_lim_acyclic} Let $A_K$ be an affinoid $K$-algebra. Then the canonical morphism 
\[ \lim_{\rho<1} \tate{A_K}{\rho^{-1}y_1,\dots,\rho^{-1}y_s} \longrightarrow \mathbb R\!\lim_{\rho<1} \tate{A_K}{\rho^{-1}y_1,\dots,\rho^{-1}y_s} \]
is a quasi-isomorphism. 
\end{lem}
\begin{proof}
    We observe that it is well-known that the maps
    \[  \tate{K}{\rho^{-1} y_1, \dots, \rho^{-1} y_s} \to \tate{K}{\sigma^{-1} y_1, \dots, \sigma^{-1} y_s} \]
    are nuclear, for $\sigma < \rho$ (\cite[Theorem 8.1.9(ii)]{Schikhof}). Therefore, we have
    \[ \lim_{\rho < 1}  \tate{K}{\rho^{-1} y_1, \dots, \rho^{-1} y_s} \cong \mathbb{R} \lim_{\rho < 1}  \tate{K}{\rho^{-1} y_1, \dots, \rho^{-1} y_s} \]
    applying \cite[Corollary 3.80]{bambozzi2018stein}  (where we used that the cofinality of the limit is countable). 
    Now, for any $r < 1$ let us write
    \[ C(r) = \lim_{\rho < r} \tate{K}{\rho^{-1} y_1, \dots, \rho^{-1} y_s}.  \]
    Notice that again by  \cite[Corollary 3.80]{bambozzi2018stein}
    \[ \mathbb{R} \lim_{\rho < 1} C(\rho) \cong   \lim_{\rho < 1} C(\rho) \cong \lim_{\rho < 1}  \tate{K}{\rho^{-1} y_1, \dots, \rho^{-1} y_s} \]
    but now all $C(\rho)$ are nuclear Fr\'echet bornological spaces.
    Then, in $\dercat(\sf{Mod}_K)$ we have
    \[ A_K \wotimes_K^{\mathbb L} (\mathbb{R} \lim_{\rho < 1} C(\rho)) \cong A_K \wotimes_K ( \lim_{\rho < 1} C(\rho)) \cong \lim_{\rho < 1}  A_K \wotimes_K C(\rho) \]
    where for the last isomorphism we used \cite[Corollary 3.65]{bambozzi2018stein} and \cite[Theorem 3.50]{bambozzi2018stein} for the first isomorphism. We also have that
    \[  \lim_{\rho < 1}  A_K \wotimes_K C(\rho) \cong \mathbb R \lim_{\rho < 1}  A_K \wotimes_K C(\rho) \]
    because the latter object is isomorphic to the Roos complex
    \[  [ \prod_{\rho < 1} A_K \wotimes_K C(\rho) \to \prod_{\rho < 1} A_K \wotimes_K C(\rho) ] \cong [  A_K \wotimes_K \prod_{\rho < 1} C(\rho) \to  A_K \wotimes_K \prod_{\rho < 1} C(\rho) ] \]
    where again we used \cite[Corollary 3.65]{bambozzi2018stein} and that the cofinality of the projective system is countable.
\end{proof}

\begin{lem} \label{open_loc_is_lim_acyclic}
    Let $A_K$ be an affinoid algebra, $f_i \in A_K$, then the canonical morphism
    \[ \lim_{\rho < 1} \frac{\tate{A_K}{\rho^{-1} y_1, \dots, \rho^{-1} y_s}}{(y_1-f_1,\dots,y_s-f_s)} \to \mathbb{R}  \lim_{\rho < 1} \frac{\tate{A_K}{\rho^{-1} y_1, \dots, \rho^{-1} y_s}}{(y_1-f_1,\dots,y_s-f_s)} \]
    is a quasi-isomorphism.
\end{lem}
\begin{proof}
     Let's denote by $B_n = \tate{A_K}{r_n^{-1} y_1,\dots, r_n^{-1} y_s}$ for a sequence $r_n \to 1$ from below. 
     Then, for each $n \in \N$ we have the strictly exact sequence
     \[ 0 \to B_n \to B_n \to \frac{B_n}{(y_1 - f_1)} \to 0 \]
     where the first map is multiplication by $(y_1 - f_1)$. This leads to the long exact sequence
     \[ 0 \to \lim_{n \in \N} B_n \to \lim_{n \in \N} B_n \to \lim_{n \in \N} \frac{B_n}{(y_1 - f_1)} \to {\lim_{n \in \N}}^{(1)}  B_n \to {\lim_{n \in \N}}^{(1)} B_n \to {\lim_{n \in \N}}^{(1)} \frac{B_n}{(y_1 - f_1)} \to 0 \]
     where $\lim_{n \in \N}^{(1)}$ is the first derived functor with respect to the left-t structure of $\dercat(\sf{Mod}_K)$. By Lemma \ref{open_disk_is_lim_acyclic} we have that ${\lim_{n \in \N}}^{(1)} B_n \cong 0$, which implies ${\lim_{n \in \N}}^{(1)} \frac{B_n}{(y_1 - f_1)} \cong 0$ and therefore $\mathbb{R} \lim_{n \in \N} \frac{B_n}{(y_1 - f_1)} \cong \lim_{n \in \N} \frac{B_n}{(y_1 - f_1)}$.
     Since $(y_1 - f_1, \ldots, y_s - f_s)$ is a regular sequence we can iterate the same reasoning and conclude by induction. For example, if we denote $B'_n = \frac{B_n}{(y_1 - f_1)}$, we have the strictly exact sequence
     \[  0 \to B_n' \to B_n' \to \frac{B_n'}{(y_2 - f_2)} \to 0 \]
     where the first map is given by multiplication by $(y_2 - f_2)$. As before, this gives $\mathbb{R} \lim_{n \in \N} \frac{B_n'}{(y_2 - f_2)} \cong \lim_{n \in \N} \frac{B_n'}{(y_2 - f_2)}$. And so on.
\end{proof}

Of course the proof of Lemma \ref{open_loc_is_lim_acyclic} applies to any projective system of quotients by a regular sequence of elements. Then we have (we refer to Remark \ref{the rigid tube} for the contruction of the rigid tube $
    ]X[_{{\hat P}}$):

\begin{cor}\label{ open tempered tube in rigid}
    In the previous hypotheses, in $\goth{S}(A_K)$, we have an immersion of open sets
    $$
    ]X[_{{\hat P}} \subset \temptube{X}{{\hat P}}. 
    $$
\end{cor}
\begin{proof}
    We have only to prove that $
    ]X[_{{\hat P}}$ is open in $\goth{S}(A_K)$. 
    Let's denote by $B_n= \tate{A_K}{r_n^{-1} y_1,\dots,r_n^{-1}y_s}/( y_1-\tilde{f}_1,\dots,y_s-\tilde{f}_s) $ and let $B = \lim_{n \in \N} B_n \cong \mathbb{R} \lim_{n \in \N} B_n$, applying Lemma \ref{open_loc_is_lim_acyclic}. 
    We then follow  the computations of \cite[Section 5]{bambozzi2018stein}. Consider the bar resolution ${\mathcal L}^{\bullet}_A(B)$:
    $$
    B {\hat \otimes}_{A_K}^{\mathbb L}B =  {\goth L}^{\bullet}_A(B) {\hat \otimes}_{A_K}B.$$
    See \cite[Definition 2.10]{bambozzi2018stein} for the definition of the bar resolution of an object in closed symmetric quasi-abelian category with enough projectives.
    For each $m$, we have 
    $${\goth L}^{m}_{A_K}(B)= A_K {\hat \otimes}{_K}( A {\hat \otimes}{_K} \dots {\hat \otimes}{_K} A {\hat \otimes}{_K} B)
    $$
    where $A_K$ is tensored $m$ times, and the structure of $A_K$-module is given by the first $A_K$ term on the left. It follows that
    $$
    {\goth L}^{m}_{A_K}(B) {\hat \otimes}_{A_K} B=  A_K {\hat \otimes}{_K} \dots {\hat \otimes}{_K} A_K {\hat \otimes}{_K} B {\hat \otimes}{_K} B.
     $$
    Let us write $C_n = \tate{A_K}{r_n^{-1} y_1,\dots,r_n^{-1}y_s}$ and $D_m = A_K {\hat \otimes}{_K} \dots {\hat \otimes}{_K} A_K {\hat \otimes}{_K} B$, so that
    \[ {\goth L}^{m}_{A_K}(B) {\hat \otimes}_{A_K} B =  D_m \wotimes_{K} \lim_{n \in \N}  \frac{C_n}{(y_1 - \tilde{f}_1,\dots,y_s-\tilde{f}_s)}. \]
    Let us consider the strictly exact sequence
    \[ 0 \to \lim_{n \in \N} C_n \to \lim_{n \in \N} C_n \to \lim_{n \in \N} \frac{C_n}{(y_1 - f_1)} \to 0 \]
    where the first map is multiplication by $(y_1 - f_1)$. Applying Lemma \ref{open_loc_is_lim_acyclic} we deduce that
    \[ \lim_{n \in \N} \frac{C_n}{(y_1 - f_1)} \cong \frac{\lim_{n \in \N} C_n}{ (y_1 - f_1) \lim_{n \in \N} C_n} \]
    and by the same argument of the proof of  Lemma \ref{open_loc_is_lim_acyclic} we get that
    \[ D_m \wotimes_K \lim_{n \in \N} C_n \cong D_m \wotimes_K \lim_{n \in \N} A_K \wotimes_K \tate{K}{r_n^{-1} y_1, \ldots, r_n^{-1} y_s} \cong \] \[ D_m \wotimes_K A_K \wotimes_K \lim_{n \in \N} \tate{K}{r_n^{-1} y_1, \ldots, r_n^{-1} y_s} \cong \lim_{n \in \N} D_m \wotimes_K A_K \wotimes_K  \tate{K}{r_n^{-1} y_1, \ldots, r_n^{-1} y_s}. \]
    Putting all together we get
    \[ D_m \wotimes_K \lim_{n \in \N} \frac{C_n}{(y_1 - f_1)} \cong \lim_{n \in \N} D_m \wotimes_K \frac{C_n}{(y_1 - f_1)}. \]
    Since $(y_1 - f_1, \ldots, y_s - f_s)$ is a regular sequence, we can apply the same reasoning iteratively and obtain
    \[  {\goth L}^{m}_{A_K}(B) {\hat \otimes}_{A_K} B \cong \rm{lim}_{n \in \N} D_m \wotimes_K B_n = \rm{lim}_{n \in \N}( A_K {\hat \otimes}{_K} \dots {\hat \otimes}{_K} A_K {\hat \otimes}{_K} B_n {\hat \otimes}{_K} B_n) \cong \] 
    \[ \lim_{n \in \N} {\goth L}^{m}_A(B_n) {\hat \otimes}_{A_K} B_n. \]
    Then, for each $n \in {\mathbb N}$ 
     $$
      {\goth L}^{\bullet }_{A_K}(B_n) {\hat \otimes}_{A_K} B_n \cong  B_n {\hat \otimes}_{A_K}^{\mathbb L} B_n  \cong B_n
     $$ 
   because the map  $ A_K \rightarrow B_n$ is a homotopy epimorphism. 
   Therefore we have    
   \begin{dmath*}
   B {\hat \otimes}_{A_K}^{\mathbb L}B \cong 
   \lim_{n \in \N} {\goth L}^{\bullet }_A(B_n) {\hat \otimes}_{A_K} B_n \cong
   {\mathbb R} \lim_{ n \in \N} B_n {\hat \otimes}_{A_K}^{\mathbb L} B_n \cong
   {\mathbb R} \lim_{n \in \N} B_n \cong \lim_{n \in \N} B_n = B
   \end{dmath*}
   where the acyclicity of the limit follows from Lemma \ref{open_loc_is_lim_acyclic}.
\end{proof}

\begin{lem}\label{lem:tempered_substitution}
Let $R$ be a Banach $K$-algebra, and let 
\[ P_i(z_1,\dots,z_l) = \sum_{q=1}^l h_{iq}z_q+\alpha_i, \qquad i=1,\dots,s, \] 
with $h_{iq}\in R^\circ$ and $\alpha_i\in R^{\circ\circ}$. Then substitution
\[ y_i \mapsto P_i(z_1,\dots,z_l) \] 
defines a morphism 
\[ \temp{R}{y_1,\dots,y_s} \longrightarrow \temp{R}{z_1,\dots,z_l}. \] 
More precisely, for every $n$ there exists $N \geq n$ such that the substitution map sends 
\[ \power{R}{y_1,\dots,y_s}_n \] 
boundedly into 
\[ \power{R}{z_1,\dots,z_l}_N. \] 
\end{lem} \begin{proof} 
Let 
\[ F=\sum_{J \in \N^s}a_Jy^J \in \power{R}{y_1,\dots,y_s}_n. \] 
Thus there exists $C>0$ such that \[ |a_J| \leq C\prod_{i=1}^s(j_i+1)^n \] 
for all $J=(j_1,\dots,j_s) \in \N^s$. Write $\alpha=\max_i|\alpha_i|<1$. After substituting $y_i=P_i(z)$, each coefficient of the resulting series in the variables $z_1,\dots,z_l$ is a convergent sum of terms  (up to binomial factors, which can only lower the norm) of the form 
\[ a_J \prod_{i,q} h_{iq}^{m_{iq}} \prod_i \alpha_i^{e_i}, \] 
where 
\[ j_i=e_i+\sum_{q=1}^l m_{iq}. \] 
Since $h_{iq}\in R^\circ$ and the multinomial coefficients have norm at most $1$, each such term has norm bounded by 
\[ |a_J|\alpha^{e_1+\cdots+e_s}. \] 
Let 
\[ M_q = \sum_{i=1}^s m_{iq}, \qquad |M| = M_1+\cdots+M_l, \qquad |J| = j_1+\cdots+j_s. \] 
Then 
\[ |J|=|M|+e_1+\cdots+e_s. \] 
Moreover, 
\[ \prod_{i=1}^s(j_i+1)^n \leq (|J|+1)^{ns}, \] 
while 
\[ \prod_{q=1}^l(M_q+1) \geq |M|+1. \]
Choose $N=ns$. Then 
\[ \frac{ \alpha^{|J|-|M|} \prod_i(j_i+1)^n }{ \prod_q(M_q+1)^N } \leq \alpha^{|J|-|M|} \left(\frac{|J|+1}{|M|+1}\right)^{ns}. \] 
The right-hand side is bounded uniformly in $J$ and $M$, because 
\[ |J|=|M|+e \] 
where $e = e_1 + \cdots + e_s$, and the function 
\[ e\mapsto \alpha^e(e+1)^{ns} \] 
is bounded on $\N$. Hence the coefficients of the substituted series satisfy a tempered estimate of level $N$. This proves the claim. 
\end{proof}

We can finally prove that, in our setting, the tempered tube is independent of the chosen presentation of the regular immersion. 

\begin{pro}\label{prop tempered is independent of the presentation} 
The tempered tube $\temptube{X}{{\hat P}}$ is independent of the choice of regular sequence defining $X$ in ${\hat P}_k$ and of the chosen liftings. 
\end{pro} 
\begin{proof} 
Suppose that the ideal of $X$ in $A_k$ is generated by two regular sequences 
\[ f_1,\dots,f_s \qquad\text{and}\qquad g_1,\dots,g_l. \] 
Choose liftings 
\[ \widetilde f_i,\widetilde g_j\in A. \]
Since the two families generate the same ideal in $A_k$, for every $i$ we may write 
\[ f_i=\sum_{q=1}^l \overline h_{iq}g_q \] 
with $\overline h_{iq}\in A_k$. Choosing liftings $h_{iq}\in A^\circ$, we obtain 
\[ \widetilde f_i = \sum_{q=1}^l h_{iq}\widetilde g_q+\alpha_i \] 
with $\alpha_i\in A_K^{\circ\circ}$. In particular, $|\alpha_i|<1$. By Lemma \ref{lem:tempered_substitution}, the substitution
\[ y_i\mapsto \sum_{q=1}^l h_{iq}z_q+\alpha_i \] 
defines a morphism $\temp{A_K}{y_1,\dots,y_s} \longrightarrow \temp{A_K}{z_1,\dots,z_l}$.
Moreover, the ideal 
\[ (y_1-\widetilde f_1,\dots,y_s-\widetilde f_s) \] 
is sent into the ideal 
\[ (z_1-\widetilde g_1,\dots,z_l-\widetilde g_l), \] 
because 
\[ \sum_{q=1}^l h_{iq}z_q+\alpha_i-\widetilde f_i = \sum_{q=1}^l h_{iq}(z_q-\widetilde g_q). \] 
Therefore we obtain a morphism 
\[ \frac{\temp{A_K}{y_1,\dots,y_s}} {(y_1-\widetilde f_1,\dots,y_s-\widetilde f_s)} \longrightarrow \frac{\temp{A_K}{z_1,\dots,z_l}} {(z_1-\widetilde g_1,\dots,z_l-\widetilde g_l)}. \] 
By the same argument in the other direction, using expressions of the $g_j$ in terms of the $f_i$, we obtain a morphism in the opposite direction. These two morphisms are inverse to each other. Indeed, after composing with the canonical injections into the corresponding algebras of functions on the rigid tube, constructed in Corollary \ref{prop tempered tube is contained in rigid}, they agree with the canonical isomorphism between the two descriptions of the usual rigid tube (see \cite[Proposition 2.3.2(ii)]{LeStum}). Therefore, the two morphisms are inverse already on the tempered quotients. Thus the tempered tube is independent of the chosen regular sequence and of the chosen liftings. \end{proof}

Our aim is now proving a tempered version of the so called \say{Weak fibration theorem}, see \cite[Corollary 2.3.16]{LeStum} for the rigid setting. We need the following preliminary result.

\begin{lem}\label{tempetale}
    Let $u\colon {\hat P}'=\operatorname{Spf}(A')\longrightarrow {\hat P}=\operatorname{Spf}(A)$ be an étale morphism of smooth affine $p$-adic formal schemes over $\mathcal V$. Suppose there is a commutative diagram of regular immersions
    \begin{equation*}
     \begin{tikzcd}
  & {\hat P}' \arrow{dd}{u}\\
X \arrow{ur} \arrow[swap]{dr}&\\
  & {\hat P}.
\end{tikzcd}
    \end{equation*}
  Then there is an isomorphism $$\temptube{X}{{\hat P}}\cong\temptube{X}{{\hat P}'}.$$
\end{lem}
\begin{proof}
By \cite[Proposition 2.3.15]{LeStum}, under the above hypotheses one has $u^{-1}(X)=X$.  Let the ideal of $X$ in ${\hat P}_k$ be generated by the regular sequence $f_1,\dots,f_s \in A_k$, and choose liftings $\widetilde f_1,\dots,\widetilde f_s\in A$. Since $A_k\to A'_k$ is \'etale, hence flat, the sequence $f_1,\dots,f_s$ remains regular after base change to $A'_k$ (see \cite[\href{https://stacks.math.columbia.edu/tag/067P}{Tag 067P}]{stacks-project}). Thus the pullbacks of the $\widetilde f_i$ define the tempered tube of $X$ in $P'$.

We now follow the proof of \cite[Proposition 2.3.15]{LeStum}. For $N\geq 1$, set 
\[ J_N = \left( y_1-\tilde{f}_1,\dots,y_s-\tilde{f}_s, y^{\alpha}=\prod y_i^{\alpha_i},\alpha \in {\N}^s, |\alpha| \geq N, \pi^N \right) \subset A\langle y_1,\dots,y_s\rangle . \] 
By base change, the morphism 
\[ A\langle y_1,\dots,y_s\rangle/J_N \longrightarrow A'\langle y_1,\dots,y_s\rangle/J_N \] 
is \'etale. Modulo the nilpotent ideal generated by 
\[ y_1,\dots,y_s,\widetilde f_1,\dots,\widetilde f_s,\pi, \] 
it is an isomorphism, because $u^{-1}(X)=X$. Since an étale morphism which is an isomorphism modulo a nilpotent ideal is itself an isomorphism, the preceding map is an isomorphism for every $N$.
Taking the limit over all $N$ we have that 
    $$
    \power{A}{y_1,\dots,y_s}/ (y_1-\tilde{f}_1,\dots,y_s-\tilde{f}_s) \rightarrow \power{A'}{y_1,\dots,y_s}/ (y_1-\tilde{f}_1,\dots,y_s-\tilde{f}_s) 
    $$ 
is an isomorphism.  By tensoring by $K$ and  via base change to $\temp{K}{y_1,\dots,y_s}$  we get  the  isomorphism $\temptube{X}{{\hat P}}\cong\temptube{X}{{\hat P}'}$.
\end{proof}
We can now prove the main result of this section.
\begin{thm}[Tempered weak fibration theorem]\label{weakfibr}
     Let 
      \begin{equation*}
     \begin{tikzcd}
  & {\hat P}' \arrow{dd}{u}\\
X \arrow{ur} \arrow[swap]{dr}&\\
  & {\hat P}, 
\end{tikzcd}   
    \end{equation*}
    be a smooth morphism of regular formal embeddings (in the sense of \cite[Chapter 2]{LeStum}) of $X$ into affine $p$-adic formal schemes. 
    Then locally on $X$ the morphism $u_{K}$ induces an isomorphism
    $$\temptube{X}{{\hat P}'}\cong\temptube{X}{{\hat P}} \times_{K} \goth S (\temp{K}{x_1,\dots,x_d}).$$
\end{thm}
\begin{proof}
    By \cite[Lemma 2.3.14 and  Corollary 2.3.16]{LeStum}, we can assume that ${\hat P}$ and ${\hat P}'$ are affine and $u$ factors as an étale map ${\hat P}'\rightarrow\hat{\mathbf{A}_{\hat P}^d}$ followed by the projection $\hat{\mathbf{A}_{\hat P}^d}\rightarrow {\hat P}$.  Because of the Lemma \ref{tempetale}, we can replace ${\hat P}'$ with  $\hat{\mathbf{A}_{\hat P}^d}$ and use    the zero section of $\hat{\mathbf{A}_{\mathcal{V}}^d}$ to embed $X$ in $\hat{\mathbf{A}_{\hat P}^d}$. Then  we have \small\begin{dmath*}\temptube{X}{{\hat P}'}\simeq\temptube{X}{\hat{\mathbf{A}_{\hat P}^d}}=\temp{\tate{A_K}{x_1,\dots,x_d}}{y_1,\dots,y_s,t_1,\dots,t_d}/(y_1-\tilde{f}_1,\dots,y_s-\tilde{f}_s,t_1-x_1,\dots,t_d-x_d)\simeq\temptube{X}{{\hat P}}\wotimes_K\temp{K}{x_1,\dots,x_d}.\end{dmath*}\normalsize
\end{proof}

\section{The analytic cotangent complex and the tempered convergent rigid cohomology}

\medskip

In this section we will globalize the definition of the tempered tube. We will use this to associate with a  scheme $X$ over $k$ a well-defined tempered convergent cohomology which is independent of the choices made to define the tempered tube (under some hypotheses of regularity  on the embedding of $X$ in a smooth formal scheme).
Such a cohomology will be the \say{tempered} version of the classical convergent rigid cohomology: the tube of radius one  (strictly less than 1) of $X$ on which one calculates the de Rham cohomology will be replaced by the tempered tube. For this reason we will introduce the ind-Banach version of the K\"ahler differentials and de Rham cohomology, that are the analytic cotangent complex and derived de Rham cohomology.

\subsection{Derived analytic stacks and spaces} \label{section cohomology}
\medskip
Let ${\hat P} $ be a $p$-adic formal scheme over $\mathcal{V}$, and let ${\hat P}_K$ be its generic fiber.
Consider an affine open cover, $\{ \mathrm{Spf}(A_i) \}_{i \in I}$, of $\hat P$, then the affinoid spaces $\mathrm{Spa}(A_{i,K})$ form an admissible open affinoid cover of ${\hat P}_K$. We can associate with $A_{i,K}$ the derived analytic spectrum $\goth{S}(A_{i,K})$, equipped with its structure sheaf, we introduced so far (see Definition \ref{DEF}). 

In order to globalize this construction we use the functor-of-points perspective. We look at global derived analytic spaces as sheaves over the (formal) homotopical Zariski site
\[ \sf{dAff}_K = \sf{Comm}(\dercat_\infty(\sf{CBorn}_K))^{\rm op}. \]
More precisely, we endow $\sf{dAff}_K$ with the homotopical Zariski topology: open immersions are induced by homotopy epimorphisms of bornological algebras, and covers are finite families of such morphisms whose pullback functors on module categories form a conservative family.
We call $\sf{dAff}_K$  the category of \emph{affine derived analytic spaces} over $K$. It is not difficult to check that the spectrum functor introduced  so far $A \mapsto (\goth{S}(A), \mathcal{O}_A)$ realizes the category $\sf{dAff}_K$ as a subcategory of the category of Ind-Banach homotopically-ringed topological spaces.

\medskip
\begin{dfn} \label{defn:Zariski_sheaf}
    A \emph{derived analytic stack} over $K$ is a $\infty$-functor 
    \[ \mathscr{F}: \sf{dAff}_K^{\rm {op}} \to \sf{sSets}, \]
    where $\sf{sSets}$ is the $\infty$-category of spaces, such that $\mathscr{F}$ satisfies $\infty$-descent for the Zariski covers introduced so far. We denote by $\sf{Stck}_{\sf{Zar}, K}$ the full $\infty$-subcategory of $\sf{Fun}(\sf{dAff}_K^{\rm {op}}, \sf{sSets})$ determined by derived analytic stacks.
\end{dfn}
\medskip

\begin{rmk} \label{rmk:hypercompleteness}
By Definition \ref{defn:Zariski_sheaf}, a derived analytic stack is a sheaf for the homotopical Zariski topology. We do not impose hyperdescent. Thus our notion is the ordinary sheaf condition in the sense of \cite[Chapter 7]{LurieHTT}, rather than the stronger condition of being a hypercomplete sheaf. These two notions need not coincide in general, even on coherent sites (see \cite[Counterexample 6.5.4.5]{LurieHTT}).
\end{rmk}
\medskip

\begin{rmk} \label{rmk:hypercompleteness1}
   It is standard computation to check that the homotopy Zariski topology is subcanonical.
\end{rmk}

\medskip

Thus, by  Remark \ref{rmk:hypercompleteness1}, the Yoneda embedding $h: \sf{dAff}_K \to \sf{Fun}(\sf{dAff}_K^{\rm {op}}, \sf{sSets})$ factors through $\sf{Stck}_{\sf{Zar}, K}$. Moreover, for any $X \in \sf{dAff}_K$ we obtain a small site over $X$ by restricting the Zariski topology given by the opens induced by homotopy epimorphism from $A$, if $X$ is in duality with $A \in \sf{Comm}(\dercat_\infty(\sf{CBorn}_K))$. The small Zariski site associated to $X$ is equivalent to the site associated with the topological space $\goth{S}(A)$ introduced so far. We therefore think to $X$ as being represented by the spectrum $(\goth S (A), \mathcal O_A)$. 

\medskip

\begin{dfn} \label{defn:derived_analytic_space}
    A \emph{derived analytic space} over $K$ is an object of $\sf{Stck}_{\sf{Zar}, K}$ that can be written as a colimit of a diagram of representable objects whose morphisms are open Zariski immersions. In such a situation we say that the derived analytic space is obtained by glueing the affine analytic spaces.
\end{dfn}
\medskip

\begin{rmk} \label{rmk:derived_analytic_space_atlas}
    Definition \ref{defn:derived_analytic_space} can be reformulated by saying that $X \in \sf{Stck}_{\sf{Zar}, K}$ is a derived analytic space if it admits an affine Zariski atlas. This means that there exists an effective epimorphism 
    \[ \coprod_{i\in I} U_i\longrightarrow X, \] 
    where each $U_i$ is an affine derived analytic space and each morphism $U_i\to X$ is a Zariski open immersion, and for every $i,j$, the fiber product $U_i\times_X U_j$ 
    is again a affine derived analytic space.
\end{rmk}

In particular, in the situation where $\hat{P}$ is a $p$-adic formal scheme and $\{\mathrm{Spf}(A_i) \}_{i \in I}$ a formal affine cover of it, we get a diagram $\{ \mathrm{Spa}(A_{i,K} \otimes_{P} A_{j,K})) \to \mathrm{Spa}(A_{i,K}) \}_{i,j \in I}$ of affinoid spaces and open affinoid embeddings. Passing to generic fibers, this gives a diagram of affine derived analytic spaces 
\[ \goth S(A_{i_0\cdots i_p,K}) \]
via the \u{C}ech nerve of the cover. We define the analytification of $\widehat P$ to be the colimit of this \u{C}ech diagram in $\sf{Stck}_{\mathrm{Zar},K}$ 
\[ \widehat P^{\mathrm{an}} = \colim_{\Delta} \goth S (A_{i_0\cdots i_p,K}). \] 
This construction is independent of the chosen affine covering. As shown in the next remark:

\begin{rmk}
Let 
\[ (-)_K\colon \sf{AffFrSp}_{\mathcal V}\longrightarrow \sf{dAff}_K \] 
be the generic fiber functor from affine $p$-adic formal schemes of topologically finite type over $\mathcal V$ to affine derived analytic spaces over $K$. It sends $\operatorname{Spf}(A)$ to the affine derived analytic space represented by $A_K$, $\goth S(A_K)$. Precomposition with $(-)_K$ induces an adjunction
\[ L: \sf{Fun}(\sf{AffFrSp}_{\mathcal{V}}^{\rm{op}}, \sf{sSets}) \leftrightarrows \sf{Fun}(\sf{dAff}_K^{\rm{op}}, \sf{sSets}) \]
the left adjoint, $L$, is the left Kan extension via $(\cdot)_K$. Since $(\cdot)_K$ is continuous (it follows from the main results of \cite{BenKre} that it sends covers for the Zariski topology of formal schemes to covers for the homotopical Zariski topology), the adjunction restricts to an adjunction
\[ \sf{Shv}_{\sf{Zar}}(\sf{AffFrSp}_{\mathcal{V}}) \leftrightarrows \sf{Stck}_{\sf{Zar}, K}. \]
Thus, if $\hat P \cong \colim_{i \in I} A_i$ as a sheaf in $\sf{Shv}_{\sf{Zar}}(\sf{AffFrSp}_K)$ then 
\[ L(\hat P) \cong L(\colim_{i \in I} A_i)  \cong  \colim_{i \in I} A_{i, K} \cong \hat P^{\rm {an}}. \]
This shows that the association $\hat P \mapsto \hat P^{\rm an}$ is functorial and therefore, given any other cover $\{ B_j \}$ of $\hat P$ we get that computing the analytification $\hat P^{\rm an}$ using this cover yields the same result, up to  equivalence. 
\end{rmk}

\medskip

\begin{rmk}
   By the same reasoning, any rigid analytic space over $K$ can be viewed as an object of $\sf{Stck}_{\mathrm{Zar},K}$, by writing it as a Zariski gluing of affinoid spaces and applying the functor-of-points construction above.
\end{rmk}
\medskip

We now recall the definition of quasi-coherent sheaves on prestacks. On affine derived analytic spaces we have the assignment 
\[ \operatorname{Spec}^{\mathrm{an}}(A) \mapsto \dercat_\infty(\sf{Mod}_A). \] 
This assignment extends to arbitrary prestacks by right Kan extension along the Yoneda embedding. Thus, for a prestack 
\[ X\in \sf{Fun}(\sf{dAff}_K^{\mathrm{op}},\sf{sSet}), \] 
we define 
\[ \sf{Q Coh}_X = \lim_{\operatorname{Spec}^{\mathrm{an}}(A)\to X} \dercat_\infty(\sf{Mod}_A), \] 
where the limit is taken over the category of affine derived analytic spaces mapping to $X$.

We now check that if $X \in \sf{Stck}_{\sf{Zar}, K}$ is a derived analytic space obtained glueing the family of affine spaces $\{ U_i \}_{i \in I}$ then we can compute
\[ \sf{Q Coh}_X \cong \lim_{\Delta} \dercat_\infty(\sf{Mod}_{A_{i_0\cdots i_p}}). \]
\medskip

\begin{pro} \label{prop:Mod_satisfies_descent}
The functor $\goth{S}(A) \mapsto \dercat_\infty(\sf{Mod}_A)$ is a sheaf of $\infty$-categories on $\sf{dAff}_K$.
\end{pro}
\begin{proof}
   The same argument in  \cite[Proposition 10.5]{ClauSchCond} applies here. Namely, for a homotopical Zariski cover $\{A\to B_i\}$, the category $\dercat_\infty(\sf{Mod}_A)$ is recovered as the limit of the corresponding \u{C}ech nerve of the diagram of module categories. The key points are that the morphisms $A\to B_i$ are homotopy epimorphisms and that the associated family of pullback functors on module categories is conservative.
\end{proof}

Let us now introduce sheaves on $\sf{Stck}_{\sf{Zar}, K}$\footnote{We ignore the set-theoretic issue of the size of the category $\sf{Stck}_{\sf{Zar}, K}$ supposing it small in a bigger, tacitly fixed, universe.}, following \cite[Notation 6.3.5.16]{LurieHTT}

\begin{dfn} \label{defn:sheaves_on_stacks}
    The category of sheaves $\sf{Sh}_{\sf{C}}(\sf{Stck}_{\sf{Zar}, K})$ over $\sf{Stck}_{\sf{Zar}, K}$ valued in a   complete $\infty$-category $\sf{C}$ is the category of limit-preserving functors $\sf{Stck}_{\sf{Zar}, K} \to \sf{C}$.
\end{dfn}

The surprising result is that $\sf{Sh}_{\sf{C}}(\sf{Stck}_{\sf{Zar}, K})$ does not give anything new.

\begin{pro} \label{prop:sheaves_on_stacks_are_stacks}
Let $\sf{C}$ be a complete $\infty$-category. Restriction along the Yoneda embedding induces an equivalence between $\sf C$-valued sheaves on $\sf{Stck}_{\mathrm{Zar},K}$ and $\sf C$-valued sheaves on $\sf{dAff}_K$.
\end{pro}
\begin{proof}
This is \cite[Lemma A.3.4]{Mann}.
\end{proof}

Since the homotopical Zariski topology is subcanonical, then for any $X \in \sf{Stck}_{\sf{Zar}, K}$ the colimit
\[ X \cong \colim_{j \in J} (\goth{S}(A_j), \mathcal{O}_{A_j}) \]
is actually computed in $\sf{Stck}_{\sf{Zar}, K}$ and 
\[ \sf{QCoh}_X = \lim_{j \in J} \dercat_\infty(\sf{Mod}_{A_j}).  \]

This shows that $\sf{QCoh}_{(-)}$ (defined via the  left Kan extension) is a sheaf of $\infty$-categories on $\sf{Stck}_{\sf{Zar}, K}$, hence the following proposition which will allow us  to deal with global spaces.

\begin{pro} \label{prop:QCoh_computed_by_Cech}
Let $X$ be a derived analytic space obtained by gluing affine derived analytic spaces $U_i=\goth S(A_i)$.  Let 
\[ U_\bullet\to X \] 
be the \u{C}ech nerve of this affine cover, and write $U_{i_0\cdots i_p} = \goth S(A_{i_0\cdots i_p})$. Then 
\[ \sf{QCoh}_X \simeq \lim_{\Delta} \dercat_\infty(\sf{Mod}_{A_{i_0\cdots i_p}}). \]
\end{pro}
\begin{proof}
We have seen that $\sf{QCoh}_{(-)}$ is a sheaf on $\sf{Stck}_{\sf{Zar}, K}$ and hence commutes with all limits.
\end{proof}

Thus $\sf{QCoh}_X$ can be computed by the derived \u{C}ech complex associated with any affine Zariski cover of $X$.
\medskip

\begin{rmk} \label{rmk:tensor_product_qcsheaves}
The limit of symmetric monoidal $\infty$-categories by a system of strictly monoidal functors has a canonical structure of symmetric monoidal $\infty$-category (see \cite[Chapter 3, \S 3.2.1--3.2.2]{gaitsgory2019study}). In particular this is true for $\sf{Q Coh}_X$ and we will denote its monoidal structure by $\wotimes_X$. This makes it possible to apply the spectrum construction of Section \ref{definition of the spectrum} to $\sf{QCoh}_X$. On the other hand, if $X$ is a derived analytic space, one can also construct a homotopically ringed space by gluing the spectra of the affine pieces in an affine cover. Under suitable hypotheses, these two constructions should agree. This suggests an analytic analogue of the Balmer--Rosenberg reconstruction theorem.
\end{rmk}

 \subsection{ The analytic cotangent complex} \label{cotangent} 
\medskip
 
Our aim now is to define the {\it cotangent complex} of the derived analytic space  $\hat P^{\rm an}$ associated with the generic fiber of a formal  $\mathcal V$-scheme,  ${\hat P}$. We will use it in order to define derived analytic de Rham cohomology.  Our definition  of cotangent complex is  taken from   that one  given in \cite{Vez}.

 The definition of the cotangent complex is local for the homotopy Zariski topology. We may therefore begin with the affine case. Let $\widehat P = \operatorname{Spf}(A)$ be an affine $p$-adic formal scheme over $\mathcal V$, with $A$ a topologically finitely generated $p$-adic $\mathcal V$-algebra, and let $A_K:=A\otimes_{\mathcal V}K$ denote its analytic generic fiber. If $\widehat P$ is obtained as the $p$-adic completion of an affine $\mathcal V$-scheme $P = \operatorname{Spec}(B)$, then we also write $B_K:=B\otimes_{\mathcal V}K$ for the algebraic generic fiber. In this situation there is a natural morphism $B_K\to A_K$, which is a sort of analytification morphism.

 We will use the classical algebraic cotangent complex, whose definition we now recall. We denote by $\dercat_{\infty}( \sf{Mod}_{K}^{alg})$ and $\dercat_{\infty}( \sf{Mod}^{alg}_{B_K})$ the closed symmetric monoidal $\infty$-categories of algebraic modules over $K$ and $B_K$ respectively, with their standard monoidal structures given by $\otimes_{K}^{\bb{L}}$  and 
 $\otimes_{{B_K}}^{\bb{L}}$. These can be presented via suitable model categories of chain complexes and satisfy the hypothesis of \cite{Vez}. In this setting, following \cite[page 27]{Vez}, the (absolute)  \emph{cotangent complex} ${\mathbb L}_{B_K/{K}}$ is the object of $\dercat_{\infty}( \sf{Mod}_{B_K}^{alg})$ defined by
 $$
 {\mathbb L}_{B_K/{K}}= {\mathbb R}{\rm Q}{\mathbb L}{\rm I} (B_K \otimes^{\mathbb L}_{K} B_K \to B_K).
 $$
Here the functor  $\rm I$ is defined in the following way. It takes a  commutative algebra $C$ over $B_K$ in $\dercat_{\infty}( \sf{Mod}_{K})$ equipped with a morphism $p: C \to B_K$ to $\ker(p)$. This means that we have the following diagram of maps
$$ B_K \xrightarrow{i} C \xrightarrow{p} B_K,
$$
and then $I$ is defined as the kernel of the second arrow. The result is non-unital $B_K$-algebra. The second functor  $\rm Q$ associates to a non-unital $B_K$-algebra $C$, a $B_K$-module,  ${\rm Q}(C)$ computes the cokernel of the map of $B_K$-monoids
$$
C \otimes_{B_K} C \xrightarrow{\mu} C
$$
where $\mu$ is the multiplication. Both functors can be derived (computing homotopy fibers and cofibers instead of kernels and cokernels) obtaining the functors  ${\mathbb R}{\rm Q}$ and 
${\mathbb L}{\rm I}$.  In particular, $H^0 ({\mathbb L}_{B_K/{K}})$ is nothing but the classical algebraic  K\"ahler differentials $\Omega^1_{B_K/{ K}}$.

In \cite[page 58]{Vez}, it is proven that, in the algebraic setting this definition coincides with the classical Illusie-Quillen one. The free $K$-polynomial resolutions are cofibrant in the underlying model category. In particular in the case where $B_K$ is smooth over $K$, we do have 
$$
 {\mathbb L}_{B_K/{ K}} \cong \Omega^1_{B_K/{ K}},
 $$
hence,  it is concentrated at level 0 and it is locally free.

By considering the analogue construction as  in  the algebraic case, we have {\it the analytic cotangent complex} $
 {\hat{\mathbb L}}_{{A_K}/{ K}}$, where now $A_K$ is considered as a Banach $K$-algebra and such a cotangent complex is computed in the HAG-context of ind-Banach algebras. Thus, we have the formula
 $$
 \widehat{\mathbb L}_{A_K/{K}}= {\mathbb R}{\rm Q}{\mathbb L}{\rm I} (A_K \wotimes^{\mathbb L}_{K} A_K \to A_K).
 $$
 where the completed tensor product takes the place of the algebraic tensor product and the objects are considered as ind-Banach modules.

 \begin{rmk} \label{cotan=0} It is clear from the definition of the cotangent complex that if $A_K \rightarrow C$ is a homotopy epimorphism, then $\widehat{\mathbb L}_{C/{A_K}} \cong 0$.
 \end{rmk} 

      \medskip

 Now we would like to compare  $ {\mathbb L}_{B_K/{ K}}$  and $
 {\hat{\mathbb L}}_{{A_K}/{ K}}$  in the case where $A_K$ is obtained as an analytic completion of $B_K$, and where $B_K$ is a finitely generated smooth $K$-algebra. In order to formulate this comparison inside the category of bornological vector spaces, we need to  recall the fine bornology. This bornology was already used implicitly in Example \ref{polynomial algebra in IndBanR}.
\medskip

 \begin{dfn} \label{defn:fine_bornology}
     Let $V$ be a $K$-vector space. Let us write $V = \bigoplus_{v \in B} K$ for a base $B \subset V$ as $K$-vector space. The \emph{fine bornology} on $V$ is the bornology given by computing the direct sum 
     \[ V_{\rm{fine} } = \bigoplus_{v \in B} K \]
     in $\sf{CBorn}_K$. 
 \end{dfn}

 Notice that Definition \ref{defn:fine_bornology} is well-posed because $K$ has only one bornology, up to isomorphism, that is separated and compatible with its Banach structure.
 \medskip

 \begin{exa} \label{exa:fine}
     The most basic object of $\dercat_{\infty}( \sf{Mod}_{K})$ equipped with the fine bornology is $K[x]$ when identified with $\bigoplus_{n \in \N} K^n$. This is the bornology used in Example \ref{polynomial algebra in IndBanR} to define the analytic affine line over $K$.
 \end{exa}
 \medskip

 The fine bornology defines a functor $(\cdot)_{\rm{fine}}: \sf{Mod}_{K}^{alg} \to \sf{Mod}_{K}$. It is easy to check that this functor is the left adjoint to the forgetful functor $\sf{Mod}_{K} \to \sf{Mod}_{K}^{alg}$, is strongly exact, strongly symmetric monoidal and that it is fully faithful, see \cite[ Examples 1.27, 1.85]{meyer2007local}.   Therefore, we obtain a symmetric strongly monoidal functor $(\cdot)_{\rm{fine}}: \dercat_{\infty}(\sf{Mod}_{K}^{alg}) \to \dercat_{\infty}(\sf{Mod}_{K})$. In particular, it sends algebraic commutative differential graded $K$-algebras  to commutative differential graded algebras in bornological $K$-vector spaces. When no confusion can arise, we will omit the functor $(\cdot)_{\mathrm{fine}}$ from the notation.

We will use the following results from  \cite{BassatMuk}.

 \begin{pro}\label{open} 
 Let $K[x_1,\dots,x_n]$ be endowed with the fine bornology. Then the natural morphism \[ K[x_1,\dots,x_n]_{\mathrm{fine}} \longrightarrow K\langle x_1,\dots,x_n\rangle \] to its affinoid completion is a homotopy epimorphism in $\dercat_\infty(\sf{Mod}_K)$. More generally, let $B_K$ be a finitely generated algebra over $K$, endowed with the fine bornology, and let $A_K$ be its affinoid completion. Then the natural morphism \[ (B_K)_{\mathrm{fine}}\longrightarrow A_K \] is a homotopy epimorphism.
 \end{pro}
 \begin{proof}  The first result is \cite[Corollary 5.6]{BassatMuk} and the second one is \cite[Proposition 5.22]{BassatMuk}. 
 \end{proof}
 
It makes sense to consider the bounded  (analytic) version of  K\"ahler differentials and their exterior powers. Let $A_K$ be a commutative Banach, or more generally bornological, $K$-algebra concentrated in degree zero. We denote 
\[ \hat{\Omega}^1_{A_K/K} = \frac{I}{I^2} \]
where
\[ I = \ker(A_K \wotimes_K A_K \to A_K) \]
is the kernel of the multiplication map. 
We also have its exterior powers (as bornological $A_K$-module)
\[ \hat{\Omega}^n_{A_K/K} =\Lambda^n \hat{\Omega}^1_{A_K/K} \]
 which combined form the \emph{analytic de Rham complex} of $A_K/K$, $ \hat{\Omega}^{\bullet}_{{A}_K^{}/K}$. 

 Finally, because the functor $(-)_{\mathrm{fine}}$ is fully faithful, exact, and strongly symmetric monoidal, for an ordinary $K$-algebra $B_K$ one has a natural identification \[ \widehat\Omega^\bullet_{(B_K)_{\mathrm{fine}}/K} \cong \Omega^\bullet_{B_K/K}. \] We will use this identification without further comment when the context is clear.

 \medskip

 \begin{thm}\label{co-fibrant} 
 Assume that $B_K$ is a finitely generated smooth $K$-algebra and that $A_K$ is an affinoid completion of $B_K$. Endow $B_K$ with the fine bornology. Then there are natural equivalences in $\dercat_\infty(\sf{Mod}_{A_K})$ 
 $$
 {\hat{\mathbb L}}_{{A_K}/{ K}} \cong  A_K  {\hat \otimes}^{\mathbb L}_{B_K} {\hat{\mathbb L}_{B_K/{ K}}}   \cong  A_K  {\hat \otimes}_{B_K} {\Omega^1_{B_K/{ K}}} \cong    \hat{\Omega}^1_{A_K/K}.
 $$
\end{thm}
\begin{proof}
By Proposition \ref{open}, the map $B_K \rightarrow A_K$ is a homotopy epimorphism. Therefore, we can localize the cotangent complex applying \cite[Proposition 1.2.1.6]{Vez} because the relative cotangent complex of $B_K \rightarrow A_K$ vanishes. This gives the first isomorphism.
 
 We remark that for $B_K$ we have
\[ \hat{\mathbb L}_{(B_K)_{\rm{fine}}/K} \cong \mathbb L_{B_K/K} \cong \Omega_{B_K/K}^1 \cong {\hat \Omega}_{(B_K)_{\rm{fine}}/K}^1 \]
as $B_K$-modules, because we suppose $B_K$ smooth and the functor $(\cdot)_{\rm{fine}}$ is strongly exact and strongly monoidal, therefore it sends cofibrant $K$-algebras (i.e. polynomial algebras) to cofibrant bornological $K$-algebras (i.e. polynomial algebra with the fine bornology). We conclude from  the fact that ${\Omega^1_{B_K/{ K}}}$ is locally free and finite rank as $B_K$-module and hence flat for the functor $A_K\wotimes_{B_K} - $. Moreover the sheaf of differentials on the rigid analytification of the affine scheme $\operatorname{Spec} B_K$, is a coherent sheaf (it is the analytification of the coherent module of the algebraic differential, see \cite{BALD}). The affinoid algebra $A_K$ identifies ad an admissible affinoid open in such a space, hence the last isomorphism. 
\end{proof}
 
\begin{rmk} 
The same argument applies in the relative case. Namely, for a morphism of affinoid completions arising from a morphism of smooth algebraic $K$-algebras, the relative analytic cotangent complex is obtained by derived completed base change from the corresponding algebraic relative cotangent complex.
\end{rmk}

 \medskip

All these constructions are local for the homotopy Zariski topology, and therefore can be glued and we can associate a quasi-coherent sheaf of modules which we will indicate by $\hat{\mathbb L}_{{\hat P}_K^{an}/K} \in \sf{QCoh}_{\hat P_K^{\rm an}}$ on the derived analytic space $\hat P_K^{\rm an}$. The object $\hat{\mathbb L}_{{\hat P}_K^{an}/K}$ is well defined because by Remark \ref{rmk:derived_analytic_space_atlas}, derived analytic spaces are $0$-geometric  stacks and, by general results like \cite[Proposition 1.4.1.8 and Proposition 1.4.1.11]{Vez}, one can check that such geometric stacks have a well-defined cotangent complex. Alternatively, one can check that the compatibility of the cotangent complex with base change through homotopy epimorphisms give a well-defined object of $\sf{QCoh}_X$.

\medskip
\begin{rmk}
    This construction can be applied also to  the various exterior powers of  $\hat{\mathbb L}_{{\hat P}_K^{an}/K}$, $ n \in \N$: they still are quasi-coherent sheaves of modules over  ${\hat P}_K^{\rm{an}}$. 
\end{rmk}
\medskip

We define our cohomology theory  using analytic  derived de Rham complex. Following \cite{raksit} (and \cite{kelly2021analytic})  we define the functor
\[ \mathbb{L} \hat{\Omega}_{-/K}^{+ \bullet}: \sf{Comm}(\dercat_{\infty}( \sf{Mod}_{ {K}})) \to \sf{DG}_+^{\ge 0}\sf{Comm}(\dercat_{\infty}( \sf{Mod}_{ {K}})) \]
where the target is the category of non-negative $h_+$-differential graded algebras in the sense of \cite[Notation 5.3.1]{raksit}. Explicitly \cite[Theorem  5.3.6]{raksit} (see also \S5 of \cite{kelly2021analytic}), for an  object $A_K \in \sf{Comm}(\dercat_{\infty}( \sf{Mod}_{ {K}})) $ we have a canonical isomorphism
\[ \mathbb{L} \hat{\Omega}_{A_K/K}^{+ n} \cong \mathbb{L} \sf{Sym}_-^n(\hat{\mathbb{L}}_{A_K/K}[1](1)) \cong \mathbb{L} \Lambda^n (\hat{\mathbb{L}}_{A_K/K})[n] \]
of graded pieces and the $h_+$-differential is given by the  derivation $A_K \rightarrow  \Lambda^1 (\hat{\mathbb{L}}_{A_K/K})[1]= (\hat{\mathbb{L}}_{A_K/K})[1] $. The object $\mathbb{L} \hat{\Omega}_{A_K/K}^{+ \bullet}$ gives a double complex over $K$, of which we can compute the direct sum realization 
\[ \mathbb L\widehat\Omega^\bullet_{A/K} = \operatorname{Tot}^{\oplus} \bigl(\mathbb L\widehat\Omega^{+\bullet}_{A/K}\bigr). \] 
More explicitly, 
\[ \mathbb L\widehat\Omega^\bullet_{A/K} \simeq \mathbb L \operatorname*{colim}_{N \in \N} \operatorname{Tot}^{\oplus} \left( \tau_{\leq N} \mathbb L\widehat\Omega^{+\bullet}_{A/K} \right). \] 
We call $\mathbb{L} \hat{\Omega}_{-/K}^{\bullet}$ it the \emph{  analytic derived de Rham complex}. In the algebraic setting the analogous functor $\mathbb{L} \Omega_{-/K}^{\bullet}$ is equivalent to the \emph{derived de Rham complex} studied in \cite{Bhatt}. 

The complex $\mathbb{L} \hat{\Omega}_{-/K}^{\bullet}$ differs in general from the Hodge-completed analytic derived de Rham complex, which is obtained by taking the direct product realization of the double complex. Indeed, the \emph{Hodge completed analytic derived de Rham complex} is defined
\[
\mathbb L\widehat\Omega^{\wedge, \bullet}_{A/K} = \operatorname{Tot}^{\scalebox{0.55}{$\prod$}} \bigl(\mathbb L\widehat\Omega^{+\bullet}_{A/K}\bigr) \cong \mathbb{R} \lim_{n \in \N} \operatorname{Tot}^{\scalebox{0.55}{$\prod$}} \bigl(\mathbb L\widehat\Omega^{+\bullet}_{A/K}/F^N \bigr)  \] 
where $F^N$ denotes the Hodge filtration.
In the smooth and finite dimensional situation considered in this paper, however, the cotangent complex is a finite projective module of finite rank. Its exterior powers therefore vanish in sufficiently high degrees, and the direct sum and Hodge-completed realizations coincide.

Our goal  (Theorem \ref{derived de rham}) is to compute $ \mathbb{L} \hat{\Omega}_{A_K/K}^{\bullet}$ for $A_K$ a smooth affinoid $K$-algebra which is the generic fiber of the $p$-adic completion of a finitely generated smooth $K$-algebra $B_K$. 

\begin{pro} \label{prop:comparison_derived_de_rham}
Let $K \to B_K$ be a smooth finitely generated $K$-algebra. Then there is a canonical quasi-isomorphism
\[ \mathbb{L} \Omega_{B_K/K}^{\bullet} \cong \Omega_{B_K/K}^\bullet. \]
After equipping $B_K$ and its modules of differentials with the fine bornology, one also has 
\[ \mathbb L\widehat\Omega^{\bullet}_{(B_K)_{\mathrm{fine}}/K} \simeq \bigl(\Omega^\bullet_{B_K/K}\bigr)_{\mathrm{fine}}. \]
\end{pro}
\begin{proof}
Since $B_K$ is smooth over $K$, its cotangent complex is concentrated in degree zero and the canonical morphism $\mathbb L_{B_K/K}\longrightarrow \Omega^1_{B_K/K}$ is an equivalence. Moreover, $\Omega^1_{B_K/K}$ is a finite projective $B_K$-module. It follows that 
\[ \mathbb L\Lambda^n (\mathbb L_{B_K/K}) \simeq \Lambda^n (\Omega^1_{B_K/K}) = \Omega^n_{B_K/K}. \]
Under these identifications, the $h_+$-differential induced by the universal derivation is precisely the usual de Rham differential.  The corresponding statement in the bornological setting follows from the fact that the fine bornology functor is exact and strongly symmetric monoidal.
\end{proof}

As the cotangent complex, the analytic derived de Rham complex is compatible with extension of scalars along homotopy epimorphisms.

\begin{pro} \label{prop:base_change_graded_derived_de_rham}
Let $B \to A$ be a homotopy epimorphism of commutative algebras in $\dercat_\infty(\sf{Mod}_K)$. Then there is a natural equivalence of $h_+$-differential graded commutative algebras 
\[ \mathbb L\widehat\Omega^{+\bullet}_{A/K} \simeq \mathbb L\widehat\Omega^{+\bullet}_{B/K} \widehat\otimes_B^{\mathbb L} A. \] 
\end{pro}

\begin{proof}
Since $B\to A$ is a homotopy epimorphism, we have that $\widehat{\mathbb L}_{A/K} \simeq A\widehat\otimes_B^{\mathbb L} \widehat{\mathbb L}_{B/K}$, by \cite[Proposition 1.2.1.6]{Vez}. Also, derived exterior powers commute with extension of scalars, and hence, for every $n\geq 0$
\[ \mathbb L\Lambda^n \bigl( \widehat{\mathbb L}_{A/K} \bigr) \cong A\widehat\otimes_B^{\mathbb L} \mathbb L\Lambda^n \bigl(\widehat{\mathbb L}_{B/K}\bigr). \] 
These equivalences are compatible with the $h_+$-differentials, giving the equivalence.
\end{proof}

Proposition \ref{prop:base_change_graded_derived_de_rham} in particular applies to the case of the homotopy epimorphism $B_K \to A_K$, of a finitely generated $K$-algebra and its $p$-adic completion.

\begin{thm} \label{derived de rham} 
Let $B_K$ be a finitely generated smooth $K$-algebra, endowed with the fine bornology, and let $A_K$ be the affinoid completion of $B_K$. Then there are natural quasi-isomorphisms 
\[ \mathbb L\widehat\Omega^{ \bullet}_{A_K/K} \cong A_K\widehat\otimes^{\mathbb L}_{(B_K)_{\mathrm{fine}}} \mathbb L\widehat\Omega^{ \bullet}_{(B_K)_{\mathrm{fine}}/K} \cong A_K\widehat\otimes_{(B_K)_{\mathrm{fine}}} \Omega^\bullet_{B_K/K} \cong \widehat\Omega^\bullet_{A_K/K}. \] 
In particular, the analytic derived de Rham complex of $A_K$ agrees with its usual analytic de Rham complex of bounded differentials.
\end{thm}

\begin{proof}  
By Proposition \ref{open}, the morphism $(B_K)_{\mathrm{fine}} \to A_K$ is a homotopy epimorphism. Proposition \ref{prop:base_change_graded_derived_de_rham} therefore gives 
\[ \mathbb L\widehat\Omega^{+ \bullet}_{A_K/K} \cong A_K\widehat\otimes^{\mathbb L}_{(B_K)_{\mathrm{fine}}} \mathbb L\widehat\Omega^{+ \bullet}_{(B_K)_{\mathrm{fine}}/K}. \]
By Proposition \ref{prop:comparison_derived_de_rham}
\[ \mathbb L\widehat\Omega^{\bullet}_{(B_K)_{\mathrm{fine}}/K} \cong (\Omega^ \bullet_{B_K/K})_{\mathrm{fine}}. \]
Since $\Omega^n_{B_K/K}$ is a finite projective $B_K$-module, it follows that the derived completed tensor product agrees with the ordinary completed tensor product, hence by Theorem \ref{co-fibrant}:
\[ A_K\widehat\otimes^{\mathbb L}_{(B_K)_{\mathrm{fine}}} \Omega^n_{B_K/K} \cong A_K\widehat\otimes_{(B_K)_{\mathrm{fine}}} (\Omega^n_{B_K/K})_{\mathrm{fine}} \cong \widehat\Omega^n_{A_K/K}. \] 
These identifications are compatible with the de Rham differentials, and therefore give an isomorphism of complexes 
\[ A_K\widehat\otimes_{(B_K)_{\mathrm{fine}}} (\Omega^\bullet_{B_K/K})_{\mathrm{fine}} \cong \widehat\Omega^\bullet_{A_K/K}. \]

\end{proof}

The final identification  in terms of differentials in Theorem \ref{derived de rham} uses the fact that $A_K$ is an affinoid algebra in the sense of rigid analytic geometry, obtained from the affine $K$-scheme given by $B_K$.  For a general homotopy epimorphism $B_{K,\rm{fine}/K} \to A$, we can only localize the cotangent complex. Thus, the argument of the proof of Theorem \ref{derived de rham} gives the following result.

\begin{cor} \label{tensor2}
Let $(B_{K})_{\mathrm{fine}} \to A$ be a  homotopical epimorphism in $\dercat_{\infty}( \sf{Mod}_{K})$, where again $B_K$ is a smooth finitely generated  $K$-algebra.  Then, there is canonical quasi-isomorphism
\[  \mathbb{L} \hat{\Omega}_{A/K}^{\bullet} \cong  (\mathbb{L} \Omega_{B_K/K}^{ \bullet})_{\mathrm{fine}} \wotimes_{(B_{K})_{\mathrm{fine}}}^{\mathbb L} A \cong
\Omega_{B/K}^\bullet \wotimes_{(B_{K})_{\mathrm{fine}}} A . \]
\end{cor}

\bigskip

\subsection{Tempered convergent rigid cohomology}\label{temcoh} 
We now have all the ingredients to define the tempered convergent rigid cohomology.  Throughout this section, $X$ will be a variety over $k$, and assume that we are given a regular closed immersion $X\hookrightarrow {\hat P}_k$, where $\hat P$ is a  formal $\cal V$-scheme, smooth and  separated. Under these hypotheses, we define the tempered tube $\temptube{X}{{\hat P}}$ by gluing the affine tempered tubes introduced above (see Definition \ref{defn:tempered_tube}).
Indeed, let ${\hat P}_i$ be an affine open cover of ${\hat P}$ in which the closed immersions $X_i=X\cap {\hat P}_{i,k} \hookrightarrow {\hat P}_{i,k} $ are defined by a regular sequence (such a cover always exists by \cite[\href{https://stacks.math.columbia.edu/tag/0E9J}{Lemma 0E9J}]{stacks-project}).
We have that $\temptube{X_i}{{\hat P}_i}$ defines an open derived analytic subspace in ${\hat P}_i$ by Proposition \ref{prop:tempered_tube_open}. Since by Proposition \ref{prop tempered is independent of the presentation} $\temptube{X_i\cap X_j}{{\hat P}_i\cap {\hat 
P}_j}$ is independent of the presentation (up to equivalence), we have 
$$\temptube{X_i\cap X_j}{{\hat P}_i\cap {\hat P}_j} \cong \ \temptube{X_i\cap X_j}{{\hat P}_i} \cong \ \temptube{X_i\cap X_j}{{\hat P}_j}.$$
This data defines a derived analytic space in the sense of Definition \ref{defn:derived_analytic_space}. 

We can now apply the results of Subsection \ref{cotangent}. Since $\widehat P$ is smooth over $\mathcal V$, its analytic cotangent complex $\widehat{\mathbb L}_{{\widehat P^{\mathrm{an}}_K}/K} \in \sf{QCoh}_{\widehat P^{\mathrm{an}}}$ is locally concentrated in degree zero. This is a consequence of Theorem \ref{co-fibrant} and the fact that every formally smooth algebra over $\mathcal V$ is the completion of a smooth finitely generated $\mathcal V$-algebra by \cite[Theorem 7 in page 582]{elkik1973solutions}, and is it identified with the sheaf of bounded analytic differentials
\[ \widehat{\mathbb L}_{{\widehat P^{\mathrm{an}}_K}/K} \simeq \widehat\Omega^1_{{\widehat P^{\mathrm{an}}_K}/K}. \]
Moreover, by Theorem \ref{derived de rham} we can locally on an affine cover (and hence globally) identify analytic derived de Rham complex with the bounded analytic de Rham complex
\[ \mathbb L\widehat\Omega^{+ \bullet}_{{\widehat P^{\mathrm{an}}_K}/K} \cong  \widehat\Omega^{\bullet}_{{\widehat P^{\mathrm{an}}_K}/K}. \] 

\bigskip

\begin{dfn} \label{defn:tempered_cohomology}
    The \emph{$n$th-tempered convergent rigid cohomology group of $X$} is the $K$-vector space  defined as 
    $$\mathrm{H}^n_{\mathrm{temp}}(X,{\hat P})=
    \mathrm{H}^n( \mathbb{R} \Gamma_{\temptube{X}{{\hat P}}} ( \mathbb{L} \iota^*( {{\mathbb L}}{\hat \Omega}^{\bullet}_{{\hat P}_K^{\mathrm{an}}}))).
    $$
    where $\iota: \temptube{X}{{\hat P}} \hookrightarrow {\hat P}_K^{\mathrm{an}}$ is the open inclusion and $\mathbb{R} \Gamma_{\temptube{X}{{\hat P}}}: \sf{Shv}_{\temptube{X}{{\hat P}}}(\dercat(\sf{Mod}_K)) \to \dercat(\sf{Mod}_K)$ is the global sections functor.
\end{dfn}

\smallskip
\begin{rmk}
 In Definition \ref{defn:tempered_cohomology} we use the left $t$-structure of $\dercat(\sf{Mod}_K)$ to define cohomology groups. In the case when these cohomological groups are finite dimensional, they canonically agree with the classical cohomology of the underlying complex of $K$-vector spaces.
\end{rmk}

\smallskip

We keep the notation $\iota: \temptube{X}{{\hat P}} \hookrightarrow {\hat P}_K^{\mathrm{an}}$ for the open inclusion of the tempered tube, introduced in Definition \ref{defn:tempered_cohomology}.

\smallskip

\begin{lem} \label{complexdR} We can cover  ${\hat P}$ via the affine open formal subschemes ${\hat P}_i= \mathrm{Spf} A_i$, and we take on it the tempered tube of $X_i= U_i \cap X$ in it, $\temptube{X_i}{{\hat P}_i}$. Then, denoting the sections of the structure sheaf by 
$$ {\cal O_{\hat{P}}}(\temptube{X_i}{{\hat P}_i}) \cong \frac{\temp{A_{i, K}}{y_1,\dots,y_s}}{(y_1-\tilde{f}_1,\dots,y_s-\tilde{f}_s)} $$
(note that it is concentrated in degree $0$), we have 
$$
{{\mathbb L}}{\hat \Omega}^{ \bullet}_{\temptube{X_i}{{\hat P}_i}} \cong \mathbb{L} \iota_i^*({{\mathbb L}}{\hat \Omega}^{\bullet}_{{\hat P}_K^{an}}) \cong 
\hat{\Omega}^{\bullet}_{A_{i, K}} \wotimes_{A_{i, K}} {\cal O_{\hat{P}}}(\temptube{X_i}{{\hat P}_i})
$$
where $\iota_i: \temptube{X_i}{{\hat P}_i} \to {\hat P}_{i, K}^{an}$ denotes the opens embeddings of derived analytic spaces.
\end{lem}
\begin{proof} 
The first equivalence is a formal consequence of the local base change property of the cotangent complex \cite[Proposition 1.2.1.6]{Vez}. For the second equivalence, we notice that the formal scheme ${\hat P}$ is smooth, hence locally the cotangent complex is isomorphic to  $\Omega_{A_{i}}ˆ1$, which is locally free, and its analytification $\hat{\Omega}_{A_{i, K}}ˆ1$ as well.
However, the sections of the structural sheaf on the open given by the tempered tube are concentrated in degree $0$ as well. Hence, by Theorem \ref{derived de rham}  and Corollary \ref{tensor2} 
in  
$\dercat_{\infty}( \sf{Mod}_{K})$ applied to the homotopical epimorphism $A_{i, K} \to \frac{\temp{A_{i, K}}{y_1,\dots,y_s}}{(y_1-\tilde{f}_1,\dots,y_s-\tilde{f}_s)}$, opposite to the embedding $\iota_i$, we have 
$$
\mathbb{L} \iota_i({{\mathbb L}}{\hat \Omega}^{\bullet}_{{\hat P}_K^{an}}) \cong
\hat{\Omega}^{\bullet}_{A_{i,K}} \wotimes_{A_{i, K}} {\cal O_{\hat{P}}}(\temptube{X_i}{{\hat P}_i}).
$$  
\end{proof} 

\medskip

The next theorem is one of our main results.

\medskip

\begin{thm}\label{thm tempered cohomology is weel def}
Let $X$ be a  variety over $k$, equipped with a regular closed immersion in a formally smooth $p$-adic formal scheme $\hat{P}$. Then the cohomology groups $\mathrm{H}^n_{\mathrm{temp}}(X,{\hat P})$ are independent of the regular  embedding $X\hookrightarrow {\hat P}$. So, we may be indicated them as $\mathrm{H}^n_{\mathrm{temp}}(X)$.
     
\end{thm}
\begin{proof}
    Let $X\hookrightarrow {\hat P}$ and $X\hookrightarrow {\hat P}'$ be two regular embeddings into formally smooth $p$-adic formal schemes. We can always reduce to the situation where we have a smooth map
    \begin{equation}  \label{indepen}
     \begin{tikzcd}
  & {\hat P}' \arrow{dd}{u}\\
X \arrow{ur} \arrow[swap]{dr}&\\
  & {\hat P}, 
\end{tikzcd}   
    \end{equation}
    by taking fiber products. By \cite[Lemma 2.3.14]{LeStum},
     there exists an affine cover of ${\hat P}$, $\{{\hat P}_i\}_{i\in I}$ and a cover of ${\hat P}'$, $\{{\hat P}'_i\}_{i\in I}$ such that $u({\hat P}'_i)\subset {\hat P}_i$ with the property $X \cap {\hat P}_i= X \cap {\hat P}'_i$.  
    Therefore problem is local for the Zariski topology. Thus, it is enough to check that the morphism $u$ induces locally an equivalence of sheaves. Hence it is enough to prove the theorem in the local case of affine schemes. Hence we can suppose that we are in the situation of \eqref{indepen} and all objects are affine i.e. $X\cap {\hat P}_i=X_i= X \cap {\hat P}'_i$. Applying Theorem \ref{weakfibr}, this means that we can write 
    \[ \hat{P}_i \cong \mathrm{Spf}(A_i). \ \ \temptube{X_i}{{\hat P}_i} \cong \goth S \left (\frac{\temp{A_{i, K}}{y_1, \ldots, y_s}}{(y_1-\tilde{f}_1,\dots,y_s-\tilde{f}_s)} \right ) \]
    and $\hat{P}_i' \cong \mathrm{Spf}(\tate{A_i}{x_1, \ldots, x_d})$, where the map $u$ is just induced by the inclusion $A_i \to \tate{A_i}{x_1, \ldots, x_d}$.
    Because we are in the affine case, we can suppose that the cotangent complex sheaf is free, hence the same for all terms of the associated de Rham complex. 
    We have to compare the two complexes
    $\mathbb{L} \hat{\Omega}^{\bullet}_{ \temptube{X_i}{{\hat P}}}$ and $u_*(\mathbb{L} \hat{\Omega}^{\bullet}_{ \temptube{X_i}{{\hat P}'}})$. Both $\mathbb{L} \hat{\Omega}^{\bullet}_{ \temptube{X_i}{{\hat P}}}$ and $\mathbb{L} \hat{\Omega}^{\bullet}_{ \temptube{X_i}{{\hat P}'}}$ are computed as in Lemma \ref{complexdR} and $u_*$ is just restriction of scalars (that is exact for affine morphisms). 
    Hence, the de Rham complex  ${{\mathbb L}}{\hat \Omega}^{\bullet}_{\temptube{X_i}{{\hat P}_i}}$  
    is given by the complex of $K$-bornological spaces    
    $$
   \dots \rightarrow  \frac{\temp{A_{i, K}}{y_1,\dots,y_s}}{(y_1-\tilde{f}_1,\dots,y_s-\tilde{f}_s) } \wotimes_{A_{i, K}} \Omega^{n}_{A_{i,K}/K} \rightarrow \frac{\temp{A_{i, K}}{y_1,\dots,y_s}}{(y_1-\tilde{f}_1,\dots,y_s-\tilde{f}_s)}\wotimes_{A_{i, K}} \Omega^{n+1}_{A_K/K} \rightarrow  \dots
    $$
    While, again applying Theorem \ref{weakfibr}, the de Rham complex  ${{\mathbb L}}{\hat \Omega}^{\bullet}_{\temptube{X_i}{{\hat P'}_i}}$  
    is given by the complex of $K$-bornological spaces 
 $$
  \dots \rightarrow   \frac{\temp{A_{i, K}}{y_1,\dots,y_s, , x_1, \dots , x_d}}{(y_1-\tilde{f}_1,\dots,y_s-\tilde{f}_s)} \wotimes_{\tate{A_{i, K}}{x_1,\dots,x_d}} \Omega^{n}_{\tate{A_{i, K}}{x_1,\dots,x_d}/K} \rightarrow $$ 
  $$ \to \frac{\temp{A_{i, K}}{y_1,\dots,y_s, , x_1, \dots , x_d}}{(y_1-\tilde{f}_1,\dots,y_s-\tilde{f}_s)} \wotimes_{\tate{A_{i, K}}{x_1,\dots,x_d}} \Omega^{n+1}_{\tate{A_{i, K}}{x_1,\dots,x_d}/K} \rightarrow  \dots .
    $$
To prove that the map induced by $u$ is a  quasi-isomorphism, we can reduce by induction to the following tempered version of Poincar\'e lemma in one variable. This proves that the sheaves $\mathbb{L} \hat{\Omega}^{\bullet}_{\temptube{X}{\hat{P}}}, u_*(\mathbb{L} \hat{\Omega}^{ \bullet}_{\temptube{X}{\hat{P}}})$ are equivalent, hence they have equivalent global sections and associated cohomology groups.
\end{proof}

\medskip

\begin{lem}[Tempered Poincaré lemma]
    The complex \begin{equation*}0\rightarrow A\hookrightarrow\temp{A}{x}\xrightarrow{\diff{x}}\temp{A}{x}\rightarrow 0 \end{equation*} is strictly exact, where $A$ is any bornological $K$-algebra.
\end{lem}
\begin{proof}
Define the evaluation map 
\[ \operatorname{ev}_0\colon\temp{A}{x}\longrightarrow A, \qquad \sum_{n\geq0}a_nx^n\mapsto a_0, \] 
and the formal integration operator 
\[ h\colon\temp{A}{x}\,dx\longrightarrow\temp{A}{x}, \qquad \sum_{n\geq0}a_nx^n\,dx \mapsto \sum_{n\geq0}\frac{a_n}{n+1}x^{n+1}. \] 
We check that $h$ is bounded. Since $K$ has characteristic zero and carries a non-archimedean valuation, there exist constants $C>0$ and $c\geq0$ such that \[ |n|^{-1}\leq Cn^c \] for every integer $n\geq1$. Thus division by $n+1$ increases the growth of the coefficients by at most a polynomial factor. Consequently, for every Banach level $\power{A}{x}_r$ of $\temp{A}{x}$, the operator $h$ induces a bounded map \[ \power{A}{x}_r\,dx \longrightarrow \power{A}{x}_{r+c}. \]
It therefore defines a bounded morphism on the filtered colimit $\temp{A}{x}$. Hence the complex is split exact in the bornological sense, and therefore strictly exact.
\end{proof}

We conclude with a comparison theorem relating our tempered cohomology to classical $p$-adic cohomology theories. In particular, this will imply the finite dimensionality of tempered cohomology for smooth proper varieties.

\begin{pro}\label{prop comparison with crystalline}
    Let $X$ be a smooth proper $k$-scheme admitting a regular closed immersion into a smooth $p$-adic formal $\mathcal V$-scheme of the type considered above. Then there are natural isomorphisms 
    $$
    \mathrm{H}^i_{\mathrm{temp}}(X_k) \cong \mathrm{H}^i_{\mathrm{crys}}(X_k) \otimes_{\mathcal{V}} K \cong \mathrm{H}^i_{\mathrm{rig}}(X_k).
    $$

    If moreover $X_k$  is the special fiber of a smooth and proper $\mathcal V$-scheme $X_{\mathcal V}$, then 
    $$
    \mathrm{H}^i_{\mathrm{temp}}(X_k) \cong  \mathrm{H}^i_{dR}(X_K).
    $$
    
\end{pro}
\begin{proof}
 By Berthelot's comparison  theorem \cite{Berth}, we know that crystalline and rigid cohomologies of $X_k$ coincide. The same method of proof  will work to get an isomorphism with tempered cohomology.
 In fact, we can argue locally for the Zariski topology for all the aforementioned cohomology theories. As in Theorem \ref{thm tempered cohomology is weel def} can reduce to an affine covering $\{ X_i\hookrightarrow {\hat P}_i \} $ and compare the complexes that compute the cohomology groups on these affine open subspaces. Then, we all such complexes for the crystalline, (convergent) rigid and  tempered cohomologies, are independent upon the smooth closed embedding. But because $X_i$ is smooth affine a smooth lifting   can  be built and it will give  a closed smooth embedding. In this case the various tubes used to define the cohomology theories (the divided powers tube, the rigid one and tempered tube) are all defined by reduction on the maximal ideal. Hence, they all coincide, and this gives the first claim of  statement.
 The comparison between  the crystalline, hence for the tempered,  and de Rham cohomology of the generic fiber is classical from \cite{BO} (if the $k$-scheme is proper and smooth). 
\end{proof}

\bibliographystyle{plain}
\bibliography{bib}

\end{document}